\documentclass[10pt,reqno]{amsart}
\usepackage{url}
\usepackage{amsthm}
\usepackage{fullpage}
\usepackage{amsmath}
\usepackage{amssymb}
\usepackage{amsfonts}
\DeclareMathOperator{\spn}{span}
\usepackage{enumitem}
\usepackage{mathscinet}
\usepackage{lipsum}
\usepackage[english]{babel}
\usepackage[autostyle]{csquotes}
\usepackage{comment}
\usepackage{footmisc}

\usepackage{graphicx}

\newtheorem{thm}{Theorem}[section]
\newtheorem{definition}[thm]{Definition}
\newtheorem{theorem}[thm]{Theorem}
\newtheorem*{oseledec*}{Oseledec Theorem}
\newtheorem{cor}[thm]{Corollary}
\newtheorem{claim}[thm]{Claim}
\newtheorem{lemma}[thm]{Lemma}
\newtheorem{prop}[thm]{Proposition}

\makeatletter
\def\moverlay{\mathpalette\mov@rlay}
\def\mov@rlay#1#2{\leavevmode\vtop{%
   \baselineskip\z@skip \lineskiplimit-\maxdimen
   \ialign{\hfil$\m@th#1##$\hfil\cr#2\crcr}}}
\newcommand{\charfusion}[3][\mathord]{
    #1{\ifx#1\mathop\vphantom{#2}\fi
        \mathpalette\mov@rlay{#2\cr#3}
      }
    \ifx#1\mathop\expandafter\displaylimits\fi}
\makeatother

\newcommand{\bigcupdot}{\charfusion[\mathop]{\bigcup}{\cdot}}

\begin{document}

\title{SYMBOLIC DYNAMICS FOR NON-UNIFORMLY HYPERBOLIC DIFFEOMORPHISMS OF COMPACT SMOOTH MANIFOLDS}
\author{Snir Ben Ovadia}
\address{Faculty of Mathematics and Computer Science, Weizmann Institute of Science, POB 26, Rehovot, 76100 ISRAEL} \email{snir.benovadia@weizmann.ac.il}

\date{}
\thanks{This work is a part of a M.Sc thesis at the Weizmann Institute of Science. The author was partly supported by the ISF grant 199/14.} 

\maketitle

\begin{abstract}
We construct countable Markov partitions for non-uniformly hyperbolic diffeomorphisms on compact manifolds of any dimension, extending earlier work of Sarig \cite{Sarig} for surfaces.

These partitions allow us to obtain symbolic coding on invariant sets of full measure for all hyperbolic  measures whose Lyapunov exponents are bounded away from zero by a fixed constant. Applications include counting results for hyperbolic periodic orbits, and structure of hyperbolic measures of maximal entropy.

   
   
   
\end{abstract}

\tableofcontents
\setcounter{section}{-1}
\section{Introduction}
In this paper we construct countable Markov partitions for non-uniformly hyperbolic $C^{1+\beta}$ diffeomorphisms on compact manifolds of any dimension. The two dimensional case was done in \cite{Sarig}. The uniformly hyperbolic case (in any dimension) was done in \cite{Si1,Si2} and in \cite{B1}. Markov partitions were first constructed for hyperbolic toral automorphisms in dimension two, in \cite{Aw,AW70}.
\subsection{Main results}
Let $\mathcal{G}$ be a directed graph with a countable collection of vertices $\mathcal{V}$ s.t. every vertex has at least one ingoing and one outgoing edge. The {\em topological Markov shift} (TMS) associated to $\mathcal{G}$ is the set
$$\Sigma=\Sigma(\mathcal{G}):=\{(v_i)_{i\in\mathbb{Z}} \in\mathcal{V}^\mathbb{Z} :v_i \rightarrow v_{i+1}\text{  }, \forall i\in\mathbb{Z}\},$$
equipped with the left-shift $\sigma:\Sigma\rightarrow\Sigma$, $\sigma((v_i)_{i\in\mathbb{Z}})=(v_{i+1})_{i\in\mathbb{Z}}$, and the metric 
$d(u,v):=\exp(-\min\{n\in\mathbb{N}_0:u_n\neq v_n\text{ or }u_{-n}\neq v_{-n}\})$ for $u=(u_n)_{n\in\mathbb{Z}},$ $v=(v_n)_{n\in\mathbb{Z}}$. With this metric $\Sigma$ is a complete separable metric space. $\Sigma$ is compact iff $\mathcal{G}$ is finite. $\Sigma$ is locally compact iff every vertex of $\mathcal{G}$ has finite ingoing and outgoing degree.

A {\em Markovian subshift} of $\Sigma(\mathcal{G})$ is a subset of the form $\Sigma(\mathcal{G}')$, where $\mathcal{G}'$ is a subgraph of $\mathcal{G}$.

\medskip
The following definition is due to Sarig: given a TMS $\Sigma$,
$$\Sigma^\#:=\{(v_i)_{i\in\mathbb{Z}}\in\Sigma:\exists v',w'\in\mathcal{V}(\Sigma)\text{ , }\exists n_k,m_k\uparrow\infty\text{ s.t. }v_{n_k}=v'\text{ , }v_{-m_k}=w',\text{ }\forall k\in\mathbb{Z}\}.$$
Notice that by the Poincar\'e recurrence theorem every $\sigma$-invariant probability measure is carried by $\Sigma^\#$. Furthermore, every periodic point of $\sigma$ is in $\Sigma^\#$. 


Let $f$ be a $C^{1+\beta}$ diffeomorphism on a compact smooth boundaryless manifold $M$ of dimension greater than $1$. Here $\beta\in (0,1)$ is the H\"older exponent of $df$.

We say that an ergodic $f$-invariant probability measure $\mu$ is a \textit{hyperbolic measure}, if it has no zero Lyapunov exponents, and among them at least one is positive and one is negative. Let $\chi(\mu):=\min\{|\chi_i|:\chi_i\text{ is a Lyapunov exponent of }\mu\}$. If $0<\chi<\chi(\mu)$, we say $\mu$ is {\em $\chi$-hyperbolic}. Similarly, we call a point in the manifold {\em $\chi$-hyperbolic} if it is Lyapunov regular (see definition \ref{NUH}) and its Lyapunov exponents are bounded away from $0$ by $\chi$. Our main results are:
\begin{theorem}\label{t4.1.1} For every $\chi>0$ there exist a locally compact TMS $\Sigma_\chi$, and a H\"older continuous map $\pi_\chi:\Sigma_\chi\rightarrow M$ s.t.: \begin{enumerate}
    \item[{\rm (1)}] $\pi_\chi\circ\sigma=f\circ\pi_\chi$.
    \item[{\rm (2)}] $\pi_\chi[\Sigma_\chi^\#]$ has full measure for all $\chi$-hyperbolic measures, and every point in $\pi_\chi[\Sigma_\chi^\#]$ has finitely many pre-images in $\Sigma_\chi^\#$.
\end{enumerate}
\end{theorem}
\medskip
\begin{theorem}\label{t4.1.3} For
every $\chi$-hyperbolic measure $\mu$ on $M$ there is an ergodic $\sigma$-invariant probability measure $\widehat{\mu}$ on $\Sigma_\chi$, such that $\mu$ equals $\widehat{\mu}\circ\pi_\chi^{-1}$ and the metric entropies of these measures coincide.
\end{theorem}

Notice that every shift-invariant probability measure on $\Sigma_\chi$ projects to an invariant probability measure on the manifold. The following theorem implies that the projection preserves the entropy.
\begin{theorem}\label{t4.1.2} For $\Sigma_\chi$ given by theorem \ref{t4.1.1}, we denote the set of states of $\Sigma_\chi$ by $\mathcal{V}_\chi$. There exists a function $N_\chi:\mathcal{V}_\chi\rightarrow\mathbb{N}$ s.t. for every $x\in M$ which can be written as $x=\pi_\chi((v_i)_{i\in\mathbb{Z}})$ with $v_i=u$ for infinitely many $i<0$, and $v_i=w$ for infinitely many $i
>0$, it holds that: $|\pi_\chi^{-1}[\{x\}]\cap \Sigma_\chi^\#|\leq N_\chi(u)\cdot N_\chi(w)$.
\end{theorem}
\subsection{Applications}
\begin{theorem}\label{t4.2.1} Suppose $f$ is a $C^{1+\beta}$ diffeomorphism of a compact smooth boundary-less manifold  of dimension greater than $1$, and suppose $f$ has a hyperbolic measure of maximal entropy. Then $$\exists p\in\mathbb{N}\text{ s.t. } \liminf_{n\rightarrow\infty,p|n}e^{-nh_{top}(f)}P_n(f)>0,$$
where $P_n(f):=\#\{x\in M: f^n(x)=x\}$ and $h_{top}(f)$ is the topological entropy of $f$.
\end{theorem}

\begin{theorem}\label{t4.2.2} Suppose $f$ is a $C^{1+\beta}$ diffeomorphism of a compact smooth boundaryless manifold of dimension greater than $1$. Then $f$ possesses at most countably many ergodic hyperbolic measures of maximal entropy.
\end{theorem}
Theorems \ref{t4.2.1} and \ref{t4.2.2} follows from our main results as in \cite{Sarig}, see also \cite{K3,Bu4}. Unlike the case of surfaces, in dimension 3 and onwards measures of positive maximal entropy are not necessarily hyperbolic (consider e.g. the product of a hyperbolic toral automorphism and an irrational rotation).
\subsection{Comparison to other results in the literature}\

Markov partitions for diffeomorphisms have different definitions (\cite{Aw,AW70,Si1,B1,Sarig}). They were constructed for several classes of systems: hyperbolic toral automorphisms \cite{Aw}, Anosov diffeomorphisms \cite{Si1,Si2}, pseudo--Anosov diffeomorphisms \cite{FS}, Axiom A diffeomorphisms \cite{B1,B2}, and general diffeomorphisms of surfaces \cite{Sarig}.
For flows,  see \cite{Ra69,Ra73,B73,SL14}. For maps with singularities, see \cite{BS80,LM16}. 

\medskip
This work treats the case of manifolds of general dimension, for general $C^{1+\beta}$ diffeomorphisms. The main tool in this work is Pesin's theory of non-uniform hyperbolicty \cite{Pesin, Pesin2}.

\medskip
Katok showed that a $C^{1+\beta}$ diffeomorphism $f$ with a hyperbolic measure $\mu$ has for each $\epsilon>0$ a compact invariant subset $\Gamma_\epsilon$ (which is called a {\em Katok horseshoe}) s.t. $f:\Gamma_\epsilon\to\Gamma_\epsilon$ has a finite Markov partition, and $h_{top}(f|_{\Gamma_\epsilon})>h_\mu(f)-\epsilon$  \cite{K1,K2,KM}. Katok's horseshoes have \textbf{finite} Markov partitions. Our ``horseshoes" have countable Markov partitions. This is unavoidable since maps with finite Markov partition can only have countably many possible values for the topological entropy.

Typically, $\Gamma_\epsilon$ will have zero measure w.r.t. $\mu$. This paper constructs for every $\chi>0$ a non-compact Katok horseshoe $\pi_\chi[\Sigma_\chi]$ with full measure for all $\chi$-hyperbolic measures (in the case of surfaces, positive entropy means hyperbolicity).

\subsection{Overview of the construction of the Markov partition}\label{SectionOverview}
Bowen's construction of Markov partitions for uniformly hyperbolic diffeomorphisms \cite{B3,B4} uses Anosov's shadowing theory for pseudo-orbits. This theory fails for general non-uniformly hyperbolic diffeomorphisms. In dimension two, Sarig developed a refined shadowing theory which does work in the non-uniformly hyperbolic setup. It consists of: \begin{enumerate}
    \item Definition of \underline{``$\epsilon$-chains"}: refinement of the definition of pseudo orbit.
    \item \underline{Shadowing lemma}: every $\epsilon$-chain ``shadows" a unique orbit.
    \item \underline{Solution for the inverse problem}: control of $\epsilon$-chains which shadow the same orbit.
\end{enumerate}
These properties, once established, allow the construction of Markov partitions using the method of Bowen  (\cite{Sarig}). Sarig worked in dimension two. We generalize his work to the higher dimensional case. 

The main difficulty in the generalization is the solution of the ``inverse problem" for the shadowing procedure (section \ref{chapter333}). The argument in \cite{Sarig} uses very strongly
\begin{itemize}
    \item the one-dimensionality of local stable and unstable manifolds
    \item the existence of canonical (up to a sign) bases for the Oseledec spaces $H^s(x), H^u(x)$ ($E^s(x), E^u(x)$ in Sarig's notations)
\end{itemize}
In the higher dimensional case these facts
are not valid. To deal with this we need to modify the definition of Pesin charts and chains as used by Sarig, re-do his graph transform estimations, and most importantly find alternative proofs for many of the comparison estimations needed to solve the ``inverse problem".

\medskip

Some of our arguments require major changes from their two dimensional analogues. When this is the case we give full details. In other cases (e.g. construction of Markov partitions given (1) , (2) and (3) above) we omit the technical details, and limit ourselves to a sketch of the main idea and references to sources in which complete details can be found.
\subsection{Notation}\label{notations}
\begin{enumerate}
    \item For any normed vector space $V$, we denote the unit sphere in $V$ by $V(1)$.
    \item Let $L$ be a linear transformation between two inner-product spaces $V,W$ of finite dimension. The {\em Frobenius norm} of $L$ is defined by $\|L\|_{Fr}=\sqrt{\sum_{i,j}a_{ij}^2}$, where $(a_{ij})$ is the representing matrix for $L$, under the choice of some pair of orthonormal bases for $V$ and $W$. This definition does not depend on the choices of bases.
    \item When a statement about two different cases that are indexed by their subscripts and/or their superscripts holds for both cases, we write in short the two cases together, with the two scripts separated by a ``/". For example, $\pi_\chi(u_{x/y})=x/y$ means $\pi_\chi(u_x)=x$ and $\pi_\chi(u_y)=y$.
    \item Given $A,B,C\in\mathbb{R}^+$ we write $A=e^{\pm B}C$ when $e^{-B}C\leq A\leq e^B C$.
    \item $|\cdot|$ denotes the Riemannian norm on $T_xM$ ($x\in M$) when applied to a tangent vector in $T_xM$, and denotes the Euclidean norm when applied to a vector in $\mathbb{R}^d$.
    \item $|\cdot|_\infty$ denotes the highest absolute value of the coordinates of a vector.
\end{enumerate}

\subsection{Acknowledgements}
I would like to thank my M.Sc. advisor, Omri Sarig, for the patient guidance, encouragement and advice he has provided throughout my time as his student, and the Weizmann Institute for excellent working conditions. I would like to thank the referee for his careful reading of the manuscript, and for his many useful suggestions. Finally, special thanks to J\'er\^ome Buzzi, Sylvain Crovisier, Yuri Lima and Yakov Pesin for useful discussions. 
\section{Local linearization- chains of charts as pseudo-orbits}
Let $f$ be a $C^{1+\beta}$ diffeomorphism of a smooth, compact and boundary-less manifold $M$ of dimension $d\geq 2$.

In this section we adapt the definitions of Pesin charts, chart overlap, and $\epsilon$-chains from \cite{Sarig} to the higher dimensional case. While our Pesin charts are similar to those used by Pesin \cite{BP} and Katok \cite{KM,K1}, there are also differences. These are needed to get the fine graph transform estimates of the next section.

\subsection{Pesin charts}
\subsubsection{Non-uniform hyperbolicity}
Set $\log^+t:=\max\{\log t,0\}$. Every invertible cocycle $A:X\times \mathbb Z\rightarrow GL_d(\mathbb R)$ is uniquely associated with a {\em generator} $A_1:X\rightarrow GL_d(\mathbb{R})$ by $A(x,0)=Id, A(x,n)=A_1(f^{n-1}(x))\cdot...\cdot A_1(x)$ for $n>0$ and $A(x,-n)=A(f^{-n}(x),n)^{-1}$.
\begin{oseledec*}
Let $(X,\mathcal{B},\mu,T)$ be an invertible probability preserving transformation. Suppose $A:X\times \mathbb Z\rightarrow GL_d(\mathbb R)$ is an invertible cocycle. 
Consider its generator $A_1:X\rightarrow GL_d(\mathbb{R})$
. If $\log^+\|A_1\|,\log^+\|A_1^{-1}\|$ are integrable, then for $\mu$-almost every $x\in X$ the following set has full measure: $$\{x\in X| \exists \lim_{n\rightarrow \pm\infty} {{\Big(A(x,n)}^t A(x,n)\Big)}^{\frac{1}{2n}}\text{ and the two limits are equal}\}.$$
\end{oseledec*}
\begin{definition}\label{NUH}
    (Lyapunov regularity, Lyapunov exponents and Lyapunov decomposition)   
    \begin{enumerate}
    \item Define the set of {\em Lyapunov regular points} for the invertible cocycle $df$:
$$LR:=\{x\in M| \exists \lim_{n\rightarrow \pm\infty} {{((d_xf^n)}^t d_xf^n)}^{\frac{1}{2n}}\text{ and the two limits are equal}\}.$$
By the Oseledec theorem, $LR$ has full measure w.r.t any $f$-invariant measure.
\item For any Lyapunov regular $x\in M$, the limit operator $\lim\limits_{n\rightarrow \pm\infty}{{((d_xf^n)}^t d_xf^n)}^{\frac{1}{2n}}$ is positive-definite, and thus diagonalizable with positive eigenvalues. Call its eigen-spaces $H_i(x)$, with respective eigenvalues $e^{\chi_i(x)}$, where $i$ ranges from smallest $\chi_i(x)$ to largest. $\chi_i(x)$ is called {\em the $i$-th Lyapunov exponent at $x$}.
\item We thus get the {\em Lyapunov decomposition} $T_xM=\oplus_{i=1}^{k(x)}H_i(x)$. Denote $\dim(H_i(x))=l_i(x)$. It follows that $k(f(x))=k(x), l_i(f(x))=l_i(x), d_xf[H_i(x)]=H_i(f(x))$, $\forall i=1,...,k(x)$.
\end{enumerate}
\end{definition}
 Notice that these are $f$-invariant functions, and thus are const. a.e. w.r.t to ergodic measures.
\begin{definition}\label{NUH2}
($\chi$-hyperbolicity, non-uniformly hyperbolic set)
\begin{enumerate}
\item For any Lyapunov regular $x\in M$, define $s(x):=\sum\limits_{i:\chi_i(x)<0}l_i(x)$, $u(x):=d-s(x)$.
\item For any $\chi>0$, we say a point $x\in LR$ is {\em $\chi$-hyperbolic} if $\min\{|\chi_i(x)|\}>\chi$ and $s(x)\in\{1,...,d-1\}$. An $f$- invariant and ergodic measure is called {\em $\chi$-hyperbolic} if $\mu$-almost every $x\in M$ is $\chi$-hyperbolic.
    \item The non-uniformly hyperbolic set of points is defined for every $\chi>0$ as follows:
$$ NUH_\chi :=\{x\in LR|\min\{|\chi_i(x)|\}>\chi,s(x)\in\{1,...,d-1\}\}.$$

\end{enumerate}
\end{definition} 
We will restrict ourselves for cases where $s(x)\in\{1,...,d-1\}$- the hyperbolic case.
Katok proves in \cite[theorem~4.2]{K1} that if an ergodic measure has no zero Lyapunov exponents, and is not supported on a single periodic orbit, then it has positive and negative Lyapunov exponents.
Notice that an ergodic probability measure is $\chi$-hyperbolic iff it gives $NUH_\chi$ full measure.
\subsubsection{Lyapunov change of coordinates}
\begin{definition} 

$\text{ }$\\\begin{itemize}
\item For a point $x\in NUH_\chi$: $H^s(x):=\oplus_{i:\chi_i<0}H_i(x)$, $H^u(x):=\oplus_{i:\chi_i>0}H_i(x)$.
\item For two linear vector spaces $V$ and $W$, $GL(V,W)$ is the space of invertible linear transformations between $V$ and $W$.
\end{itemize}
\end{definition}

\begin{theorem}\label{pesinreduction} Oseledec-Pesin $\epsilon$-reduction  theorem (\cite{Pesin},\cite[\textsection~S.2.10]{KM}):

Let $M$  be  a compact Riemannian manifold. Then for each $\chi$-hyperbolic point $x\in M$ there exists an invertible linear transformation $C_\chi(x):\mathbb{R}^d\rightarrow T_xM$ ($C_\chi(\cdot):NUH_\chi\rightarrow GL({\mathbb{R}}^d,T_\cdot M)$) such that the map $D_\chi(x)=C_\chi^{-1}(f(x)) \circ d_xf \circ C_\chi(x)$ has the Lyapunov block form:
$$
\begin{pmatrix}D_s(x)  &   \\  & D_u(x)
\end{pmatrix},
$$
where $D_{s}(x)\in GL(s(x),\mathbb{R}),D_{u}(x)\in GL(u(x),\mathbb{R})$. In addition, for every $x \in  NUH_\chi$ we can decompose $T_xM=H^s(x)\oplus H^u(x)$ and $\mathbb{R}^d=\mathbb{R}^{s(x)}\oplus\mathbb{R}^{u(x)}$, and $C_\chi(x)$ sends each $\mathbb{R}^{s(x)/u(x)}$ to $H^{s/u}(x)$. Furthermore, $\exists \kappa(\chi,f)$ s.t.  $$\kappa^{-1}\leq\frac{|D_sv_s|}{|v_s|}\leq e^{-\chi},\kappa\geq\frac{|D_uv_u|}{|v_u|}\geq e^{\chi},\text{for all  non-zero }v_s\in\mathbb{R}^{s(x)},v_u\in \mathbb{R}^{u(x)}.$$ 
\end{theorem}
\begin{proof} Denote the Riemannian inner product on $T_xM$ by $\langle\cdot,\cdot\rangle$. Set
$$\langle u,v \rangle_{x,s}':=2\sum_{m=0}^\infty\langle d_xf^mu,d_xf^mv\rangle e^{2m\chi},\text{ for }u,v\in H^s(x);$$
$$\langle u,v \rangle_{x,u}':=2\sum_{m=0}^\infty\langle d_xf^{-m}u,d_xf^{-m}v\rangle e^{2m\chi},\text{ for }u,v\in H^u(x).$$
Since $x\in NUH_\chi$, the series converge absolutely. Thus, the series can be rearranged, and $\langle \cdot,\cdot \rangle_{x,s}',\langle \cdot,\cdot \rangle_{x,u}'$ are bilinear and are well defined as inner products on $H^s(x),H^u(x)$, respectively.
Define:
\begin{equation}\label{LIP}
\langle u,v \rangle_x':=\langle \pi_s u,\pi_s v \rangle_{x,s}'+\langle \pi_u u,\pi_u v \rangle_{x,u}',
\end{equation}
where $u,v\in T_xM$, and $v=\pi_sv+\pi_uv$ is the unique decomposition s.t. $\pi_{s}v\in H^{s}(x),\pi_{u}v\in H^{u}(x)$ (the tangent vector $u$ decomposes similarly).

\medskip
Choose measurably\footnote{Notice that $C_\chi^{-1}(x)$ are only determined up to an orthogonal self map of $H^{s/u}(x)$- thus a measurable choice is needed. Let $\{D_i\}_{i=1}^N$ be a finite cover of open discs for the compact manifold $M$. Each open disc is orientable, and there exists a continuous choice of orthonormal
(w.r.t to the Riemannian metric) bases of $T_xM$, $(\widetilde{e}_1(x),...,\widetilde{e}_d(x))$. The contiuous choice of bases on the discs induces a meaurable choice of bases on the whole manifold by the partition $D_1, D_2\setminus D_1,D_3\setminus (D_1\cup D_2),...,D_N\setminus\cup_{i=1}^{N-1}D_i$. Using projections to $H^{s/u}(x)$ (which are defined measurably), and the Gram-Schmidt process, we construct orthonormal bases for $H^{s/u}(x)$ in a measurable way, w.r.t to $\langle\cdot,\cdot\rangle_{x,s/u}'$. Denote these bases by $\{b_i(x)\}_{i=1}^{s(x)}$ and $\{b_i(x)\}_{i=s(x)+1}^{d}$, respectively. Define $C_\chi^{-1}(x):b_i(x)\mapsto e_i$, where $\{e_i\}_{i=1}^d$ is the standard basis for $\mathbb{R}^d$.} $C_\chi^{-1}(x):T_xM\rightarrow\mathbb{R}^d$ to be a linear transformation satisfying 
\begin{equation}\label{omrisuggestnumber}
    \langle u,v \rangle_x'=\langle C_\chi^{-1}(x)u,C_\chi^{-1}(x)v \rangle,
\end{equation}
and $C_\chi^{-1}[H^s(x)]=\mathbb{R}^{s(x)}$ (and hence $C_\chi^{-1}[H^u(x)]=(\mathbb{R}^{s(x)})^\perp$). Thus $C_\chi^{-1}[H^s(x)]\perp C_\chi^{-1}[H^u(x)]$, and therefore $D_\chi$ is in this block form. Therefore, when $v_s\neq 0$, we get
$$\langle d_xfv_s,d_xfv_s\rangle_{f(x),s}'=2\sum_{m=0}^\infty |d_xf^{m+1}v_s|^2e^{2(m+1)\chi}e^{-2\chi}=e^{-2\chi}(\langle v_s,v_s\rangle_{x,s}'-2|v_s|^2)$$
$$\Rightarrow \frac{\langle d_xfv_s,d_xfv_s\rangle_{f(x),s}'}{\langle v_s,v_s\rangle_{x,s}'}= e^{-2\chi}(1-\frac{2|v_s|}{\langle v_s,v_s\rangle_{x,s}'}).$$
Call $M_f:=\max_{x\in M}\{\|d_xf\|,\|d_xf^{-1}\|\}$ and notice $\langle v_s,v_s\rangle_{x,s}'>2(|v_s|^2+|d_xfv_s|^2e^{2\chi})\geq2(|v_s|^2+M_f^{-2}|v_s|^2e^{2\chi})$. So:
$$\frac{\langle d_xfv_s,d_xfv_s\rangle_{f(x),s}'}{\langle v_s,v_s\rangle_{x,s}'}\in[\kappa^{-2},e^{-2\chi}]$$
where $\kappa=\max\{e^\chi(1-\frac{1}{1+M_f^{-2}e^{2\chi}})^{-1/2},e^\chi(1+M_f^2)^{1/2}\}$.

The same holds for $D_u$:
$$\langle d_xfv_u,d_xfv_u\rangle_{f(x),u}'=2\sum_{m=0}^\infty |d_xf^{-m+1}v_u|^2e^{2(m-1)\chi}e^{2\chi}=e^{2\chi}(\langle v_u,v_u\rangle_{x,u}'+2|d_xfv_u|^2e^{-2\chi}),$$
and since $\langle v_u,v_u\rangle_{x,u}'>2|v_u|^2$: $\frac{\langle d_xfv_u,d_xfv_u\rangle}{|v_u|_{x,u}'}<\frac{\langle d_xfv_u,d_xfv_u\rangle}{2|v_u|}$, we can conclude:
$$\kappa^2\geq\frac{\langle d_xfv_u,d_xfv_u\rangle_{f(x),u}'}{\langle v_u,v_u\rangle_{x,u}'}\geq e^{2\chi}.$$ Let $w_s\in \mathbb{R}^{s(x)}$, and define $v_s:=C_\chi(x)w_s$.
$$\frac{|D_s(x)w_s|^2}{|w_s|^2}=\frac{|C_\chi^{-1}(f(x))d_xfv_s|^2}{|C_\chi^{-1}(x)v_s|^2}=\frac{|d_xfv_s|_{f(x),s}^{'2}}{|v_s|_{x,s}^{'2}}\in[\kappa^{-2}, e^{-2\chi}].$$
Similarly $\frac{|D_u(x)w_u|^2}{|w_u|^2}\in[e^{2\chi},\kappa^2]$.
\end{proof}
\begin{definition}\label{chi_z}
$\text{ }$

\medskip
\begin{itemize}
    \item $\forall z\in NUH_\chi, \chi_z:=\min\{|\chi_i(z)|\}>\chi.$
    \item We define another Lyapunov change of coordinates, similarly to theorem \ref{pesinreduction}, which will be of use for lemma \ref{RegularOnVs}: 
    $$\langle \xi,\eta \rangle_{z,s}'':=2\sum_{m=0}^\infty\langle d_zf^m\xi,d_zf^m\eta\rangle e^{2m\frac{\chi+\chi_z}{2}},\text{ for }\xi,\eta\in H^s(z);$$
$$\langle \xi,\eta \rangle_{z,u}'':=2\sum_{m=0}^\infty\langle d_zf^{-m}\xi,d_zf^{-m}\eta\rangle e^{2m\frac{\chi+\chi_z}{2}},\text{ for }\xi,\eta\in H^u(z);$$
\begin{equation}\nonumber
\langle \xi,\eta \rangle_z'':=\langle \pi_s \xi,\pi_s \eta \rangle_{z,s}''+\langle \pi_u \xi,\pi_u \eta \rangle_{z,u}''.
\end{equation}
$\widetilde{C}_\chi(\cdot)$ is chosen measurably such that $\langle \xi,\eta \rangle_z''=\langle \widetilde{C}_\chi^{-1}(z)\xi,\widetilde{C}_\chi^{-1}(z)\eta \rangle$.
\end{itemize}
\end{definition}
 Notice that $\|\widetilde{C}^{-1}_\chi(z)\|\geq \|C_\chi^{-1}(z)\|$ for all $z\in NUH_\chi$, since 
\begin{align*}
    \langle \xi_s,\xi_s \rangle_{z,s}''=&2\sum_{m=0}^\infty| d_zf^m\xi_s|^2 e^{2m\chi}e^{m(\chi_z-\chi)}>\langle \xi_s,\xi_s \rangle_{z,s}'\text{ },\\
    \langle \xi_u,\xi_u \rangle_{z,u}''=&2\sum_{m=0}^\infty| d_zf^{-m}\xi_u|^2 e^{2m\chi}e^{m(\chi_z-\chi)}>\langle \xi_u,\xi_u \rangle_{z,u}'\text{ }.
\end{align*}

\begin{claim}\label{cocycle} For every $x\in NUH_\chi$, $D_\chi(x)$ extends to a cocycle the following way: $$D_\chi(x,n)=C_\chi(f^n(x))^{-1}\circ d_xf^n\circ C_\chi(x),$$
and the Oseledec theorem is applicable to it.
\end{claim}
\begin{proof} $NUH_\chi$ is $f$-invariant, and thus $D_\chi(x,n)$ is well defined. It is easy to check that $D_\chi(x,n)$ is a cocycle. By theorem \ref{pesinreduction}, $\|D_\chi^{\pm1}(x)\|\leq\kappa$ $\forall x\in NUH_\chi$; thus making $\log^+\|D_\chi^\pm(x)\|$ uniformly bounded on $NUH_\chi$, and in particularly integrable, as needed.
\end{proof}
\begin{definition}\label{scalingfuncs} {\em Scaling functions ($S/U$ parameters)}: For any $x\in M$, $\xi\in T_xM$ we define:
$$S^2(x,\xi)=2\sum_{m=0}^\infty|d_xf^m\xi|^2e^{2\chi m}, U^2(x,\xi)=2\sum_{m=0}^\infty|d_xf^{-m}\xi|^2e^{2\chi m}\in [2|\xi|^2,\infty].$$
Here, $|\cdot|$ is the Riemannian norm on $T_xM$.
\end{definition}
\noindent\textbf{Remark}: $U(x,\xi),S(x,\xi)$ will be calculated in two cases: \begin{enumerate}
    \item Where $x\in NUH_\chi$, $\xi\in H^{u/s}(x)$ respectively (in this case the quantities are finite).
    \item Where $x$ belongs to a $u/s$- manifold $V^{u/s}$ (definition \ref{def135}), and $\xi\in T_x V^{u/s}$ respectively.
\end{enumerate}
\begin{theorem}\label{SandU} There exists an $F_0>0$ only depending on $f,\chi$ and $M$  such that for any $x\in NUH_\chi$: $$\frac{\|C_\chi^{-1}(f(x))\|}{\|C_\chi^{-1}(x)\|}\in[F_0^{-1},F_0]$$
\end{theorem}
\begin{proof} Notice that by definition:
$$\|C_\chi^{-1}(x)\|^2=\sup_{\xi\in H^s(x),\eta\in H^u(x),\xi+\eta\neq0}\frac{S^2(x,\xi)+U^2(x,\eta)
}{|\xi+\eta|^2}.$$

\textit{Claim 1}: For a fixed $x\in NUH_\chi$ $(\xi,\eta)\mapsto \frac{S^2(x,\xi)+U^2(x,\eta)}{|\xi+\eta|^2}$ is continuous where defined.

\textit{Proof}: The norms of vectors are clearly  continuous. So is
$S^2$,
since $S^2(x,\xi)=\langle\xi,\xi\rangle_{x,s}'$- 
a bilinear form
on a finite-dimensional linear space
. The same argument applies for $U^2$
. This concludes the proof of claim 1.

\medskip
A continuous function on a compact set attains its maximum, and the mapping $(\xi,\eta)\mapsto \frac{S^2(x,\xi)+U^2(x,\eta)}{|\xi+\eta|^2}$ is homogeneous. So let $\xi\in H^s(f(x))$ and $\eta\in H^u(f(x))$ be any two tangent vectors s.t. $\frac{S^2(f(x),\xi)+U^2(f(x),\eta)}{|\xi+\eta|^2}$ is maximal. Whence,
\begin{align}\label{eq1}\frac{\|C_\chi^{-1}(f(x))\|^2}{\|C_\chi^{-1}(x)\|^2}=&\frac{\sup_{\xi_1,\eta_1} (S^2(f(x),\xi_1)+U^2(f(x),\eta_1))/|\xi_1+\eta_1|^2}{\sup_{\xi_2,\eta_2} (S^2(x,\xi_2)+U^2(x,\eta_2))/|\xi_2+\eta_2|^2}\nonumber\\
\leq&\frac{\sup_{\xi_1,\eta_1} (S^2(f(x),\xi_1)+U^2(f(x),\eta_1))/|\xi_1+\eta_1|^2}{(S^2(x,d_{f(x)}f^{-1}\xi)+U^2(x,d_{f(x)}f^{-1}\eta))/|d_{f(x)}f^{-1}\xi+d_{f(x)}f^{-1}\eta|^2}\nonumber\\
    =&\frac{S^2(f(x),\xi)+U^2(f(x),\eta)}{S^2(x,d_{f(x)}f^{-1}\xi)+U^2(x,d_{f(x)}f^{-1}\eta)}\cdot\frac{|d_{f(x)}f^{-1}\xi+d_{f(x)}f^{-1}\eta|^2}{|\xi+\eta|^2}.
\end{align}
\textit{Claim 2}: The left fraction of (\ref{eq1}) is bounded away from $0$ and $\infty$ uniformly by a constant (and its inverse) depending only on $f,M$ and $\chi$.

\textit{Proof}: Let $v\in H^s(x)$ (WLOG $|v|=1$),
$$S^2(f(x),d_xfv)=2\sum_{m=0}^\infty|d_{f(x)}f^md_xfv|^2e^{2\chi m}=\Big(2\sum_{m=0}^\infty|d_{x}f^mv|^2e^{2\chi m}-2\Big)e^{-2\chi}=e^{-2\chi}(S^2(x,v)-2).$$
This, and the inequality $S^2(f(x),d_xfv)>2$, yields
$$\frac{1}{1+e^{-2\chi}}e^{-2\chi}S^2(x,v)<S^2(f(x),d_xfv)< e^{-2\chi}S^2(x,v).$$
A similar calculation gives $\frac{U^2(f(x),d_xfv)}{U^2(x,v)}\in[e^{2\chi},e^{2\chi}+M_f^2]$. The claim follows.

\textit{Claim 3}: The right fraction of (\ref{eq1}) is bounded away from $0$ and $\infty$ uniformly by constants
depending only on $f$ and $M$.

\textit{Proof}: Recall that $M_f:=\max_{x\in M}\{\|d_xf\|,\|d_xf^{-1}\|^{-1}\}$. Then
$$\frac{|d_{f(x)}f^{-1}\xi+d_{f(x)}f^{-1}\eta|^2}{|\xi+\eta|^2}=\frac{|d_{f(x)}f^{-1}(\xi+\eta)|^2}{|\xi+\eta|^2}\in [M_f^{-2},M_f^2].$$
The three claims together give the upper bound in the theorem. If we repeat the previous argument with $\xi,\eta$ maximizing the denominator instead of the numerator, we get the lower bound.
\end{proof}
\begin{lemma}\label{contraction}
$C_\chi, \widetilde{C}_\chi$ are contractions (w.r.t to the Eucleadan norm on $\mathbb{R}^d$ and the Riemannian norm on the tangent space).
\end{lemma}
\begin{proof} 
$\forall v_{s}\in H^{s}$,
$|C_\chi^{-1}(x)v_{s}|^2={|v_{s}|'}_{x,s}^2>2|v_{s}|^2$. Similarly $\forall v_{u}\in H^{u}$,
$|C_\chi^{-1}(x)v_{u}|^2={|v_{u}|'}_{x,u}^2>2|v_{u}|^2$. Hence, $\forall w\in \mathbb{R}^d:w=w_s+w_u$:
$$|C_\chi(x)w|^2=(|C_\chi(x)w_s|+|C_\chi(x)w_u|)^2<(\frac{1}{2}(|w_s|+|w_u|))^2\leq(|w_s+w_u|)^2,$$
where the last inequality is due to the fact that $w_s\bot w_u$. The proof that $\widetilde{C}_\chi$ is a contraction is similar, and we leave it to the reader.
\end{proof}

\begin{definition}\text{ }

    \begin{itemize}

    \item Similarly to claim \ref{cocycle}, $\widetilde{D}_\chi(x):=\widetilde{C}^{-1}_\chi(f(x))\circ d_x f\circ \widetilde{C}_\chi(x)$ ($x\in NUH_\chi$) also extends to a cocycle, as follows: $\widetilde{D}_\chi(x,n):=\widetilde{C}_\chi^{-1}(f^n(x))\circ d_xf^n\circ \widetilde{C}_\chi(x)$, $x\in NUH_\chi$. Define $NUH_\chi^\dagger$ to be the set of all points $x\in NUH_\chi$ which are Lyapunov regular w.r.t the invertible cocycle $\widetilde{D}_\chi(x,n)$. This is well defined by the arguments in claim \ref{cocycle}.

    \item \normalfont $\forall N\in\mathbb{N}$: $X_N:=\{x\in NUH_\chi^\dagger: N-1<\|\widetilde{C}_\chi^{-1}(x)\|\leq N\}$.
    Let $Y_N\subset X_N$ be the set of points in $X_N$ which return to $X_N$ infinitely many times with iteration of $f$, and of $f^{-1}$. By Poincar\'{e}'s recurrence theorem, $Y_N$ is of equal measure to $X_N$ w.r.t any $\chi$-hyperbolic measure. Define
    $$NUH_\chi^*:=\bigcupdot_{N\in\mathbb{N}}Y_N.$$
    \end{itemize}

\end{definition}
 Notice that $NUH_\chi^\dagger$ is of full measure w.r.t all $\chi$-hyperbolic measures as an intersection of two full measure sets, and that $NUH_\chi^*$ is of full measure w.r.t all $\chi$-hyperbolic measures, and both are $f$-invariant. 
    
\begin{claim}\label{tempered}
$$\lim_{|n|\rightarrow\infty}\frac{1}{n}\log\|\widetilde{C}_\chi^{\pm1}(f^n(x))\|=0,\lim_{|n|\rightarrow\infty}\frac{1}{n}\log\|C_\chi^{\pm1}(f^n(x))\|=0\text{ }\forall x\in NUH_\chi^*.$$
\end{claim}
\begin{proof}\footnote{The following short elegant proof (of a slightly weaker statement) has been shown to me by the referee: Define $a(x):=\log\|\widetilde{C}_\chi^{-1}(x)\|$, and $b(x):=a(f(x))-a(x)$. Then by theorem \ref{SandU}, $|b|\leq\log F_0$, whence $b\in L^2(M)$ for every $\chi$-hyperbolic measure. We wish to estimate $\lim_{|n|\rightarrow\infty}\Big(\frac{1}{n}\sum_{k=0}^{n-1}b(f^k(x))+\frac{1}{n}a(x)\Big)$. It is then a standard result in Ergodic theory that this limit is 0 (even though $a(\cdot)$ is not necessarily intergable) for a.e. point w.r.t every $\chi$-hyperbolic measure. Note that this argument does not show the existence of the limit on $NUH_\chi^*$.} (See \cite{KM}, last part of theorem S.2.10) Proof that $\lim\limits_{n\rightarrow\pm\infty}\frac{1}{n}\log\|\widetilde{C}_\chi^{-1}(f^n(x))\|=0$ (it will conclude the case for $\|C_\chi^{-1}\|$ as well, since $1\leq\|C_\chi^{-1}(x)\|\leq\|\widetilde{C}_\chi^{-1}(x)\|$): Let $x\in NUH_\chi^*, \xi\in H^s(x)$. By definition, $$|\widetilde{D}_\chi(x,n)\widetilde{C}_\chi^{-1}\xi|=|\widetilde{C}_\chi^{-1}(f^n(x))d_xf^n\xi|.$$
By the definition of $NUH_\chi^*$, $\exists N_x\in\mathbb{N}$ s.t. $x\in Y_{N_x}$. Thus, $\exists n_k\uparrow\infty$ s.t. $f^{n_k}(x)\in Y_{N_x}$ for all $k\in\mathbb{N}$. Also by the definition of $NUH_\chi^*$, as a subset of $NUH_\chi^\dagger$, the following limits exist,
\begin{itemize}
    \item $\lim\limits_{n\rightarrow\pm\infty}\frac{1}{n}\log|d_xf^n\xi|$,
    \item $\lim\limits_{n\rightarrow\pm\infty}\frac{1}{n}\log|\widetilde{D}_\chi(x,n)\widetilde{C}_\chi^{-1}(x)\xi|$.
\end{itemize}
Hence,
\begin{align*}
\lim\limits_{n\rightarrow\pm\infty}\frac{1}{n}\log|\widetilde{D}_\chi(x,n)\widetilde{C}_\chi^{-1}(x)\xi|=&\lim\limits_{k\rightarrow\infty}\frac{1}{n_k}\log|\widetilde{D}_\chi(x, n_k)\widetilde{C}_\chi^{-1}(x)\xi|=\lim\limits_{k\rightarrow\infty}\frac{1}{n_k}\log|\widetilde{C}_\chi^{-1}(f^{n_k}(x))d_xf^{n_k}\xi|\\
=&\lim\limits_{k\rightarrow\infty}\frac{1}{n_k}\log|d_xf^{n_k}\xi| (\because x\in NUH_\chi^*, f^{n_k}(x)\in N_x \forall k\in\mathbb{N})\\
=&\lim\limits_{n\rightarrow\pm\infty}\frac{1}{n}\log|d_xf^n\xi|.
\end{align*}
Thus we get that $df$ and $\widetilde{D}_\chi$ have the same Lyapunov spectrum. Furthermore, denote the Lyapunov space corresponding to $\chi_i$ w.r.t to $\widetilde{D}_\chi(x)$ by $\widetilde{H}_i(x)$, then we get $\widetilde{H}_i(x)=\widetilde{C}_\chi^{-1}(x)[H_i(x)]$.

By definition, for any $n\in \mathbb{Z}$, $\widetilde{C}_\chi^{-1}(f^n(x))=\widetilde{D}_\chi(x,n)\circ \widetilde{C}_\chi^{-1}(x)\circ (d_xf^n)^{-1}$. For every $\eta\in T_{f^n(x)}M$, we can decompose it into components in the following way: $\eta=\eta_s+\eta_u$, $\eta_{s/u}\in H^{s/u}(f^n(x))$. Also, recall $\|\widetilde{C}_\chi(z)\|\leq1$ for any $z\in NUH_\chi$. Thus,
$$|\eta|\leq|\widetilde{C}_\chi^{-1}(f^n(x))\eta|\leq|\widetilde{D}_\chi(x,n)\circ \widetilde{C}_\chi^{-1}(x)\circ (d_xf^n)^{-1}\eta_s|+|\widetilde{D}_\chi(x,n)\circ \widetilde{C}_\chi^{-1}(x)\circ (d_xf^n)^{-1}\eta_u|.$$
The first inequality is due to lemma \ref{contraction}. Decompose $\eta_s$ into components from $\{H_i(f^n(x)): \chi_i<0\}$, and denote by $i$ the index of the Lyapunov space with the largest exponent s.t. $\eta_s$ has a component in it.
Since $\widetilde{H}_i(x)=\widetilde{C}_\chi^{-1}(x)[H_i(x)]$, $H_i(x)=(d_xf^n)^{-1}[H_i(f^n(x))]$ and $df$ and $\widetilde{D}_\chi$ have the same Lyapunov spectrum, we get that the left summand equals $e^{\chi_in+o(n)}e^{-\chi_in+o(n)}|\eta_s|=e^{o(n)}|\eta_s|$. The same arguments can be done for the right summand. Thus, 
$$|\eta|\leq|\widetilde{C}_\chi^{-1}(f^n(x))\eta|\leq e^{o(n)}|\eta_s|+e^{o(n)}|\eta_u|.$$
By the Oseledec theorem, we got that  $\sphericalangle(H^s(f^n(x)),H^u(f^n(x)))$  does not diminish exponentially\footnote{See proposition \ref{Xi}, \textit{Part 1}, for definition of such an angle and more details regarding the use of the Sine theorem.}. By the Sine theorem, $|\eta_{s/u}|\leq \frac{|\eta|}{\sin\sphericalangle(H^s(f^n(x)),H^u(f^n(x)))}
$ . Thus,  $$|\eta|\leq|\widetilde{C}_\chi^{-1}(f^n(x))\eta|\leq (e^{o(n)}+e^{o(n)})|\eta|\cdot\frac{1}{\sin\sphericalangle(H^s(f^n(x)),H^u(f^n(x)))}=e^{o(n)}|\eta|\cdot e^{o(n)}=e^{o(n)}|\eta|.$$
Since $\eta\in T_{f^n(x)}M$ was arbitrary, we get that $lim_{n\rightarrow\pm\infty}\frac{1}{n}\log\|\widetilde{C}_\chi^{-1}(f^n(x))\|=0$ as required. The case for $\|\widetilde{C}_\chi(f^n(x))\|$ is done similarly by the identity $\widetilde{C}_\chi(f^n(x))=d_xf^n\circ \widetilde{C}_\chi(x)\circ \widetilde{D}_\chi(x,n)^{-1}$.

\end{proof}

\subsubsection{Pesin charts}\label{forrho}
Let $\exp_x:T_xM\rightarrow M$ be the exponential map. Since $M$ is compact, there exist $r=r(M),\rho=\rho(M) >0$ s.t. $\forall x\in M$ $\exp_x$ maps $B_{\sqrt{d}\cdot r}^x:=\{v\in T_xM: |v|_x<\sqrt{d\cdot} r\}$ diffeomorphically onto a neighborhood of $B_\rho(x)=\{y\in M: d(x,y)<\rho\}$, where $d(\cdot,\cdot)$ is the distance function on $M\times M$, w.r.t to the Riemannian metric.
\begin{definition}\label{balls}$\text{ }$\\
\begin{enumerate}
\item[(a)] 
    \begin{enumerate}
        \item[(i)] $B_\eta(0), R_\eta(0)$ are open balls located at the origin with radius $\eta$ in $\mathbb{R}^d$ w.r.t Euclidean norm and supremum norm respectively.
        \item[(ii)] For every $x\in M$, we define the following open ball in $T_xM$: $\forall v\in T_xM, r>0$, $B_r^x(x):=\{w\in T_xM: |v-w|<r\}$.
    \end{enumerate}

\item[(b)] We take $\rho$ so small that $(x,y)\mapsto \exp_x^{-1}(y)$ is well defined and 2-Lipschitz on $B_\rho(z)\times B_\rho(z)$ for $z\in M$, and so small that $\|d_y \exp_x^{-1}\|\leq2$ for all $y\in B_\rho(x)$.

\item[(c)] For all $x\in NUH_\chi^*$: $$\psi_x:=\exp_x\circ C_\chi(x):R_r(0)\rightarrow M.$$ Since $C_\chi(x)$ is a contraction, $\psi_x$ maps $R_r(0)$ diffeomorphically into $M$.

\item[(d)] $$f_x:=\psi_{f(x)}^{-1}\circ f\circ \psi_x :R_r(0)\rightarrow \mathbb{R}^d.$$ Then we get $d_0f_x=D_\chi(x)$ (c.f. theorem \ref{pesinreduction}).

\item[(e)] $\forall\epsilon>0: I_\epsilon:=\{e^{-\frac{1}{3}\ell\epsilon}\}_{l\in\mathbb{N}}$, $\widetilde{Q}_\chi(x):=\frac{1}{3^{6/\beta}}\epsilon^\frac{90}{\beta}\|C_\chi^{-1}(x)\|^\frac{-48}{\beta}$,
$$Q_\epsilon(x):=\max\{q\in I_\epsilon |q\leq \widetilde{Q}_\chi(x)\}$$
\end{enumerate}
\end{definition}
This definition is a bit different than Sarig's, and will come in handy in some calculations later on.
\noindent\textbf{Remark}: Depending on context, the notations of $(a)(i)$ will be used to denote balls with the same respective norms in some subspace of $\mathbb{R}^d$ (i.e. $\{v\in \mathbb{R}^{s(x)}:|v|<\eta\}$, $x\in NUH_\chi(f)$).

\begin{theorem}\label{linearization} (See \cite{BP}, theorem 5.6.1) For all $\epsilon$ small enough, and $x\in NUH_\chi^*$:
\begin{enumerate}
    \item $\psi_x(0)=x$ and $\psi_x:R_{Q_\epsilon(x)}(0)\rightarrow M$ is a diffeomorphism onto its image s.t. $\|d_u\psi_x\|\leq2$ for every $u\in R_{Q_\epsilon(x)}(0)$.
\item $f_x$ is well defined and injective on $R_{Q_\epsilon(x)}(0)$, and \begin{enumerate}
\item $f_x(0)=0,d_0f_x=D_\chi(x)$. 
\item $\|f_x-d_0f_x\|_{C^{1+\beta/2}}\leq \epsilon$ on $R_{Q_\epsilon(x)}(0)$.    
\end{enumerate}
\item The symmetric statement holds for $f_x^{-1}$.
\end{enumerate}
\end{theorem}
\begin{proof} This proof is the same as \cite[theorem~5.6.1]{BP} and \cite[theorem~2.7]{Sarig}.
\end{proof}
\begin{definition} Suppose $x\in NUH_\chi^*$ and $\eta\in (0,Q_\epsilon(x)]$, then the Pesin Chart $\psi_x^\eta$ is the map $\psi_x:R_\eta(0)\rightarrow M$. Recall, $R_\eta(0)=[-\eta,\eta]^d$.
\end{definition}
\begin{lemma}\label{omega0} For all $\epsilon$ small enough, for every $x\in NUH_\chi^*$:

\begin{enumerate}
\item[(1)] $Q_\epsilon(x)<\epsilon^{3/\beta}$ on $NUH_\chi^*$.

\item[(2)]$\|C_\chi^{-1}(f^i(x))\|^{48}<\epsilon^{2/\beta}/Q_\epsilon(x)$ for i=-1,0,1.

\item[(3)] For all $t>0$ $\#\{Q_\epsilon(x)| Q_\epsilon(x)>t, x\in M\}<\infty$.

\item[(4)] $\exists\omega_0$ depending only on $\chi$, $M$ and $f$ s.t. $\omega_0^{-1}\leq Q_\epsilon\circ f/Q_\epsilon\leq \omega_0$ on $NUH_\chi^*$.

\item[(5)]$\exists q_\epsilon:NUH_\chi^*\rightarrow(0,1)$ s.t. $q_\epsilon\leq\epsilon\cdot Q_\epsilon$ and $e^{-\epsilon}\leq\frac{q_\epsilon\circ f}{q_\epsilon}\leq e^\epsilon$.
\end{enumerate}
\end{lemma}
\begin{proof}
(1) Clear by definition of $\widetilde{Q}_\epsilon(x)$.

(2)$\|C_\chi^{-1}(x)\|^{48}=\widetilde{Q}_\epsilon^{-\beta}(x)\cdot\frac{1}{3^{6}}\epsilon^{90}\leq \epsilon^3/Q_\epsilon^\beta(x)=\epsilon^{2/\beta}/Q_\epsilon(x)\cdot \frac{\epsilon^{3-2/\beta}}{Q_\epsilon^{\beta-1}(x)}=\epsilon^{2/\beta}/Q_\epsilon(x)\cdot \frac{{(\epsilon^{3\beta-2})}^{1/\beta}}{Q_\epsilon^{\beta-1}(x)}=$

$=\epsilon^{2/\beta}/Q_\epsilon(x)\cdot {(\frac{\epsilon^{3/\beta}}{Q_\epsilon(x)})}^{\beta-1}\cdot \epsilon^{1/\beta}$.

Hence by part (1) and the fact that $\beta<1$, $$\|C_\chi^{-1}(x)\|^{48}\ <\Big(\epsilon^{2/\beta}/Q_\epsilon(x)\Big)\cdot \epsilon^{1/\beta}.$$
Using part (4), and noticing $\epsilon^{1/\beta}<F_0^{-1}$ (for small enough $\epsilon$; recall $F_0$ from theorem \ref{SandU}):
$$\|C_\chi^{-1}(f^ix)\|^{48}\ <\epsilon^{2/\beta}/Q_\epsilon(x) \text{  where i=-1,0,1}.$$

(3) For all $t>0$: $\{Q_\epsilon(x)| Q_\epsilon(x)>t, x\in M\}\subset\{e^{-\epsilon l\frac{1}{3}}\}_{l\in\mathbb{N}}\cap (t,\epsilon^{3/\beta}]$.

(4) By theorem \ref{SandU}:  $F_0^{-1}\leq\frac{\|C_\chi^{-1}(f(x))\|}{\|C_\chi^{-1}(x)\|}\leq F_0$ on $NUH_\chi^*$. Now $\widetilde{Q}_\epsilon(x)=(\|C_\chi^{-1}(x)\|)^{-48/\beta}\epsilon^{90/\beta}\frac{1}{3^{6/\beta}}$ hence
$$\frac{\widetilde{Q}_\epsilon(f(x))}{\widetilde{Q}_\epsilon(x)}\in[F_0^{-48/\beta},F_0^{48/\beta}].$$
By definition of $Q_\epsilon(x)$: $Q_\epsilon\in(e^{-\epsilon/3}\widetilde{Q}_\epsilon,\widetilde{Q}_\epsilon]$, hence
 $$e^{-\epsilon/3}F_0^{-48/\beta}\leq Q_\epsilon\circ f/Q_\epsilon\leq F_0^{48/\beta}.$$
 Set $\omega_0:=e^{\epsilon/3}F_0^{48/\beta}$.

(5) By claim \ref{tempered} $\lim\limits_{n\rightarrow\pm\infty}\frac{1}{n}\log\widetilde{Q}_\epsilon(f^n(x))=0$, and hence by the fact seen in (4) that $Q_\epsilon(x)\in(e^{-\epsilon}\widetilde{Q}_\epsilon(x),\widetilde{Q}_\epsilon(x)]$ we get also $\lim\limits_{n\rightarrow\pm\infty}\frac{1}{n}\log Q_\epsilon(f^n(x))=0$. 
$$\frac{1}{q_\epsilon(x)}:=\frac{1}{\epsilon}\sum_{k\in\mathbb{Z}} e^{-|k|\epsilon}\frac{1}{Q_\epsilon(f^k(x))}.$$
The sum converges because $\frac{1}{k}\log Q_\epsilon(f^k(x))\xrightarrow[k\rightarrow\pm\infty]{}0$, and it is easy to check that the sum behaves as we wish (Pesin's Tempering Kernel Lemma \cite{BP}, lemma 3.5.7).
\end{proof}

\begin{lemma}
Let $(X,\mathcal{B},\mu,T)$ be a p.p.t . Let $q:X\rightarrow\mathbb{R}^+$ be a measurable function. Then the set $\{x\in X: \overline{\lim}q(T^nx)>0\}$ has full measure.
\end{lemma}
\begin{proof}
Let $A_m:=\{x: q(x)\in[2^{-m},2^{-m+1})\}, m\in\mathbb{Z}$. By Poincar\'{e}'s recurrence theorem:

$$B_m:=\{x\in A_m: \#\{n:T^nx\in A_m\}=\infty\}\subset A_m\text{ and  }\mu(A_m)=\mu(B_m).$$


$X=\bigcupdot\limits_{m\in\mathbb{Z}}A_m=\bigcupdot\limits_{m\in\mathbb{Z}}B_m\text{ mod }\mu$, and for all $x\in B_m$: $\overline{\lim}q(T^nx)\geq 2^{-m}>0$.
\end{proof}

\begin{definition}\label{nuhchisharp}
$$NUH_\chi^\#:=\{x\in NUH_\chi^*|\limsup\limits_{n\rightarrow\infty}q_\epsilon(f^n(x)),\limsup\limits_{n\rightarrow\infty}q_\epsilon(f^{-n}(x))>0\}.$$ 
\end{definition}
By the previous lemma, $NUH_\chi^\#$ has full measure. This is the set we will be interested in from now on.
\subsection{Overlapping charts}
\subsubsection{The overlap condition}
The following definition is from \cite{Sarig}, \textsection 3.1.
\begin{definition}\label{isometries}$\text{ }$\\
\begin{enumerate}
\item[(a)] For every $x\in M$ there is an open neighborhood $D$ of diameter less than $\rho$ and a smooth map $\Theta_D:TD\rightarrow\mathbb{R}^d$ s.t.:
\begin{enumerate}
    \item[(1)] $\Theta_D:T_xM\rightarrow\mathbb{R}^d$ is a linear isometry for every $x\in D$

    \item[(2)] Define $\nu_x:=\Theta_D|_{T_xM}^{-1}:\mathbb{R}^d\rightarrow T_xM$, then $(x,u)\mapsto(\exp_x\circ\nu_x)(u)$ is smooth and Lipschitz on $D\times B_2(0)$ w.r.t the metric $d(x,x')+|u-u'|$

    \item[(3)] $x\mapsto\nu_x^{-1}\circ \exp_x^{-1}$ is a Lipschitz map from $D$ to $C^2(D,\mathbb{R}^d)=\{C^2-\text{maps from }D\text{ to }\mathbb{R}^d\}$.

    Let $\mathcal{D}$ be a finite cover of $M$ by such neighborhoods. Denote with $\varpi(\mathcal{D})$ the Lebesgue number of that cover: If $d(x,y)<\varpi(\mathcal{D})$ then $x$ and $y$ belong to the same $D$ for some $D$.
\end{enumerate}
\item[(b)] We say that two Pesin charts $\psi_{x_1}^{\eta_1},\psi_{x_2}^{\eta_2}$ {\em $\epsilon-$overlap} if $e^{-\epsilon}<\frac{\eta_1}{\eta_2}<e^\epsilon$, and for some $D\in\mathcal{D}$, $x_1,x_2\in D$ and $d(x_1,x_2)+\|\Theta_D\circ C_\chi(x_1)-\Theta_D\circ C_\chi(x_2)\|<\eta_1^4\eta_2^4$.
\end{enumerate}

\noindent\textbf{Remark}: The overlap condition is symmetric and monotone: if $\psi_{x_i}^{\eta_i},i=1,2$ $\epsilon-$overlap, then $\psi_{x_i}^{\xi_i},i=1,2$ $\epsilon-$overlap for all $\eta_i\leq\xi_i\leq Q_\epsilon(x_i)$ s.t. $e^{-\epsilon}<\frac{\xi_1}{\xi_2}<e^\epsilon$.
    
\end{definition}

\begin{prop}\label{chartsofboxes} The following holds for all $\epsilon$ small enough: If $\psi_x:R_\eta(0)\rightarrow M$ and $\psi_y:R_\zeta(0)\rightarrow M$ $\epsilon$-overlap, then: \begin{enumerate}
\item $\psi_x[R_{e^{-2\epsilon}\eta}(0)]\subset\psi_y[R_\zeta(0)]$ and $\psi_y[R_{e^{-2\epsilon}\zeta}(0)]\subset\psi_x[R_\eta(0)]$,
\item $dist_{C^{1+\beta/2}}(\psi_{x/y}^{-1}\circ\psi_{y/x},Id)<\epsilon\eta^2\zeta^2$ (recall the notations of \ref{notations}, section 4) where the $C^{1+\beta/2}$ distance is calculated on $R_{e^{-\epsilon} r(M)}(0)$ and $r(M)$ is defined in \textsection1.1.3.
\end{enumerate}
\end{prop}
\begin{proof}


Suppose $\psi_x^\eta$ and $\psi_y^\zeta$ $\epsilon$-overlap, and fix some $D\in\mathcal{D}$ which contains $x,y$ such that $d(x,y)+\|\Theta_D\circ C_\chi(x)-\Theta_D\circ C_\chi(y)\|<\eta^4\zeta^4$. For $z=x,y$, write $C_{z}=\Theta_D\circ C_\chi(z)$, and $\psi_{z}=\exp_{z}\circ \nu_{z}\circ C_{z}$. By the definition of Pesin charts, $\eta\leq Q_\epsilon(x)< \epsilon^\frac{90}{\beta}\|C_\chi^{-1}(x)\|^\frac{-48}{\beta}$, $\zeta\leq Q_\epsilon(y)<\epsilon^\frac{90}{\beta}\|C_\chi^{-1}(y)\|^\frac{-48}{\beta}$. In particular, $\eta,\zeta\leq \epsilon^\frac{3}{\beta}$. 
Our first constraint on $\epsilon$ is that it would be less than $1$, and so small that $$\epsilon^{3/\beta}<\epsilon<\frac{\min\{1,r(M),\rho(M)\}}{5(L_1+L_2+L_3+L_4)^3},$$ 
where $r(M),\rho(M)$ are defined on \textsection1.1.3, and: \begin{enumerate}
\item $L_1$ is a uniform Lipschitz constant for the maps $(x,v)\mapsto(\exp_x\circ\nu_x)(v)$ on $D\times B_{r(M)}(0)$ ($D\in\mathcal{D}$).
\item  $L_2$ is a uniform Lipschitz constant for the maps $x\mapsto\nu_x^{-1}\exp_x^{-1}$ from $D$ into $C^2(D,\mathbb{R}^2)$ ($D\in\mathcal{D}$).
\item $L_3$ is a uniform Lipschitz constant for $\exp_x^{-1}:B_{\rho(M)}(x)\rightarrow T_xM$ ($x\in M$).
\item $L_4$ is a uniform Lipschitz constant for $\exp_x:B_{r(M)}(0)\rightarrow M$ ($x\in M$).
\end{enumerate}
We assume WLOG that these constants are all larger than one. Now proceed as in \cite[proposition~3.2]{Sarig}.
\end{proof}

\noindent\textbf{Remark}: By item (2) in the proposition above, the greater the distortion of $\psi_x$ or $\psi_y$ the closer they are one to another. This ``distortion compensating bound" will be used in the sequel to argue that $\psi_{f(x)}^{-1}\circ f\circ\psi_x$ remains close to a linear hyperbolic map if we replace $\psi_{f(x)}$ by an overlapping map $\psi_y$  (proposition  \ref{3.4inomris} below).

\begin{lemma}\label{overlap} Suppose $\psi_{x_1}^{\eta_1},\psi_{x_2}^{\eta_2}$ $\epsilon-$overlap, and write $C_1=\Theta_D\circ C_\chi(x_1), C_2=\Theta_D\circ C_\chi(x_2)$ for the $D$ s.t. $D\ni x_1,x_2$. Then \begin{enumerate}
\item 
$\|C_1^{-1}-C_2^{-1}\|<2\epsilon \eta_1 \eta_2$,
\item $\frac{\|C_1^{-1}\|}{\|C_2^{-1}\|}=e^{\pm \eta_1\eta_2}.$
\end{enumerate}
\end{lemma}
The proof is similar to \cite[proposition~3.2]{Sarig}.
\begin{proof} 

$\psi_{x_2}^{-1}\circ\psi_{x_1}$ maps $R_{e^{-\epsilon}\eta_1}(0)$ into $\mathbb{R}^d$. Its derivative at the origin is :
\begin{align*}
A  :=&C_\chi^{-1}(x_2)d_{x_1}\exp_{x_2}^{-1}C_\chi(x_1)=C_2^{-1}d_{x_1}[\nu_{x_2}^{-1}\exp_{x_2}^{-1}]\nu_{x_1}C_1\\
=&C_2^{-1}C_1+C_2^{-1}(d_{x_1}[\nu_{x_2}^{-1}\exp_{x_2}^{-1}]-\nu_{x_1}^{-1})\nu_{x_1}C_1\\
 \equiv &C_2^{-1}C_1+C_2^{-1}(d_{x_1}[\nu_{x_2}^{-1}\exp_{x_2}^{-1}]-d_{x_1}[\nu_{x_1}^{-1}\exp_{x_1}^{-1}])\nu_{x_1}C_1.
\end{align*}
Let $L_1,...,L_4$ be as in the previous proof. Since $\|d_{x_1}[\nu_{x_2}^{-1}\exp_{x_2}^{-1}]-d_{x_1}[\nu_{x_1}^{-1}\exp_{x_1}^{-1}]\|\leq L_2d(x_1,x_2)<L_2\eta_1^4\eta_2^4<\epsilon\eta_1^2\eta_2^2$, and $\nu_{x_1}C_1$ is a contraction, and $\|A-Id\|\leq d_{C^1}(\psi_{x_2}^{-1}\circ\psi_{x_1},Id)<\epsilon\eta_1^2\eta_2^2$,
$$\|C_2^{-1}C_1-Id\|<2\epsilon\|C_2^{-1}\|\eta_1^2\eta_2^2.$$
We have that $\|C_1^{-1}-C_2^{-1}\|\leq\|C_1^{-1}\|\cdot\|C_2^{-1}C_1-Id\|\leq2\epsilon\|C_1^{-1}\|\cdot\|C_2^{-1}\|\eta_1^2\eta_2^2<2\epsilon \eta_1\eta_2$. This concludes part 1.

For part 2: 
$$\frac{\|C_1^{-1}\|}{\|C_2^{-1}\|}\leq1+\frac{\eta_1\eta_2}{\|C_2^{-1}\|}\leq e^{ \eta_1\eta_2},$$
and by symmetry $\frac{\|C_1^{-1}\|}{\|C_2^{-1}\|}=e^{\pm\eta_1\eta_2}$
\end{proof}
\subsubsection{The form of $f$ in overlapping charts}

We introduce a notation: for two vectors $u\in\mathbb{R}^{d_1}$ and $v\in\mathbb{R}^{d_2}$, $(-u-,-v-)\in\mathbb{R}^{d_1+d_2}$ means the new vector whose coordinates are the coordinates of these two, put in the same order as written.
\begin{prop}\label{3.4inomris} The following holds for all $\epsilon$ small enough. Suppose $x,y\in NUH_\chi^*$, $s(x)=s(y)$ and $\psi_{f(x)}^\eta$ $\epsilon-$overlaps $\psi_y^{\eta'}$, then $f_{xy}:=\psi_y^{-1}\circ f\circ\psi_x$ is a well defined injective map from $R_{Q_\epsilon(x)}(0)$ to $\mathbb{R}^d$, and there are matrices $D_s:\mathbb{R}^{s(x)}\rightarrow\mathbb{R}^{s(x)}, D_u:\mathbb{R}^{d-s(x)}\rightarrow\mathbb{R}^{d-s(x)}$ and differentiable maps $h_s:R_{Q_\epsilon(x)}(0)\rightarrow \mathbb{R}^{s(x)},h_u:R_{Q_\epsilon(x)}(0)\rightarrow \mathbb{R}^{u(x)}$, s.t. $f_{xy}$ can be put in the form
$$f_{xy}(-v_s-,-v_u-)=(D_s v_s+h_s(v_s,v_u),D_u v_u+h_u(v_s,v_u)),$$ where $v_{s/u}$ are the ``stable"/"unstable" components of the input vector $u$ for $\psi_x$; and $\kappa^{-1}\leq\|D_s^{-1}\|^{-1},\|D_s\|\leq e^{-\chi}$, $e^\chi\leq\|D_u^{-1}\|^{-1},\|D_u\|\leq\kappa$, $|h_{s/u}(0)|<\epsilon\eta$, $\|\frac{\partial(h_s,h_u)}{\partial(v_s,v_u)}\|<\epsilon\eta^{\beta/3}$ on $R_\eta(0)$ (in particular $\|\frac{\partial(h_s,h_u)}{\partial(v_s,v_u)}\rvert_0\|<\epsilon\eta^{\beta/3}$), and $\|\frac{\partial(h_s,h_u)}{\partial(v_s,v_u)}\rvert_{v_1}-\frac{\partial(h_s,h_u)}{\partial(v_s,v_u)}\rvert_{v_2}\|\leq\epsilon|v_1-v_2|^{\beta/3}$ on $R_{Q_\epsilon(x)}(0)$. A similar statement holds for $f_{xy}^{-1}$, assuming that $\psi_{f^{-1}(y)}^{\eta'}$ $\epsilon-$overlaps $\psi_x^\eta$.
\end{prop}
For proof see \cite[proposition~3.4]{Sarig}.

\subsubsection{Coarse graining}
Recall the definition of $s(x)$ for a Lyapunov regular $x\in NUH\chi$, in definition  \ref{NUH2}, and the assumption $s(x)\in\{1,...,d-1\}$.
\begin{prop}\label{discreteness} The following holds for all $\epsilon$ small enough: There exists a countable collection $\mathcal{A}$ of Pesin charts with the following properties:
\begin{enumerate}
\item {\em Discreteness}:  $\{\psi\in\mathcal{A}:\psi=\psi_x^\eta,\eta>t\}$ is finite for every $t>0$.
\item {\em Sufficiency}: For every $x\in NUH_\chi^*$ and for every sequence of positive numbers $0<\eta_n\leq e^{-\epsilon/3}Q_\epsilon(f^n(x))$ in $I_\epsilon$ s.t. $e^{-\epsilon}\leq\frac{\eta_n}{\eta_{n+1}}\leq e^\epsilon$ there exists a sequence $\{\psi_{x_n}^{\eta_n}\}_{n\in\mathbb{Z}}$ of elements of $\mathcal{A}$ s.t. for every $n$:
\begin{enumerate}
\item $\psi_{x_n}^{\eta_n}$ $\epsilon$-overlaps $\psi_{f^n(x)}^{\eta_n}$, $\frac{Q_\epsilon(f^n(x))}{Q_\epsilon(x_n)}=e^{\pm\frac{\epsilon}{3}}$ and $s(x_n)=s(f^n(x))$ \normalfont($=s(x)$, since Lyapunov exponents and dimensions are $f$-invariant);
\item $\psi_{f(x_n)}^{\eta_{n+1}}$ $\epsilon$-overlaps $\psi_{x_{n+1}}^{\eta_{n+1}}$;
\item $\psi_{f^{-1}(x_n)}^{\eta_{n-1}}$ $\epsilon$-overlaps $\psi_{x_{n-1}}^{\eta_{n-1}}$;
\item $\psi_{x_n}^{\eta_n'}\in\mathcal{A}$ for all $\eta_n'\in I_\epsilon$ s.t. $\eta_n\leq\eta_n'\leq\min\{Q_\epsilon(x_n),e^\epsilon\eta_n\}$.
\end{enumerate}
\end{enumerate}
\end{prop}
\begin{proof} The proof is the same as the proof of \cite[proposition~3.5]{Sarig} except for one difference: In the higher-dimensional case $s(x)\in\{1,...,d-1\}$ has $d-1$ possible values instead of just one. To deal with this we apply the discretization of \cite{Sarig} to $NUH_\chi^*\cap [s(x)=s]$ for each $s=1,..., d-1$.

\end{proof}
\subsection{$\epsilon$-chains and infinite-to-one Markov extension of $f$}\label{epsilonchains}
\subsubsection{Double charts and $\epsilon$-chains}

Recall that $\psi_x^\eta$ ($0<\eta\leq Q_\epsilon(x)$) stands for the Pesin chart $\psi_x:R_\eta(0)\rightarrow M$. An {\em $\epsilon$-double Pesin chart} (or just \enquote{double chart}) is a pair $\psi_x^{p^s,p^u}:=(\psi_x^{p^s},\psi_x^{p^u})$ where $0<p^u,p^s\leq Q_\epsilon(x)$.
\begin{definition}\label{edges} $\psi_{x}^{p^s,p^u}\rightarrow\psi_{y}^{q^s,q^u}$ means :
\begin{itemize}
\item $\psi_{x}^{p^s\wedge p^u}$ and $\psi_{f^{-1}(y)}^{p^s\wedge p^u}$ $\epsilon$-overlap;
\item $\psi_{f(x)}^{q^s\wedge q^u}$ and $\psi_{y}^{q^s\wedge q^u}$ $\epsilon$-overlap;
\item $q^u=\min\{e^\epsilon p^u,Q_\epsilon(y)\}$ and $p^s=\min\{e^\epsilon q^s,Q_\epsilon(x)\}$;
\item $s(x)=s(y)$ \normalfont (this requirement did not appear in Sarig's definition, as it holds automatically in the two dimensional hyperbolic case).
\end{itemize}
\end{definition}
\begin{definition}\label{defepsilonchains} (see \cite{Sarig}, definition 4.2)
$\{\psi_{x_i}^{p^s_i,p^u_i}\}_{i\in\mathbb{Z}}$ (resp. 
$\{\psi_{x_i}^{p^s_i,p^u_i}\}_{i\geq0}$ and 
$\{\psi_{x_i}^{p^s_i,p^u_i}\}_{i\leq0}$) is called an {\em $\epsilon$-chain} (resp. positive , negative $\epsilon$-chain) if $\psi_{x_i}^{p^s_i,p^u_i}\rightarrow\psi_{x_{i+1}}^{p^s_{i+1},p^u_{i+1}}$ for all $i$. We abuse terminology and drop the $\epsilon$ in ``$\epsilon$-chains".
\end{definition}

Let $\mathcal{A}$ denote the countable set of Pesin charts we have constructed in \textsection1.2.3 and recall that $I_\epsilon=\{e^{-k\epsilon/3}:k\in\mathbb{N}\}$.
\begin{definition}\label{graphosaurus} $\mathcal{G}$ is the directed graph with vertices $\mathcal{V}$ and $\mathcal{E}$ where:
\begin{itemize}
\item $\mathcal{V}:=\{\psi_{x}^{p^s,p^u}:\psi_{x}^{p^s\wedge p^u}\in \mathcal{A},p^s,p^u\in I_\epsilon, p^s,p^u\leq Q_\epsilon(x)\}$.
\item $\mathcal{E}:=\{(\psi_{x}^{p^s,p^u},\psi_{y}^{q^s,q^u})\in \mathcal{V}\times\mathcal{V}:\psi_{x}^{p^s,p^u}\rightarrow\psi_{y}^{q^s,q^u}\}$.
\end{itemize}
\end{definition}
This is a countable directed graph. Every vertex has a finite degree, because of the following lemma and the discretization of $\mathcal{A}$ (proposition  \ref{discreteness}). Let 
$$\Sigma(\mathcal{G}):=\{x\in\mathcal{V}^\mathbb{Z}:(x_i,x_{i+1})\in\mathcal{E} ,\forall i\in\mathbb{Z}\}.$$
$\Sigma(\mathcal{G})$ is the Markov shift associated with the directed graph $\mathcal{G}$. It comes equipped with the left-shift $\sigma$, and the standard metric.

\begin{lemma}\label{lemma131} If $\psi_x^{p^s,p^u}\rightarrow\psi_y^{q^s,q^u}$ then $\frac{q^u\wedge q^s}{p^u\wedge p^s}=e^{\pm\epsilon}$. Therefore for every $\psi_x^{p^s,p^u}\in\mathcal{V}$ there are only finitely many $\psi_y^{q^s,q^u}\in\mathcal{V}$ s.t. $\psi_x^{p^s,p^u}\rightarrow\psi_y^{q^s,q^u}$ or $\psi_y^{q^s,q^u}\rightarrow\psi_x^{p^s,p^u}$.
\end{lemma}
\begin{proof}
See lemma 4.4 in \cite{Sarig}.
\end{proof}

\begin{definition} Let $(Q_k)_{k\in\mathbb{Z}}$ be a sequence in $I_\epsilon=\{e^{-\frac{\ell\epsilon}{3}}\}_{l\in\mathbb{N}}$. A sequence of pairs $\{(p^s_k,p^u_k)\}_{k\in\mathbb{Z}}$ is called {\em $\epsilon$-subordinated} to $(Q_k)_{k\in\mathbb{Z}}$ if for every $k\in\mathbb{Z}$: \begin{itemize}
\item $0<p^s_k,p^u_k\leq Q_k$
\item $p^s_k,p^u_k\in I_\epsilon$
\item $p_{k+1}^u=\min\{e^\epsilon p_k^u,Q_{k+1}\}$ and $p_{k-1}^s=\min\{e^\epsilon p_k^s,Q_{k-1}\}$
\end{itemize}
\end{definition}
For example, if $\{\psi_{x_k}^{p^s_k,p^u_k}\}_{k\in\mathbb{Z}}$ is a chain, then $\{(p^s_k,p^u_k)\}_{k\in\mathbb{Z}}$ is $\epsilon$-subordinated to $\{Q_\epsilon(x_k)\}_{k\in\mathbb{Z}}$.

\begin{lemma}\label{subordinatedchain}(\cite[lemma~4.6]{Sarig})
Let $(Q_k)_{k\in\mathbb{Z}}$ be a sequence in $I_\epsilon$, and suppose $q_k\in I_\epsilon$ satisfy $0<q_k\leq Q_k$ and $\frac{q_k}{q_{k+1}}=e^{\pm\epsilon}$ for all $k\in\mathbb{Z}$. Then there exists a sequence $\{(p^s_k,p^u_k)\}_{k\in\mathbb{Z}}$ which is $\epsilon$-subordinated to $(Q_k)_{k\in\mathbb{Z}}$ and $p^s_k\wedge p^u_k\geq q_k$ for all $k\in\mathbb{Z}$.
\end{lemma}

\begin{lemma}(\cite[lemma~4.7]{Sarig})
 Suppose $\{(p^s_k,p^u_k)\}_{k\in\mathbb{Z}}$ is $\epsilon$-subordinated to a sequence $(Q_k)_{k\in\mathbb{Z}}\subset I_\epsilon$. If $\limsup_{n\rightarrow\infty}(p^s_n\wedge p^u_n)>0$ and $\limsup_{n\rightarrow-\infty}(p^s_n\wedge p^u_n)>0$, then $p_n^u$ (resp. $p_n^s$) is equal to $Q_n$ for infinitely many $n>0$, and for infinitely many $n<0$.
\end{lemma}

\begin{prop}\label{prop131}
For every $x\in NUH_\chi^\#$ there is a chain $\{\psi_{x_k}^{p_k^s,p_k^u}\}_{k\in\mathbb{Z}}\subset\Sigma(\mathcal{G})$ s.t. $\psi_{x_k}^{p_k^s\wedge p_k^u}$ $\epsilon$-overlaps $\psi_{f^k(x)}^{p_k^s\wedge p_k^u}$ for every $k\in\mathbb{Z}$.
\end{prop}
\begin{proof} This follows from proposition \ref{discreteness} as in \cite[proposition~4.5]{Sarig}. 
\end{proof}

\section{``Shadowing lemma"- admissible manifolds and the Graph Transform}
In this section we prove the shadowing lemma: for every $\epsilon$-chain $\{\psi_{x_i}^{p_i^s,p_i^u}\}_{i\in\mathbb{Z}}$ there is some $x\in M$ s.t. $f^i(x)\in \psi_{x_i}[R_{p^s_i\wedge p^u_i}(0)]$ for all $i$. We say that the chain shadows the orbit of $x$.

The construction of $x$, and the properties of the shadowing operation, are established using a graph transform argument. The bulk of this section is the detailed analysis of the graph transform and its properties. Historical account of the graph transform can be found in page \pageref{graphtransform}.

The following definition is motivated by \cite{KM,Sarig}.
\begin{definition}\label{def135} Recall $s(x)=\dim H^s(x)$.
Let $x\in NUH_\chi$, a {\em $u-$manifold} in $\psi_x$ is a manifold $V^u\subset M$ of the form
$$V^u=\psi_x[\{(F_1^u(t_{s(x)+1},...,t_d),...,F_{s(x)}^u(t_{s(x)+1},...,t_d),t_{s(x)+1},...t_d) : |t_i|\leq q\}],$$
where $0<q\leq Q_\epsilon(x)$, and $\overrightarrow{F}^u$ is a $C^{1+\beta/3}$ function s.t. $\max\limits_{\overline{R_q(0)}}|\overrightarrow{F}^u|_\infty\leq Q_\epsilon(x)$.

Similarly we define an {\em $s-$manifold} in $\psi_x$:
$$V^s=\psi_x[\{(t_1,...,t_{s(x)},F_{s(x)+1}^s(t_1,...,t_{s(x)}),...,F_d^s(t_1,...,t_{s(x)})): |t_i|\leq q\}],$$
with the same requirements for $\overrightarrow{F}^s$ and $q$. We will use the superscript ``$u/s$" in statements which apply to both the $s$ case and the $u$ case. The function $\overrightarrow{F}=\overrightarrow{F}^{u/s}$ is called the {\em representing function} of $V^{u/s}$ at $\psi_x$. The parameters of a $u/s$ manifold in $\psi_x$ are: 
\begin{itemize}
\item $\sigma-$parameter: $\sigma(V^{u/s}):=\|d_{\cdot}\overrightarrow{F}\|_{\beta/3}:=\max\limits_{\overline{R_q(0)}}\|d_{\cdot}\overrightarrow{F}\|+\text{H\"ol}_{\beta/3}(d_{\cdot}\overrightarrow{F})$,

where $\text{H\"ol}_{\beta/3}(d_{\cdot}\overrightarrow{F}):=\max\limits_{\vec{t_1},\vec{t_2}\in\overline{R_q(0)}}\{\frac{\|d_{\overrightarrow{t_1}}\overrightarrow{F}-d_{\overrightarrow{t_2}}\overrightarrow{F}\|}{|\overrightarrow{t_1}-\overrightarrow{t_2}|^{\beta/3}}\}$ and $\|A\|:=\sup\limits_{v\neq0}\frac{|Av|_\infty}{|v|_\infty}$.
\item $\gamma-$parameter: $\gamma(V^{u/s}):=\|d_0\overrightarrow{F}\|$
\item $\varphi-$parameter: $\varphi(V^{u/s}):=|\overrightarrow{F}(0)|_\infty$
\item $q-$parameter: $q(V^{u/s}):=q$
\end{itemize}

A {\em $(u/s,\sigma,\gamma,\varphi,q)-$manifold} in $\psi_x$ is a $u/s$ manifold $V^{u/s}$ in $\psi_x$ whose parameters satisfy $\sigma(V^{u/s})\leq\sigma,\gamma(V^{u/s})\leq\gamma,\varphi(V^{u/s})\leq\varphi,q(V^{u/s})\leq q$.
\end{definition}
\noindent\textbf{Remark}: Notice that the dimensions of an $s$ or a $u$ manifold in $\psi_x$ depend on $x$. Their sum is $d$.
\begin{definition}\label{admissible} Suppose $\psi_x^{p^u,p^s}$ is a double chart. A $u/s$-admissible manifold in $\psi_x^{p^u,p^s}$ is a {\em $(u/s,\sigma,\gamma,\varphi,q)-$manifold} in $\psi_x$ s.t.
$$\sigma\leq\frac{1}{2},\gamma\leq\frac{1}{2}(p^u\wedge p^s)^{\beta/3},\varphi\leq10^{-3}(p^u\wedge p^s),\text{ and }q = \begin{cases} p^u & u-\text{manifolds} \\ 
p^s & s-\text{manifolds}. \end{cases}$$
\end{definition}
\noindent\textbf{Remark}: If $\epsilon<1$ (as we always assume), then these conditions together with $p^u,p^s<Q_\epsilon(x)$ force $$\mathrm{Lip}(\overrightarrow{F})=\max\limits_{\vec{t_1},\vec{t_2}\in\overline{R_q(0)}}\frac{|\overrightarrow{F}(\overrightarrow{t_1})-\overrightarrow{F}(\overrightarrow{t_2})|_\infty}{|\overrightarrow{t}_1-\overrightarrow{t}_2|_\infty}<\epsilon,$$ because for every $\overrightarrow{t}_{u/s}$ in the domain of $\overrightarrow{F}$, $|\overrightarrow{t}_{u/s}|_\infty\leq p^{u/s}\leq Q_\epsilon(x)<\epsilon^{3/\beta}$ and $\|d_{\overrightarrow{t}_{u/s}}F\|\leq\|d_0F\|+\text{H\"ol}_{\beta/3}(d_{\cdot}\overrightarrow{F})\cdot|\overrightarrow{t}_{u/s}|_\infty^{\beta/3}\leq\frac{1}{2}(p^u\wedge p^s)^{\beta/3}+\frac{1}{2}(p^{u/s})^{\beta/3}<(p^{u/s})^{\beta/3}<\epsilon$ and by Lagrange's mean-value theorem applied to the restriction of $\vec{F}$ to the interval connecting each $\vec{t}_1$  and $\vec{t}_2$, we are done.
Another important remark: If $\epsilon$ is small enough then $\max\limits_{\overline{R_q(0)}}|\overrightarrow{F}|_\infty<10^{-2}Q_\epsilon(x)$, since $\max\limits_{\overline{R_q(0)}}|\overrightarrow{F}|_\infty\leq|\overrightarrow{F}(0)|_\infty+\mathrm{Lip}(\overrightarrow{F})\cdot p^{u/s}<\varphi+\epsilon p^{u/s}\leq(10^{-3}+\epsilon)p^{u/s}<10^{-2}p^{u/s}$
\begin{definition}\label{def137}$\text{ }$\\\begin{enumerate}
    \item Let $V_1,V_2$ be two $u-$ manifolds (resp. $s-$manifolds) in $\psi_x$ s.t. $q(V_1)=q(V_2)$, then:
$$dist(V_1,V_2):=\max|F_1-F_2|_\infty,dist_{C^1}(V_1,V_2):=\max|F_1-F_2|_\infty+\max\|d_{\cdot}F_1-d_{\cdot}F_2\|$$
\item For a map $E:\mathrm{Dom}\rightarrow M_{r_1}(\mathbb{R})$, where $M_{r_1}(\mathbb{R})$ is the space of real matrices of dimensions $r_1\times r_1$, and $\mathrm{Dom}$ is the closure of some open and bounded subset of $\mathbb{R}^{r_2}$ ($r_1,r_2\in\mathbb{N}$), the {\em $\alpha$-norm} of $E(\cdot)$ is
$$\|E(\cdot)\|_{\alpha}:=\|E(\cdot)\|_\infty+\text{H\"ol}_{\alpha}(E(\cdot))\text{, where, }$$

$\|E(\cdot)\|_\infty=\sup\limits_{s\in\mathrm{Dom}}\|E(s)\|=\sup\limits_{s\in\mathrm{Dom}}\sup\limits_{v\neq 0}\frac{|E(s)v|_\infty}{|v|_\infty}$, and $\text{H\"ol}_{\alpha}(E(\cdot))=\sup\limits_{s_1,s_2 \in\mathrm{Dom}}\{\frac{\|E(s_1)-E(s_2)\|}{|s_1-s_2|_\infty^\alpha}\}$.
\end{enumerate}
\end{definition}
\noindent\textbf{Remark}: \begin{enumerate}
    \item The second part of this definition is a generalization of the specific $\frac{\beta}{3}$-norm of admissible manifolds, as represented by the $\sigma$ parameter in definition \ref{def135}.
    \item  Notice that the difference between $\|\cdot\|$ and $\|\cdot\|_\infty$ in our notation, is that $\|\cdot\|$ is the operator norm, while $\|\cdot\|_\infty$ is the supremum norm for operator valued maps.
    \item $\|\cdot\|_\alpha$ is submultiplicative: for any two  functions $\varphi$ and $\psi$ as in definition \ref{def137}(2),
$$\|\psi\cdot\varphi\|_\alpha\leq\|\varphi\|_\alpha\cdot\|\psi\|_\alpha.
\footnote{To see that, it is enough to show that $\text{H\"ol}_\alpha(\varphi\cdot\psi)\leq \text{H\"ol}_\alpha(\psi)\cdot\|\varphi\|_\infty+\text{H\"ol}_\alpha(\varphi)\cdot\|\psi\|_\infty$. This holds since,
$$\sup\limits_{x,y\in\mathrm{Dom}}\frac{|\varphi(x)\psi(x)-\varphi(y)\psi(y)|}{|x-y|^\alpha}=\sup\limits_{x,y\in\mathrm{Dom}}\frac{|\varphi(x)\psi(x)-\varphi(x)\psi(y)+\varphi(x)\psi(y)-\varphi(y)\psi(y)|}{|x-y|^\alpha}\leq \mathrm{\mathrm{RHS}}.$$} $$
\end{enumerate}
\begin{lemma}\label{banachalgebras}
Let $E:\mathrm{dom}(E)\rightarrow M_r(\mathbb{R})$ be a map as in definition \ref{def137}. Then for any $0<\alpha<1$, if $\|E(\cdot)\|_\alpha<1$, then  $$\|(Id+E(\cdot))^{-1}\|_\alpha\leq\frac{1}{1-\|E(\cdot)\|_\alpha}$$
\end{lemma}
\begin{proof}

In addition to the third item of the remark above, it is also easy to check that $BH:=\{\varphi:dom(E)\rightarrow M_r(\mathbb{R}): \|\varphi\|_\alpha<\infty\}$ is a complete Banach space, and hence a Banach algebra. By a well known lemma in Banach algebras, if $\|E\|_\alpha<1$ then $\|(Id+E(\cdot))^{-1}\|_\alpha\leq\frac{1}{1-\|E(\cdot)\|_\alpha}$.

\end{proof}

\begin{prop}\label{firstbefore} The following holds for all $\epsilon$ small enough. Let $V^u$ be a $u-$admissible manifold in $\psi_x^{p^s,p^u}$, and $V^s$ an $s-$admissible manifold in $\psi_x^{p^s,p^u}$, then:
\begin{enumerate}
\item $V^u$ and $V^s$ intersect at a unique point $P$,
\item $P=\psi_x(\overrightarrow{t})$ with $|\overrightarrow{t}|_\infty\leq10^{-2}(p^s\wedge p^u)$,
\item $P$ is a Lipschitz function of $(V^s,V^u)$, with Lipschitz constant less than 3 (w.r.t the distance $dist(\cdot,\cdot)$ in definition \ref{def137}).
\end{enumerate}
\end{prop}
\begin{proof} Assume $\epsilon\in(0,\frac{1}{2})$. Remark: We will omit the $\overrightarrow{\cdot}$ notation when it is clear that the object under inspection is a vector in $\mathbb{R}^t$ ($t=s(x),u(x)$). Write $V^u=\psi_x[\{(-F(t_u)-,-t_u-): |t_u|_\infty\leq p^u\}]$ and $V^s=\psi_x[\{(-t_s-,-G(t_s)-): |t_s|_\infty\leq p^s\}]$.

Let $\eta:=p^s\wedge p^u$. Note that $\eta<\epsilon$ and that $|F(0)|_\infty,|G(0)|_\infty\leq10^{-3}\eta$ and $\mathrm{Lip}(F),\mathrm{Lip}(G)<\epsilon$ by the remark following definition \ref{admissible}. Hence the maps $F$ and $G$ are contractions, and they map the closed cubes $\overline{R_{10^{-2}\eta}(0)}$ in the respective dimensions into themselves: for every $(H,t)\in\{(F,t_u),(G,t_s)\}$:
$$|H(t)|_\infty\leq|H(0)|_\infty+\mathrm{Lip}(H)|t|_\infty<10^{-3}\eta+\epsilon10^{-2}\eta=(10^{-1}+\epsilon)10^{-2}\eta<10^{-2}\eta.$$
It follows that $G\circ F$ is a $\epsilon^2$ contraction of $\overline{R_{10^{-2}\eta}(0)}$ into itself. By the Banach fixed point theorem it has a unique fixed point: $G\circ F(w)=w$. Denote $v=F(w)$. We claim $V^s,V^u$ intersect at $$P=\psi_x(-F(w)-,-w-)=\psi_x(-v-,-w-)=\psi_x(-v-,-G(v)-),$$
\begin{itemize}
\item $P\in V^u$ since $v=F(w)$ and $|w|_\infty\leq10^{-2}\eta<p^u$.
\item $P\in V^s$ since $w=G\circ F(w)=G(v)$ and $|v|_\infty<|F(0)|_\infty+\mathrm{Lip}(F)|w|_\infty\leq10^{-3}\eta+\epsilon10^{-2}\eta<10^{-2}\eta<p^s$.
\end{itemize}
We also see that $|v|_\infty,|w|_\infty\leq10^{-2}\eta$.

We claim that $P$ is the unique intersection point of $V^s,V^u$. Denote $\xi=p^u \vee p^s$ and extend $F,G$ arbitrarily to $\epsilon-$Lipschitz functions $\widetilde{F},\widetilde{G}:\overline{R_\xi(0)}\rightarrow\overline{R_{Q_\epsilon(x)}(0)}$ 
using McShane's extension formula \cite{McShane}. \footnote{$\widetilde{F}_i(s)=\sup\limits_{u\in dom(F_i)}\{F_i(u)-\mathrm{Lip}(F_i)\cdot |s-u|_\infty\},\widetilde{G}_i(s)=\sup\limits_{u\in dom(G_i)}\{G_i(u)-\mathrm{Lip}(G_i)\cdot |s-u|_\infty\}$ for each coordinate of $F,G$, which in turn induces $\epsilon$-Lipschitz extensions to $F,G$ since the relevant norm is $|\cdot|_\infty$.\label{lipextensions}}
Let $\widetilde{V}^s,\widetilde{V}^u$
denote the $u/s-$sets represented by $\widetilde{F},\widetilde{G}$. Any intersection point of $V^s,V^u$ is an intersection point of $\widetilde{F},\widetilde{G}$, which takes the form $\widetilde{P}=\psi_x(-\widetilde{v}-,-\widetilde{w}-)$ where $\widetilde{v}=\widetilde{F}(\widetilde{w})$ and $\widetilde{w}=\widetilde{G}(\widetilde{v})$. Notice that $\widetilde{w}$ is a fixed point of $\widetilde{G}\circ\widetilde{F}$. The same calculations as before show that $\widetilde{G}\circ\widetilde{F}$ contracts $\overline{R_\xi(0)}$ into itself. Such a map has a unique fixed point--- hence $\widetilde{w}=w$, whence $\widetilde{P}=P$. This concludes parts 1,2.

Next we will show that $P$ is a Lipschitz function of $(V^s,V^u)$. Suppose $V_i^u,V_i^s$ ($i=1,2$) are represented by $F_i$ and $G_i$ respectively. Let $P_i$ denote the intersection points of $V_i^s\cap V_i^u$. We saw above that $P
_i=\psi_x(-v_i-,-w_i-)$, where $w_i$ is a fixed point of $f_i:=G_i\circ F_i:\overline{R_{10^{-2}\eta}(0)}\rightarrow\overline{R_{10^{-2}\eta}(0)}$. Then $f_i$ are $\epsilon^2$ contractions of $\overline{R_{10^{-2}\eta}(0)}$ into itself. Therefore:
$$|w_1-w_2|_\infty=|f_1^n(w_1)-f_2^n(w_2)|_\infty\leq|f_1(f_1^{n-1}(w_1))-f_2(f_1^{n-1}(w_1))|_\infty+|f_2(f_1^{n-1}(w_1))-f_2(f_2^{n-1}(w_2))|_\infty\leq$$
$$\leq\max\limits_{\overline{R_{10^{-2}\eta}(0)}}|f_1-f_2|_\infty+\epsilon^2|f_1^{n-1}(w_1)-f_2^{n-1}(w_2)|_\infty\leq$$
$$\leq\cdots\leq\max\limits_{\overline{R_{10^{-2}\eta}(0)}}|f_1-f_2|_\infty(1+\epsilon^2+...+\epsilon^{2(n-1)})+\epsilon^{2n}|w_1-w_2|_\infty$$
and passing to the limit $n\rightarrow\infty$: $$|w_1-w_2|_\infty\leq(1-\epsilon^2)^{-1}\max\limits_{\overline{R_{10^{-2}\eta}(0)}}|f_1-f_2|_\infty.$$
Similarly $v_i$ is a fixed point of $g_i=F_i\circ G_i$, a contraction from $\overline{R_{10^{-2}\eta}(0)}$ to itself. The same arguments give $|v_1-v_2|_\infty\leq(1-\epsilon^2)^{-1}\max\limits_{\overline{R_{10^{-2}\eta}(0)}}|g_1-g_2|_\infty$.
Since $\psi_x$ is 2-Lipschitz,
$$d(P_1,P_2)<\frac{2}{1-\epsilon^2}(\max\limits_{\overline{R_{10^{-2}\eta}(0)}}|g_1-g_2|_\infty+\max\limits_{\overline{R_{10^{-2}\eta}(0)}}|f_1-f_2|_\infty).$$
Now
$$\max\limits_{\overline{R_{10^{-2}\eta}(0)}}|F_1\circ G_1-F_2\circ G_2|_\infty\leq\max\limits_{\overline{R_{10^{-2}\eta}(0)}}|F_1\circ G_1-F_1\circ G_2|_\infty+\max\limits_{\overline{R_{10^{-2}\eta}(0)}}|F_1\circ G_2-F_2\circ G_2|_\infty\leq$$
$$\leq \mathrm{Lip}(F_1)\max\limits_{\overline{R_{10^{-2}\eta}(0)}}|G_1-G_2|_\infty+\max\limits_{\overline{R_{10^{-2}\eta}(0)}}|F_1-F_2|_\infty$$
and in the same way
$$\max\limits_{\overline{R_{10^{-2}\eta}(0)}}|G_1\circ F_1-G_2\circ F_2|_\infty\leq \mathrm{Lip}(G_1)\max\limits_{\overline{R_{10^{-2}\eta}(0)}}|F_1-F_2|_\infty+\max\limits_{\overline{R_{10^{-2}\eta}(0)}}|G_1-G_2|_\infty$$
So since $\mathrm{Lip}(G),\mathrm{Lip}(F)\leq\epsilon$: $d(P_1,P_2)<\frac{2(1+\epsilon)}{1-\epsilon^2}[dist(V_1^s,V_2^s)+dist(V_1^u,V_2^u)]$.

It remains to choose $\epsilon$ small enough so that $\frac{2(1+\epsilon)}{1-\epsilon^2}<3$.

\end{proof}

The next theorem analyzes the action of $f$ on admissible manifolds. Similar Graph Transform Lemmas were used to show Pesin's stable manifold theorem (\cite{BP} chapter 7, \cite{Pesin}), see also \cite{perron1,Perron2}. Our analysis shows that the Graph Transform preserves admissibility. The general idea is similar to what Sarig does in \cite[\textsection~4.2]{Sarig} for the two dimensional case. Katok and Mendoza treat the general dimensional case in \cite{KM}. Our proof complements theirs, by establishing additional analytic properties.

Recall the notation $\psi_x^{p^s,p^u}\rightarrow\psi_y^{q^s,q^u}$ from definition \ref{edges}.
\begin{theorem}\label{graphtransform}
(Graph Transform) The following holds for all $\epsilon$ small enough: suppose $\psi_x^{p^s,p^u}\rightarrow\psi_y^{q^s,q^u}$, and $V^u$ is a u-admissible manifold in $\psi_x^{p^s,p^u}$, then:
\begin{enumerate}
\item $f[V^u]$ contains a u-manifold $\widehat{V}^u$ in $\psi_y^{q^s,q^u}$ with parameters:
\begin{itemize}
\item $\sigma(\widehat{V}^u)\leq e^{\sqrt{\epsilon}}e^{-2\chi}[\sigma(V^u)+\sqrt{\epsilon}]$
\item $\gamma(\widehat{V}^u)\leq e^{\sqrt{\epsilon}}e^{-2\chi}[\gamma(V^u)+\epsilon^{\beta/3}(q^s\wedge q^u)^{\beta/3}]$
\item $\varphi(\widehat{V}^u)\leq e^{\sqrt{\epsilon}}e^{-\chi}[\varphi+\sqrt{\epsilon}(q^s\wedge q^u)]$
\item $q(\widehat{V}^u)\geq\min\{e^{-\sqrt{\epsilon}}e^\chi q(V^u),Q_\epsilon(y)\}$
\end{itemize}
\item $f[V^u]$ intersects any s-admissible manifold in $\psi_y^{q^s,q^u}$ at a unique point.
\item $\widehat{V}^u$ restricts to a u-manifold in $\psi_y^{q^s,q^u}$. This is the unique u-admissible manifold in $\psi_y^{q^s,q^u}$ inside $f[V^u]$. We call it $\mathcal{F}_u(V^u)$.
\item Suppose $V^u$ is represented by the function $F$. If $p:=(F(0),0)$ then $f(p)\in\mathcal{F}_u(V^u)$.
\end{enumerate}
Similar statements hold for the $f^{-1}$-image of an s-admissible manifold in $\psi_y^{q^s,q^u}$.
\end{theorem}
\begin{proof}
Let $V^u=\psi_x[\{(F_1^u(t_{s(x)+1},...,t_d),...,F_{s(x)}^u(t_{s(x)+1},...,t_d),t_{s(x)+1},...t_d) : |t_i|\leq p^u\}]$ be a $u$-admissible manifold in $\psi_x^{p^s,p^u}$. We omit the $^u$ super-script of $F$ to ease notations, and denote the parameters of $V^u$ by $\sigma,\gamma,\varphi$ and $q$. Let $\eta:=p^s\wedge p^u$. By the admissibility of $V^u$, $\sigma\leq\frac{1}{2},\gamma\leq\frac{1}{2}\eta^{\beta/3},\varphi\leq10^{-3}\eta, q=p^u$, and $\mathrm{Lip}(F)<\epsilon$. We analyze $\Gamma_y^u=\psi_y^{-1}[f[V^u]]\equiv f_{xy}[\mathrm{Graph} (F)]$, where $f_{xy}=\psi_y^{-1}\circ f\circ\psi_x$ and $graph(F):=\{(-F(t)-,-t-): |t|_\infty\leq q\}$. Since $V^u$ is admissible, $graph(F)\subset R_{Q_\epsilon(x)}(0)$. On this domain 
(proposition  \ref{3.4inomris}),
$$f_{xy}(v_s,v_u)=(D_sv_s+h_s(v_s,v_u),D_uv_u+h_u(v_s,v_u)),$$
where $\kappa^{-1}\leq\|D_s^{-1}\|^{-1},\|D_s\|\leq e^{-\chi}$ and $e^\chi\leq\|D_u^{-1}\|^{-1},\|D_u\|\leq\kappa$, $|h_{s/u}(0)|<\epsilon\eta$, $\|\frac{\partial(h_s,h_u)}{\partial(v_s,v_u)}\rvert_0\|<\epsilon\eta^{\beta/3}$, and $\|\frac{\partial(h_s,h_u)}{\partial(v_s,v_u)}\rvert_{v_1}-\frac{\partial(h_s,h_u)}{\partial(v_s,v_u)}\rvert_{v_2}\|\leq\epsilon|v_1-v_2|^{\beta/3}$.
By proposition  \ref{3.4inomris}, $\|\frac{\partial(h_s,h_u)}{\partial(v_s,v_u)}\|<\epsilon\eta^{\beta/3}$ on $R_\eta(0)$, so since $\eta\leq Q_\epsilon(x)\leq\epsilon^{3/\beta}$ we get: \begin{equation}\label{remarkproof}\|\frac{\partial(h_s,h_u)}{\partial(v_s,v_u)}\|<\epsilon^2,\end{equation} and $|h_{s/u}(0)|<\epsilon^2$ on $graph(F)$. Using the equation for $f_{xy}$ we can put $\Gamma_y^u$ in the following form:
\begin{equation}\label{astrix}
    \Gamma_y^u=\{(D_sF(t)+h_s(F(t),t),D_ut+ h_u(F(t),t): |t|_\infty\leq q\}, \text{ } \Big( t=(t_{s(x)+1},...,t_d)\Big).
\end{equation}
The idea (as in \cite{KM} and \cite{Sarig}) is to call the ``$u$" part of the coordinates $\tau$, solve for $t=t(\tau)$, and then substitute the result in the ``$s$" coordinates.\\

\textit{Claim 1}: The following holds for all $\epsilon$ small enough: $D_ut+h_u(F(t),t)=\tau$ has a unique solution $t=t(\tau)$ for all $\tau\in R_{e^{\chi-\sqrt{\epsilon}}q}(0)$, and
\begin{enumerate}[label=(\alph*)]
\item $\mathrm{Lip}(t)<e^{-\chi+\epsilon}$,
\item $|t(0)|_{\infty}<2\epsilon\eta$,
\item The $\frac{\beta}{3}$-norm of $d_{\cdot}t$ is smaller than $e^{-\chi}e^{3\epsilon}$.
\end{enumerate}
\textit{Proof}: Let $\tau(t):=D_ut+h_u(F(t),t)$. For every $|t|_\infty\leq q$ and a unit vector $v$:
\begin{align*}
|d_t\tau v|\geq & |D_u v|-\max\|\frac{\partial(h_s,h_u)}{\partial(v_s,v_u)}\|\cdot(\|d_tF\|+u(x))\geq|D_u v|-\max\|\frac{\partial(h_s,h_u)}{\partial(v_s,v_u)}\|\cdot(\|d_tF\|+d)\\
\geq & |D_u v|-\epsilon^2(\epsilon+d)\text{ }(\because\text{ }\eqref{remarkproof}\text{ })\\
\geq & |D_u v|(1-\epsilon^2(\epsilon+d))(\because |D_u v|\geq\|D_u^{-1}\|^{-1} \geq e^\chi>1)>e^{-\epsilon}\|D_u^{-1} \|^{-1}\text{(provided }\epsilon\text{ is small enough)}
\end{align*}
Since $v$ was arbitrary, it follows that $\tau$ is expanding by a factor of at least $e^{-\epsilon}\|D_u^{-1}\|^{-1}\geq e^{\chi-\epsilon}>1$, whence one-to-one. Since $\tau$ is one-to-one, $\tau^{-1}$ is well-defined on $\tau[\overline{R_{q}(0)}]$. We estimate this set: Since $\tau$ is continuous and $e^{-\epsilon}\|D_u^{-1}\|^{-1}$ expanding (recall, these calculations are done w.r.t the supremum norm, as defined in definition \ref{def135}): $\tau[\overline{R_{q}(0)}]\supset \overline{R_{e^{-\epsilon}\|D_u^{-1}\|^{-1}q}(\tau(0))}$. The center of the box can be estimated as follows:
$$|\tau(0)|_\infty=|h_u(F(0),0)|_\infty\leq |h_u(0,0)|_\infty+\max\|\frac{\partial(h_s,h_u)}{\partial(v_s,v_u)}\||F(0)|_\infty\leq\epsilon\eta+\epsilon^2 10^{-3}\eta\leq2\epsilon\eta$$
Recall that $\eta\leq q$, and therefore $|\tau(0)|_\infty\leq2\epsilon q$, and hence 

$\tau[\overline{R_{q}(0)}]\supset \overline{R_{(e^{-\epsilon}\|D_u^{-1}\|^{-1}-2\epsilon)q}(0)}\supset \overline{R_{\|D_u^{-1}\|^{-1}(e^{-\epsilon}-2\epsilon)q}(0)}\supset \overline{R_{e^\chi e^{-\sqrt{\epsilon}}q}(0)}$ for $\epsilon$ small enough. So $\tau^{-1}$ is well-defined on this domain.

Since $t(\cdot)$ is the inverse of a $\|D_u^{-1}\|^{-1}e^{-\epsilon}$-expanding map, $\mathrm{Lip}(t)\leq e^{\epsilon}\|D_u^{-1}\|<e^{-\chi+\epsilon}$ proving (a). 

\medskip
We saw above that $|\tau(0)|_\infty<2\epsilon\eta$. For all $\epsilon$ small enough, this is significantly smaller than $e^{\chi-\sqrt{\epsilon}}q$, therefore $\tau(0)$ belongs to the domain of $t$. It follows that:
\begin{equation}\label{secondastrix}
    |t(0)|_\infty=|t(0)-t(\tau(0))|_\infty\leq \mathrm{Lip}(t)|\tau(0)|_\infty<e^{-\chi+\epsilon}2\epsilon\eta.
\end{equation}
For all $\epsilon$ small enough this is less than $2\epsilon\eta$, proving (b).

\medskip
Next we will calculate the $\frac{\beta}{3}$-norm of $d_\vartheta t$. We use the following notation: Let

$G(\vartheta):=(F(t(\vartheta)),t(\vartheta))$, then $d_{G(\vartheta)}h_u$ is a $u(x)\times d$ matrix, where $u(x):=d-s(x)$. The matrix $d_{t(\vartheta)}F$ is a $u(x)\times s(x)$ matrix, then let $A'(\vartheta)$:=\begin{tabular}{| l |}
\hline $d_{t(\vartheta)}F $  \\ \hline
$I_{u(x)\times u(x)}$   \\ \hline
\end{tabular}, and $A(\vartheta):=d_{G(\vartheta)}h_u \cdot A'(\vartheta)$; where the notation in the definition of $A'$ means that it is the matrix created by stacking the two matrices $d_{t(\vartheta)}F$ (represented in the standard bases, as implied by the notation) and $I_{u(x)\times u(x)}$.
Using this and the identity $$\vartheta=\tau(t(\vartheta))=D_ut(\vartheta)+h_u(F(t(\vartheta)),t(\vartheta))$$  we get: \begin{equation}\label{banana}
    d_\vartheta t=(D_u+A(\vartheta))^{-1}=(Id+D_u^{-1}A(\vartheta))^{-1}D_u^{-1}.
    \end{equation}
    From lemma \ref{banachalgebras} we see that in order to bound the $\frac{\beta}{3}$-norm of $d_\cdot t$, it is enough to show that the $\frac{\beta}{3}$-norm of $D_u^{-1}A(\cdot)$ is less than 1. The lemma is applicable since $d_\vartheta t$ is a square matrix, as both $t$ and $\vartheta$ in this setup are vectors in $\mathbb{R}^{u(x)}$ (recall, $u(x)=d-s(x)$).
    
Here is the bound of the $\frac{\beta}{3}$-norm of $D_u^{-1}A$:

\begin{align*}
\|A(\vartheta_1)-A(\vartheta_2)\|=&\|d_{G(\vartheta_1)}h_uA'(\vartheta_1)-d_{G(\vartheta_2)}h_uA'(\vartheta_2)\|\\
=&\|d_{G(\vartheta_1)}h_uA'(\vartheta_1)-d_{G(\vartheta_1)}h_uA'(\vartheta_2)+d_{G(\vartheta_1)}h_uA'(\vartheta_2)-d_{G(\vartheta_2)}h_uA'(\vartheta_2)\|\\ \leq&\|d_{G(\vartheta_1)}h_u[A'(\vartheta_1)-A'(\vartheta_2)]\|+\|[d_{G(\vartheta_1)}h_u-d_{G(\vartheta_2)}h_u]A'(\vartheta_2)\|\\
\leq&\Big(\mathrm{Lip}(t)^\frac{\beta}{3}\cdot\|d_{G(\vartheta_1)}h_u\|\cdot \text{H\"ol}_{\beta/3}(d_\cdot F)+\text{H\"ol}_{\beta/3}(d_\cdot h_u)\cdot \mathrm{Lip}(G)^\frac{\beta}{3}\cdot\|d_{t(\vartheta_2)}F\|\Big)|\vartheta_1-\vartheta_2|_\infty^{\beta/3}.
\end{align*}
Recall:
\begin{itemize}
\item $\mathrm{Lip}(F)<\epsilon$
\item $\mathrm{Lip}(t)\leq e^{-\chi+\epsilon}$
\item $\text{H\"ol}_{\beta/3}(d_\cdot h_u)\leq \epsilon$
\item $d_0h_u\leq\epsilon\eta^{\beta/3}$
\end{itemize}
From the identity $G(\vartheta)=(F(t(\vartheta)),t(\vartheta))$ we get that Lip($G)\leq\mathrm{Lip}(F)\cdot\mathrm{Lip}(t)+\mathrm{Lip}(t)$; and from that and the first two items above we deduce that $\mathrm{Lip}(G)\leq 1$, whence $\mathrm{Lip}(G)^\frac{\beta}{3}\leq 1$. We also know $\|d_{t(r)}F\|\leq \mathrm{Lip}(F)\leq (p^u)^\frac{\beta}{3}\leq\epsilon^2$  (by the remark after definition \ref{admissible}). From the last two items we deduce $\|d_\cdot h_u\|_\infty\leq(\epsilon\sqrt{2}q)^{\beta/3}+\epsilon\eta^{\beta/3}$. From the admissibility of the manifold represented by $F$, we get H\"ol$_{\frac{\beta}{3}}(d_\cdot F)\leq \frac{1}{2}$.

So it is clear that for a small enough $\epsilon$, $\text{H\"ol}_{\beta/3}(A(\cdot))<2\epsilon^2$, and hence $\text{H\"ol}_{\beta/3}(D_u^{-1}A(\cdot))<2\epsilon^2$. Also, $\|D_u^{-1}A(0)\|\leq e^{-\chi}\|A(0)\|=e^{-\chi}\|d_{(F(0),t(0))}h_uA'(0)\|$, and by proposition  \ref{3.4inomris} $\|d_{G(\vartheta)}h_u\|\leq\epsilon^2$. Since $\|d_\cdot F\|\leq\epsilon$,
for all small enough $\epsilon$, $\|A'(0)\|\leq\epsilon^2(1+\epsilon)<2\epsilon^2$, and $\|D_u^{-1}A(0)\|\leq e^{-\chi}2\epsilon^4$. Hence $\|A(\vartheta)\|\leq\|A(0)\|+\text{H\"ol}_\frac{\beta}{3}(A(\cdot))\cdot(e^\chi q)^\frac{\beta}{3}\leq2\epsilon^4$.

In total, $\|A(\cdot)\|_\frac{\beta}{3}\leq2\epsilon^4+2\epsilon^2<4\epsilon^2$, and so $\|D_u^{-1}A(\cdot)\|_{\beta/3}<e^{-\chi}4\epsilon^2<1$. Thus we can use lemma \ref{banachalgebras} and \eqref{banana} to get
$$\| d_\cdot t\|_{\beta/3}\leq \frac{1}{1-4\epsilon^2}\cdot\|D_u^{-1}\|\leq e^{8\epsilon^2}\cdot e^{-\chi}\leq e^{3\epsilon}\cdot e^{-\chi},$$
for small enough $\epsilon$. Here we used that $D_u$ is at least $e^\chi$-expanding. This completes claim 1. 

\medskip We now substitute $t=t(\vartheta)$ in \eqref{astrix}, and find that $$\Gamma_y^u\supset\{(H(\vartheta),\vartheta):|\vartheta|_\infty<e^{\chi-\sqrt{\epsilon}}q\},$$ where $\vartheta\in \mathbb{R}^{u(x)}$ (as it is the right set of coordinates of $(H(\vartheta),\vartheta)$), and $H(\vartheta):=D_sF(t(\vartheta))+h_s(F(t(\vartheta)),t(\vartheta))$. Claim 1 guarantees that $H$ is well-defined and $C^{1+\beta/3}$ on $R_{e^{\chi-\sqrt{\epsilon}}}(0)$. We find the parameters of $H$:

\textit{Claim 2}: For all $\epsilon$ small enough, $|H(0)|_\infty<e^{-\chi+\sqrt{\epsilon}}[\varphi+\sqrt{\epsilon}(q^u\wedge q^s)]$, and $|H(0)|_\infty<10^{-3}(q^u\wedge q^s)$.

\textit{Proof}: 
By \eqref{secondastrix}, $|t(0)|_\infty<2\epsilon\eta$. Since $\mathrm{Lip}(F)<\epsilon,|F(0)|_\infty<\varphi$ and $\varphi\leq10^{-3}\eta$: $|F(t(0))|_\infty<\varphi+2\epsilon^2\eta<\eta$ provided $\epsilon$ is small enough. Thus
\begin{align*}
|H(0)|_\infty\leq&\|D_s\||F(t(0))|_\infty+|h_s(F(t(0)),t(0))|_\infty\\ \leq&\|D_s\|(\varphi+2\epsilon^2\eta)+[|h_s(0)|_\infty+\max\|d_\cdot h_s\||(F(t(0)),t(0))|_\infty]\\
\leq&\|D_s\|(\varphi+2\epsilon^2\eta)+[\epsilon\eta+3\epsilon^2\max\{|F(t(0))|_\infty,|t(0)|_\infty\}]\\
\leq&\|D_s\|(\varphi+2\epsilon^2\eta)+[\epsilon\eta+3\epsilon^2\eta](\because |F(t(0))|_\infty<\eta,|t(0)|_\infty<2\epsilon\eta<\eta)\\
\leq&\|D_s\|[\varphi+\eta\kappa(\epsilon+5\epsilon^2)].
\end{align*}
Recalling that $\|D_s\|\leq e^{-\chi}$,  $\eta\equiv p^u\wedge p^s\leq e^\epsilon(q^u\wedge q^s)$ and $\varphi<10^{-3}(p^u\wedge p^s)$:
$|H(0)|_\infty\leq e^{-\chi+\epsilon}[10^{-3}+2\kappa\epsilon](q^u\wedge q^s)$ and for all $\epsilon$ sufficiently small this is less than $10^{-3}(q^u\wedge q^s)$. The claim follows.

\medskip
\textit{Claim 3}: For all $\epsilon$ small enough, $\|d_0H\|<e^{-\chi+\epsilon}[\gamma+\epsilon^{\beta/3}(q^u\wedge q^s)^{\beta/3}]$, and $\|d_0H\|<\frac{1}{2}(q^u\wedge q^s)^{\beta/3}$

\textit{Proof}: $\|d_0H\|\leq\|d_0t\|[\|D_s\|\cdot\|d_{t(0)}F\|+\|d_{(F(t(0)),t(0))}h_s\|\cdot\max\|A'(0)\|]$ and
\begin{itemize}
\item $\|d_0t\|\leq \mathrm{Lip}(t)<e^{-\chi+\epsilon}$ (claim 1).
\item $\|d_{t(0)}F\|<\gamma+\frac{2}{3}\epsilon^{\beta/3}\eta^{\beta/3}$, because $\text{H\"ol}_{\beta/3}(d_\cdot F)\leq\frac{1}{2}$ and therefore by claim 1(b), 

$\|d_{t(0)}F\|<\|d_0F\|+\text{H\"ol}_{\beta/3}(d_\cdot F)|t(0)|_\infty^{\beta/3}<\gamma+\sigma(3\epsilon\eta)^{\beta/3}<\gamma+\frac{1}{2}3^\frac{1}{3}\epsilon^{\beta/3}\eta^{\beta/3}\leq \gamma+\frac{3}{4}\epsilon^{\beta/3}\eta^{\beta/3}$.
\item $\|d_{(F(t(0)),t(0))}h_s\|\leq3\epsilon\eta^{\beta/3}$, because $|F(t(0))|_\infty<\eta$ (proof of claim 2), and $|t(0)|_\infty<2\epsilon\eta$ (claim 1), so by the \text{H\"ol}der regularity of $d_\cdot h_{s/u}$:

$\|d_{(F(t(0)),t(0))}h_s\|\leq \|d_0h_s\|+\epsilon\max\{|F(t(0))|_\infty,|t(0)|_\infty\}^{\beta/3}\leq2\epsilon\eta^{\beta/3}$.
\item $\|A'(0)\|\leq 1+\epsilon<2$.
\end{itemize}

These estimates together give us that:
$$\|d_0H\|<e^{-\chi+\epsilon}\|D_s\|[\gamma+\frac{3}{4}\epsilon^{\beta/3}\eta^{\beta/3}+\|D_s\|^{-1}3\epsilon\eta^{\beta/3}\cdot2]<e^{-2\chi+\epsilon}[\gamma+\eta^{\beta/3}(\frac{3}{4}\epsilon^{\beta/3}+6\kappa\epsilon)]\leq$$
$$\leq e^{-2\chi+\epsilon}[\gamma+e^{\epsilon\beta/3}(q^u\wedge q^s)^{\beta/3})(\frac{3}{4}\epsilon^{\beta/3}+6\kappa\epsilon)].$$
This implies that for all $\epsilon$ small enough $\|d_0H\|<e^{-2\chi+\epsilon}[\gamma+\epsilon^{\beta/3}(q^u\wedge q^s)^{\beta/3}]$, which is stronger than the estimate in the claim.

Since $\gamma\leq\frac{1}{2}(p^u\wedge p^s)^{\beta/3}$ and $(p^u\wedge p^s)\leq e^\epsilon(q^u\wedge q^s)$  (by lemma \ref{lemma131}), we also get that for all $\epsilon$ small enough: $\|d_0H\|\leq\frac{1}{2}(q^u\wedge q^s)^{\beta/3}$, as required. 

\medskip
\textit{Claim 4}: For all $\epsilon$ small enough $\|d_\cdot H\|_{\beta/3}<e^{-2\chi+\sqrt{\epsilon}}[\sigma+\sqrt{\epsilon}]$, and $\|d_\cdot H\|_{\beta/3}<\frac{1}{2}$

\textit{Proof}: By claim 1 and its proof: \begin{itemize}
\item $\|d_\cdot t\|_{\beta/3}\leq \|D_u^{-1}\|e^{3\epsilon}$
\item $\|d_{t(\cdot)} F\|_{\beta/3}\leq \sigma$ and $t$ is a contraction.
\item $\|d_{(F(t),t)} h_s\|_{\beta/3}<\epsilon$
\end{itemize}
We need the following fact: Suppose
$\psi,\varphi:\mathrm{Dom}\rightarrow M_{r_1}(\mathbb{R})$, where Dom of some open and bounded subset of $\mathbb{R}^{r_2}$ ($r_1,r_2\in\mathbb{N}$)- then  $\|\varphi\cdot\psi\|_\alpha\leq\|\varphi\|_\alpha\cdot\|\psi\|_\alpha$ (see item (3) in the remark after definition \ref{def137}).

Recall that $H(\vartheta)=D_sF(t(\vartheta))+h_s(F(t(\vartheta)),t(\vartheta))$, hence $$d_\vartheta H=D_sd_{(t(\vartheta))}Fd_\vartheta t+A(\vartheta).$$
By the
bound for $\|A(\cdot)\|_\frac{\beta}{3}$ in claim 1 (recall $A'(\vartheta)$:=\begin{tabular}{| l |}
\hline $d_{t(\vartheta)}F $  \\ \hline
$I_{u(x)\times u(x)}$   \\ \hline
\end{tabular}, $A(\vartheta):=d_{G(\vartheta)}h_u \cdot A'(\vartheta)$), we get $\|d_\cdot H\|_{\beta/3}\leq \|D_s\|\cdot\sigma\cdot e^{3\epsilon-\chi}+4\epsilon^2$. This and $\|D_s\|\leq e^{-\chi}$ gives: $$\|d_\cdot H\|_{\beta/3}\leq e^{-2\chi+\sqrt{\epsilon}}\cdot[\sigma+\sqrt{\epsilon}].$$ 

If $\epsilon$ is small enough, then since $\sigma\leq\frac{1}{2}$, $\|d_\cdot H\|_{\beta/3}\leq e^{-2\chi+\sqrt{\epsilon}}[\sigma+\sqrt{\epsilon}]\leq\frac{1}{2}$.

\medskip
\textit{Claim 5}: For all $\epsilon$ small enough $\widehat{V}^u:=\{\psi_y((H(\vartheta),\vartheta)):|\vartheta|_\infty\leq\min\{e^{\chi-\sqrt{\epsilon}}q,Q_\epsilon(y)\}\}$ is a $u$-manifold in $\psi_y$. The parameters of $\widehat{V}^u$ are as in part 1 of the statement of the theorem, and $\widehat{V}^u$ contains a $u$-admissible manifold in $\psi_y^{q^u,q^s}$.

\textit{Proof}: To see that $\widehat{V}^u$ is a $u$-manifold in $\psi_y$ we have to check that $H$ is $C^{1+\beta/3}$ and $\|H\|_\infty\leq Q_\epsilon(y)$. Claim 1 shows that $H$ is $C^{1+\beta/3}$. To see $\|H\|_\infty\leq Q_\epsilon(y)$, we first observe that for all $\epsilon$ small enough $\mathrm{Lip}(H)<\epsilon$ because $\|d_\vartheta H\|\leq\|d_0 H\|+\text{H\"ol}_{\beta/3}(d_\cdot H)Q_\epsilon(y)^{\beta/3}\leq\epsilon+\frac{1}{2}\epsilon<\sqrt{\epsilon}$ (for small enough $\epsilon$).

It follows that $\|H\|_\infty\leq|H(0)|_\infty+\sqrt{\epsilon}Q_\epsilon(y)<(10^{-3}+\sqrt{\epsilon})Q_\epsilon(y)<Q_\epsilon(y)$.

Next we claim that $\widehat{V}^u$ contains a $u$-admissible manifold in $\psi_y^{q^s,q^u}$. Since $\psi_x^{p^s,p^u}\rightarrow\psi_y^{q^s,q^u},q^u=\min\{e^\epsilon p^u,Q_\epsilon(y)\}$. Consequently, for every $\epsilon$ small enough $$e^{\chi-\sqrt{\epsilon}}q\equiv e^{\chi-\sqrt{\epsilon}}p^u>e^\epsilon p^u\geq q^u$$ So $\widehat{V}^u$ restricts to a $u$-manifold with $q$-parameter equal to $q^u$. Claims 2-4 guarantee that the parameter bounds for $\widehat{V}^u$ satisfy $u$-admissibility in $\psi_y^{q^s,q^u}$ (see definition \ref{admissible}), and that part 1 of theorem \ref{graphtransform} holds.

\medskip
\textit{Claim 6}: $f[V^u]$ contains exactly one $u$-admissible manifold in $\psi_y^{q^s,q^u}$. This manifold contains $f(p)$ where $p=(F(0),0)$.

\textit{Proof}: The previous claim shows existence. We prove uniqueness. Using the identity $\Gamma_y^u=\{(D_sF(t)+h_s(F(t),t),D_ut+h_u(F(t),t)):|t|_\infty\leq q\}$, any $u$-admissible manifold in $\psi_y^{q^s,q^u}$ which is contained in $f[V^u]$ is a subset of $$\psi_y[\{(D_sF(t)+h_s(F(t),t),D_ut+h_u(F(t),t)):|t|_\infty\leq q, |D_ut+h_u(F(t),t)|_\infty\leq q^u\}]$$ We just at the end of claim 5 that for all $\epsilon$ small enough, $q^u<e^{\chi-\sqrt{\epsilon}}q$. By claim 1 the equation $\tau=D_ut+h_u(F(t),t)$ has a unique solution $t=t(\tau)\in R_q(0)$ for all $\tau\in R_{q^u}(0)$. Our manifold must therefore be equal to $$\psi_y[\{(D_sF(t(\tau))+h_s(F(t(\tau)),t(\tau)),\tau):|\tau|_\infty\leq q^u\}]$$ This is exactly the $u$-admissible manifold that we have constructed above.

Let $\mathcal{F}[V^u]$ denote the unique $u$-admissible manifold in $\psi_y^{q^s,q^u}$ contained in $f[V^u]$. We claim that $f(p)\in\mathcal{F}[V^u]$. By the previous paragraph it is enough to show that the second set  of coordinates (the $u$-part) of $\psi_y^{-1}[\{f(p)\}]$ has infinity norm of less than $q^u$. Call the $u$-part $\tau$, then:\begin{align*}
    |\tau|_\infty=&\text{second set of coordinates of }f_{xy}(F(0),0)=|h_u(F(0),0)|_\infty\\
\leq &|h_u(0,0)|_\infty+\max\|d_\cdot h_u\|\cdot|F(0)|_\infty<\epsilon\eta+3\epsilon^2\cdot 10^{-3}\eta<e^{-\epsilon}\eta<(q^s\wedge q^u)\leq q^u.
\end{align*}
This concludes the proof of the claim.

\medskip
\textit{Claim 7}: $f[V^u]$ intersects every $s$-admissible manifold in $\psi_y^{q^s,q^u}$ at a unique point.

\textit{Proof}: 
 In this part we use the fact that $\psi_x^{p^s,p^u}\rightarrow\psi_y^{q^s,q^u}$ implies $s(x)=s(y)$. Let $W^s$ be an $s$-admissible manifold in $\psi_y^{q^s,q^u}$. We saw in the previous claim that $f[V^u]$ and $W^s$ intersect at least at one point. Now we wish to show that the intersection point is unique: Recall that we can put $f[V^u]$ in the form $$f[V^u]=\psi_y[\{(D_sF(t)+h_s(F(t),t),D_ut+h_u(F(t),t)):|t|_\infty\leq q\}].$$ We saw in the proof of claim 1 that the second coordinate $\tau(t):=D_ut+h_u(F(t),t)$ is a one-to-one continuous map whose image contains $R_{q^u}(0)$. We also saw that $\|d_\cdot t^{-1}\|^{-1}>e^{-\epsilon}\|D_u^{-1}\|^{-1}\geq e^{\chi-\epsilon}>1$. Consequently the inverse function $t:Im(\tau)\rightarrow \overline{R_q(0)}$ satisfies $\|d_\cdot t\|<1$, and so $$f[V^u]=\psi_y[\{(H(\vartheta),\vartheta):\vartheta\in Im(\tau)\}],\text{ where }\mathrm{Lip}(H)\leq\epsilon.$$ Let $I:\overline{R_{q^u}(0)}\rightarrow\mathbb{R}^{u(x)}$ denote the function which represents $W^s$ in $\psi_y$, then $\mathrm{Lip}(I)\leq\epsilon$. Extend it to an $\epsilon$-Lipschitz function on $Im(\tau)$ (by McShane's  extension formula \cite{McShane}, similarly to footnote \footref{lipextensions}). The extension represents a Lipschitz manifold $\widetilde{W}^s\supset W^s$. We wish to use the same arguments we used in proposition \ref{firstbefore} to show the uniqueness of the intersection point (this time of $f[V^u]$ and $\widetilde{W}^s$). We need the following observations:
\begin{enumerate}
 \item $\overline{R_{q^u}(0)}\subset \overline{R_{e^{\chi-\sqrt{\epsilon}}q}(0)}\subset Im(\tau)$, as seen in the estimation of $Im(\tau)$ in the beginning of claim 1.
 \item $Im(\tau)$ is compact, as a continuous image of a compact set.
 \item $\tau(t)=D_ut+h_u(F(t),t)$, hence $\|d_\cdot \tau\|\leq \kappa+\epsilon$. Hence $\mathrm{diam}(Im(\tau))\leq q(1+\kappa)$. So for $\epsilon<\frac{1}{\kappa+1}\wedge\frac{1}{3}$, we get that for $I\circ H$ (an $\epsilon^2$-Lipschitz function from $Im(\tau)$ to itself), its image is contained in $\overline{R_{\epsilon^2\|d_\cdot \tau\|q+|I\circ H(\tau(0))|}(0)}\subset \overline{R_q(0)}\subset Im(\tau)$ (the last inclusion is due to the bound we have seen $|\tau(0)|\leq2\epsilon\eta$)
\end{enumerate}
As in the proof of proposition \ref{firstbefore}, we can use the Banach fixed point theorem 
This gives us that $f[V^u]$ and $W^s$ intersect in at most one point.

This concludes the proof of claim 7, and thus the proof of the theorem for $u$-manifolds. 

The case of $s$-manifolds follows from the symmetry between $s$ and $u$-manifolds:

\begin{enumerate}
\item $V$ is a $u$-admissible manifold w.r.t $f$ iff $V$ is an $s$-admissible manifold w.r.t to $f^{-1}$, and the parameters are the same.
\item $\psi_x^{p^s,p^u}\rightarrow\psi_y^{q^s,q^u}$ w.r.t $f$ iff $\psi_y^{q^s,q^u}\rightarrow\psi_x^{p^s,p^u}$ w.r.t $f^{-1}$. 
\end{enumerate}
\end{proof}

\begin{definition}
Suppose $\psi_x^{p^s,p^u}\rightarrow\psi_y^{q^s,q^u}$.
\begin{itemize}
    \item  The {\em Graph Transform} (of an $s$-admissible manifold) $\mathcal{F}_{s}$ maps an $s$-admissible manifold $V^{s}$ in $\psi_y^{q^s,q^u}$ to the unique $s$-admissible manifold in $\psi_x^{p^s,p^u}$ contained in $f^{-1}[V^{s}]$.
    \item The {\em Graph Transform} (of a $u$-admissible manifold) $\mathcal{F}_{u}$ maps a $u$-admissible manifold $V^{u}$ in $\psi_x^{p^s,p^u}$ to the unique $u$-admissible manifold in $\psi_y^{q^s,q^u}$ contained in $f[V^{u}]$.
\end{itemize}
\end{definition}

The operators $\mathcal{F}_{s},\mathcal{F}_{u}$ depend on the edge $\psi_x^{p^s,p^u}\rightarrow\psi_y^{q^s,q^u}$, even though the notation does not specify it.

\begin{prop}\label{prop133}
If $\epsilon$ is small enough then the following holds: For any $s/u$-admissible manifold $V_1^{s/u},V_2^{s/u}$ in $\psi_x^{p^s,p^u}$
$$dist(\mathcal{F}_{s/u}(V_1^{s/u}),\mathcal{F}_{s/u}(V_2^{s/u}))\leq e^{-\chi/2}dist(V_1^{s/u},V_2^{s/u}),$$
$$dist_{C^1}(\mathcal{F}_{s/u}(V_1^{s/u}),\mathcal{F}_{s/u}(V_2^{s/u}))\leq e^{-\chi/2}[dist_{C^1}(V_1^{s/u},V_2^{s/u})+dist(V_1^{s/u},V_2^{s/u})^{\beta/3}].$$
\end{prop}
\begin{proof} (For the $u$ case) Suppose $\psi_x^{p^s,p^u}\rightarrow\psi_y^{q^s,q^u}$, and let $V_i^u$ be two $u$-admissible manifolds in $\psi_x^{p^s,p^u}$. We take $\epsilon$ to be small enough for the arguments in the proof of theorem \ref{graphtransform} to work. These arguments give us  that if $V_i=\psi_x[\{(F_i(t),t):|t|_\infty\leq p^u\}]$, then $\mathcal{F}_u[V_i]=\psi_y[\{(H_i(\tau),\tau):|\tau|_\infty\leq q^u\}]$, where $t_i$ and $\tau$ are of length $u(x)$ (as a right set of components) and
\begin{itemize}
\item $H_i(\tau)=D_sF_i(t_i(\tau))+h_s(F_i(t_i(\tau)),t_i(\tau))$,
\item $t_i(\tau)$ is defined implicitly by $D_ut_i(\tau)+h_u(F_i(t_i(\tau)),t_i(\tau))=\tau$ and $|d_\cdot t_i|_\infty<1$,
\item $\|D_s^{-1}\|\leq\kappa,\|D_s\|\leq e^{-\chi},\|D_u^{-1}\|\leq e^{-\chi},\|D_u\|\leq \kappa$,
\item $|h_{s/u}(0)|_\infty<\epsilon(p^s\wedge p^u), \text{H\"ol}_{\beta/3}(d_\cdot h_{s/u})\leq\epsilon,max\|d_\cdot h_{s/u}\|<3\epsilon^2$.
\end{itemize}
In order to prove the proposition we need to estimate $\|H_1-H_2\|_\infty,\|d_\cdot H_1-d_\cdot H_2\|_\infty$ in terms of $\|F_1-F_2\|_\infty,\|d_\cdot F_1-d_\cdot F_2\|_\infty$.

\medskip
\textit{Part 1:} For all $\epsilon$ small enough $\|t_1-t_2\|_\infty\leq\epsilon\|F_1-F_2\|_\infty$

\textit{Proof}: By definition, $D_ut_i(\tau)+h_u(F_i(t_i(\tau)),t_i(\tau))=\tau$. By differentiating both sides we get

$|D_u(t_1-t_2)|_\infty\leq|h_u(F_1(t_1),t_1)-h_u(F_2(t_2),t_2)|_\infty\leq\max\|d_\cdot h_u\|\cdot\max\{|F_1(t_1)-F_2(t_2)|_\infty,|t_1-t_2|_\infty\}\leq$ 

$\leq\max\|d_\cdot h_u\|\cdot\max\{|F_1(t_1)-F_1(t_2)|_\infty+|F_1(t_2)-F_2(t_2)|_\infty,|t_1-t_2|_\infty\}\leq$

$\leq3\epsilon^2\max\{\|F_1-F_2\|_\infty+\mathrm{Lip}(F_1)|t_1-t_2|_\infty,|t_1-t_2|_\infty\}\leq3\epsilon^2\|F_1-F_2\|_\infty+3\epsilon^2(\epsilon+1)|t_1-t_2|_\infty$

Rearranging terms and recalling $\|D_u^{-1}\|\leq e^{-\chi}$: $$\|t_1-t_2\|_\infty\leq\frac{3\epsilon^2\|F_1-F_2\|_\infty}{e^{\chi-\epsilon}-3\epsilon^2(1+\epsilon)}$$ The claim follows.

\medskip
\textit{Part 2}: For all $\epsilon$ small enough $\|H_1-H_2\|_\infty<e^{-\chi/2}\|F_1-F_2\|_\infty$, whence proving the first claim of the proposition.

\textit{Proof}: $|H_1-H_2|_\infty\leq\|D_s\|\cdot|F_1(t_1)-F_2(t_2)|_\infty+|h_s(F_1(t_1),t_1)-h_s(F_2(t_2),t_2)|_\infty\leq$

$\leq\|D_s\|\cdot|F_1(t_1)-F_2(t_2)|_\infty+\|d_\cdot h_s\|_\infty\max\{|F_1(t_1)-F_2(t_2)|_\infty,|t_1-t_2|_\infty\}\leq$

$\leq\|D_s\|\cdot(|F_1(t_1)-F_2(t_1)|_\infty+|F_2(t_1)-F_2(t_2)|_\infty)+$

$+\|d_\cdot h_s\|_\infty\cdot\max\{|F_1(t_1)-F_2(t_1)|_\infty+|F_2(t_1)-F_2(t_2)|_\infty,|t_1-t_2|_\infty\}\leq$

$\leq\|D_s\|(\mathrm{Lip}(F_1)+\mathrm{Lip}(F_2))|t_1-t_2|_\infty+\|d_\cdot h_s\|_\infty\max\{\mathrm{Lip}(F_1)+\mathrm{Lip}(F_2)\}|t_1-t_2|_\infty\leq$

$\leq\Big((\|D_s\|+\|d_\cdot h_s\|_\infty)(\mathrm{Lip}(F_1)+\mathrm{Lip}(F_2))+\|d_\cdot h_s\|_\infty\Big)|t_1-t_2|_\infty\leq$

$\leq\Big((e^{-\chi}+3\epsilon^2)2\epsilon+3\epsilon^2\Big)|t_1-t_2|_\infty\leq\Big((e^{-\chi}+3\epsilon^2)2\epsilon+3\epsilon^2\Big)\frac{3\epsilon^2\|F_1-F_2\|_\infty}{e^{\chi-\epsilon}-3\epsilon^2(1+\epsilon)}(\because\text{ part 1})$

Now when $\epsilon
\rightarrow0$ the RHS multiplier goes to $0$, hence for small enough $\epsilon$ it is smaller than $e^{-\chi/2}$.

\medskip
\textit{Part 3}: For all $\epsilon$ small enough: $$\|d_\cdot t_1-d_\cdot t_2\|_\infty<\sqrt{\epsilon}(\|d_\cdot F_1-d_\cdot F_2\|_\infty+\|F_1-F_2\|_\infty^{\beta/3})$$

\textit{Proof}: Define as in claim 1 of theorem \ref{graphtransform}, $A'_i(\vartheta)$:=\begin{tabular}{| l |}
\hline $d_{t_i(\vartheta)}F_i $  \\ \hline
$I_{u(x)\times u(x)}$   \\ \hline
\end{tabular} and $A_i(\vartheta):=d_{G_i(\vartheta)}h_u \cdot A_i'(\vartheta)$, where  $G_i(\vartheta)=(F_i(t_i(\vartheta)),t_i(\vartheta))$.

We have seen at the beginning of the Graph Transform that

$Id=(D_u+A_i(\vartheta))d_\vartheta t_i$. By taking differences we obtain:

$(D_u+A_1(\vartheta))(d_\vartheta t_1-d_\vartheta t_2)=(D_u+A_1(\vartheta))d_\vartheta t_1-(D_u+A_2)d_\vartheta t_2+(A_2-A_1)d_\vartheta t_2=Id-Id+(A_2-A_1)d_\vartheta t_2=$ 

$=(A_2-A_1)d_\vartheta t_2=(d_{G_1}h_u\cdot A'_1-d_{G_2}h_uA'_2)\cdot d_\vartheta t_2$

Define $J_i(t):=$\begin{tabular}{| r |}
\hline $d_{t}F_i $  \\ \hline
$I_{u(x)\times u(x)}$   \\ \hline
\end{tabular}, which is the same as $A'_i(\vartheta)$, with a parameter change ($t=t_i(\vartheta)$). So now we see that:
\begin{align*}
(D_u+A_1)(d_\vartheta t_1-d_\vartheta t_2)=&(d_{G_1}h_u\cdot J_1(t_1)-d_{G_2}h_u\cdot J_2(t_2))d_\vartheta t_2\\
=&(d_{G_1}h_u\cdot J_1(t_1)-d_{G_2}h_u\cdot J_1(t_1)+d_{G_2}h_u\cdot J_1(t_1)-d_{G_2}h_u\cdot J_2(t_2))d_\vartheta t_2
\\=&\Big([d_{G_1}h_u\cdot J_1(t_1)-d_{G_2}h_u\cdot J_1(t_1)]+[d_{G_2}h_u\cdot J_1(t_1)-d_{G_2}h_u\cdot J_2(t_1)]\\
+&[d_{G_2}h_u\cdot J_2(t_1)-d_{G_2}h_u\cdot J_2(t_2)]\Big)d_\vartheta t_2:=I+II+III.
\end{align*}
Since $\|D_u^{-1}\|^{-1}\geq e^\chi,\|d_\cdot F_1\|_\infty<1$ and $\|d_\cdot h_u\|_\infty<3\epsilon^2$, we get $\|D_u+A_1\|\geq e^\chi-3\epsilon^2$, and hence:
$$\|d_\cdot t_1-d_\cdot t_2\|_\infty\leq\frac{1}{e^\chi-3\epsilon^2}\|I+II+III\|_\infty\leq \|I\|_\infty+\|II\|_\infty+\|III\|_\infty (\text{for }\epsilon\text{ small enough})$$

We begin analyzing: recall $\text{H\"ol}_{\beta/3}(d_\cdot h_{s/u})\leq\epsilon$:
\begin{itemize}
\item $\|d_{(F_1(t_1),t_1)}h_u-d_{(F_2(t_1),t_1)}h_u\|\leq\epsilon\|F_1-F_2\|_\infty^{\beta/3}$
\item $\|d_{(F_2(t_1),t_1)}h_u-d_{(F_2(t_2),t_1)}h_u\|\leq\epsilon\|t_1-t_2\|_\infty^{\beta/3}(\because \mathrm{Lip}(F_2)<1)$ 
\item $\|d_{(F_2(t_2),t_1)} h_u-d_{(F_2(t_2),t_2)} h_u\|_\infty<3\epsilon\|F_1-F_2\|_\infty^{\beta/3}$
\end{itemize}
By part 1 $\|t_1-t_2\|_\infty\leq\epsilon\|F_1-F_2\|_\infty^{\beta/3}$, it follows that:
$$\|d_{G_1}h_u-d_{G_2}h_u\|<3\epsilon\|F_1-F_2\|_\infty^{\beta/3}$$ and hence $$\|I\|_\infty\leq3\epsilon\|F_1-F_2\|_\infty^{\beta/3}$$
Using the facts that $\|d_\cdot t_{1/2}\|_\infty<1,\|d_\cdot h_u\|_\infty<3\epsilon^2,\mathrm{Lip}(d_\cdot F_2)<1$ and $\text{H\"ol}_{\beta/3}(d_\cdot F_2)<1$ (from the definition of admissible manifolds and the proof of the Graph Transform) we get that:

\hspace{16pt}$\|II\|_\infty\leq3\epsilon^2\|d_\cdot F_1-d_\cdot F_2\|_\infty$,

\hspace{16pt}$\|III\|_\infty\leq3\epsilon^2\|t_1-t_2\|_\infty\leq3\epsilon^2\|F_1-F_2\|_\infty\leq3\epsilon^2\|F_1-F_2\|_\infty^{\beta/3}$.

So for all $\epsilon$ sufficiently small $\|d_\cdot t_1-d_\cdot t_2\|_\infty<\sqrt{\epsilon}(\|d_\cdot F_1-d_\cdot F_2\|_\infty+\|F_1-F_2\|_\infty^{\beta/3})$

\medskip
\textit{Part 4}: $\|d_\cdot H_1-d_\cdot H_2\|_\infty<e^{-\chi/2}(\|d_\cdot F_1-d_\cdot F_2\|_\infty+\|F_1-F_2\|_\infty^{\beta/3})$

\textit{Proof}: By the definition of $H_i$: $d_\vartheta H_i=D_s\cdot d_{t_i} F_i\cdot d_\vartheta t_i+d_{G_i}h_u\cdot J_i(t_i)\cdot d_\vartheta t_i$

Taking differences we see that: $$\|d_\vartheta H_1-d_\vartheta H_2\|\leq\|d_\vartheta t_1-d_\vartheta t_2\|\cdot\|D_\vartheta d_{t_1}F_1+d_{G_1}h_u\|+\|d_\vartheta t_2\|\cdot\|D_s\|\cdot(\|d_{t_1}F_1-d_{t_1}F_2\|+\|d_{t_1}F_2-d_{t_2}F_2\|)$$
$$+\|d_\vartheta t_2\|\cdot\|d_{G_1}h_u-d_{G_2}h_u\|\cdot\|d_{t_1}F_1\|+\|d_\vartheta t_2\|\cdot\|d_{G_2}h_u\|\cdot\|d_{t_1}F_1-d_{t_2}F_2\|:=I'+II'+III'+IV'$$
Using the same arguments as in part 3 we can show that:
\begin{align*}
I'\leq&\|d_\cdot t_1-d_\cdot t_2\|(e^{-\chi}+6\epsilon^2)<\sqrt{\epsilon}(\|d_\cdot F_1-d_\cdot F_2\|_\infty+\|F_1-F_2\|_\infty^{\beta/3)}\\
II'\leq& e^{-\chi}(\|d_\cdot F_1-d_\cdot F_2\|_\infty+\|t_1-t_2\|_\infty^{\beta/3})\leq e^{-\chi}(\|d_\cdot F_1-d_\cdot F_2\|_\infty+\|F_1-F_2\|_\infty^{\beta/3})\text{ because of part 1}\\
III'\leq&3\epsilon\|F_1-F_2\|_\infty^{\beta/3}\text{ (same as in the estimate of I in part 3)}\\
IV'\leq&3\epsilon^2\|d_\cdot F_1-d_\cdot F_2\|_\infty+3\epsilon^2\|F-1-F_2\|_\infty^{\beta/3}.
\end{align*}
It follows that $$\|d_\cdot H_1-d_\cdot H_2\|_\infty<(e^{-\chi}+10\epsilon+\sqrt{\epsilon})(\|d_\cdot F_1-d_\cdot F_2\|_\infty+\|F_1-F_2\|_\infty^{\beta/3}).$$
If $\epsilon$ is small enough then $$\|d_\cdot H_1-d_\cdot H_2\|_\infty<e^{-\chi/2}(\|d_\cdot F_1-d_\cdot F_2\|_\infty+\|F_1-F_2\|_\infty^{\beta/3})$$
\end{proof}

\subsubsection{A Markov extension}
Recall the definitions from the begining of \textsection \ref{defepsilonchains}:

$$\mathcal{V}=\{\psi_x^{p^s,p^u}:\psi_x^{p^s\wedge p^u}\in\mathcal{A},p^s,p^u\in I_\epsilon,p^s,p^u\leq Q_\epsilon(x)\}$$
$$\mathcal{E}=\{(\psi_x^{p^s,p^u},\psi_y^{q^s,q^u})\in\mathcal{V}\times\mathcal{V}:\psi_x^{p^s,p^u}\rightarrow\psi_y^{q^s,q^u}\}$$
$$\mathcal{G}\text{ is the directed graph with vertices }\mathcal{V}\text{ and edges }\mathcal{E}$$
\begin{definition}
    $\Sigma(\mathcal{G}):=\{v=(v_i)_{i\in\mathbb{Z}}:v_i\in\mathcal{V},v_i\rightarrow v_{i+1}\forall i\}$, equipped with the left-shift $\sigma$, and the metric $d(v,w)=e^{-\min\{|i|:v_i\neq w_i\}}$.
\end{definition}
\begin{definition} Suppose $(v_i)_{i\in\mathbb{Z}},v_i=\psi_{x_i}^{p_i^s,p_i^u}$ is a chain. Let $V_{-n}^u$ be a $u$-admissible manifold in $v_{-n}$. 
$\mathcal{F}_u^n(v_{-n})$ is the $u$-admissible manifold in $v_0$ which is a result of the application of the graph transform $n$ times along $v_{-n},...,v_{-1}$ (each application is the transform described in section 3 of theorem \ref{graphtransform}). Similarly any $s$-admissible manifold in $v_n$ is mapped by $n$ applications of the graph transform to an $s$-admissible manifold in $v_0$: $\mathcal{F}_s^n(v_n)$. These two manifolds depend on $v_{-n},...,v_n$.
\end{definition}
\begin{definition} Let $\psi_x$ be some chart. Assume $V_n$ is a sequence of $s/u$ manifolds in $\psi_x$. We say that $V_n$ converges uniformly to $V$, an $s/u$ manifold in $\psi_x$, if the representing functions of $V_n$ converge uniformly to the representing function of $V$.
\end{definition}
\begin{prop}\label{firstofchapter}
Suppose $(v_i=\psi_{x_i}^{p_i^s,p_i^u})_{i\in\mathbb{Z}}$ is a chain of double charts, and choose arbitrary $u$-admissible manifolds $V^u_{-n}$ in $v_{-n}$, and arbitrary $s$-admissible manifolds $V_n^s$ in $v_n$. Then:
\begin{enumerate}
\item The limits $V^u((v_i)_{i\leq0}):=\lim\limits_{n\rightarrow\infty}\mathcal{F}_u^n(V^u_{-n}),V^s((v_i)_{i\geq0}):=\lim\limits_{n\rightarrow\infty}\mathcal{F}_s^n(V^s_n)$ exist, and are independent of the choice of $V_{-n}^u$ and $V_n^s$.
\item $V^{u}((v_i)_{i\leq0})/V^{s}((v_i)_{i\geq0})$ is an $s/u$-admissible manifold in $v_0$.
\item $f[V^s(v_i)_{i\geq0}]\subset V^s(v_{i+1})_{i\geq0}$ and $f^{-1}[V^u(v_i)_{i\leq0}]\subset V^u(v_{i-1})_{i\leq0}$
\item \begin{align*}
V&^s((v_i)_{i\geq0})=\{p\in\psi_{x_0}[R_{p_0^s}(0)]:\forall k\geq0, f^k(p)\in\psi_{x_k}[R_{10Q_\epsilon(x_k)}(0)]\}\\
V&^u((v_i)_{i\leq0})=\{p\in\psi_{x_0}[R_{p_0^u}(0)]:\forall k\geq0, f^{-k}(p)\in\psi_{x_{-k}}[R_{10Q_\epsilon(x_{-k})}(0)]\}
\end{align*}
\item The maps $(v_i)_{i\in\mathbb{Z}}\mapsto V^u((v_i)_{i\leq0}),V^s((v_i)_{i\geq0})$ are H\"older continuous:
$$\exists K>0, \theta\in(0,1):\forall n\geq0,\vec{v},\vec{w}\text{ chains: if }v_i=w_i\text{ }\forall|i|\leq n\text{ then:}$$
$$dist_{C^1}(V^u((v_i)_{i\leq0}),V^u((w_i)_{i\leq0}))<K\theta^n,dist_{C^1}(V^s((v_i)_{i\geq0}),V^s((w_i)_{i\geq0}))<K\theta^n$$
\end{enumerate}
\end{prop}
\begin{proof} Parts (1)--(4) are a version of Pesin's Stable Manifold Theorem \cite{Pesin}
. Part (5) is a version of Brin's Theorem on the H\"older continuity of the Oseledets distribution \cite{Bri}. Whereas Brin's theorem only states H\"older continuity on Pesin sets, part (5) gives H\"older continuity everywhere, as in Sarig's version--- but with respect to a different metric (symbolic, not Riemannian).

The proofs of parts (1), (3), (4) and (5) are quite similar to the proofs of \cite[proposition~4.15]{Sarig}, we nevertheless present the proof since the understanding of this proposition is important to the ideas presented in this work. 

We treat the case of $u$ manifolds, the stable case is similar.

\textit{Part 1}: By the Graph Transform theorem $\mathcal{F}_u^n(V_{-n}^u)$ is a $u$-admissible manifold in $v_0$. By the previous proposition for any other choice of $u$-admissible manifolds $W_{-n}^u$ in $v_{-n}$:
$$dist(\mathcal{F}_u^n(V_{-n}^u),\mathcal{F}_u^n(W_{-n}^u))<e^{-\frac{1}{2}\chi n}dist(V_{-n}^u,W_{-n}^u)\leq e^{-\frac{1}{2}\chi n}2Q_\epsilon(x)<e^{-\frac{1}{2}\chi n},$$
(here we need the fact that for any admissible manifold with representing function $F$, $\|F\|_\infty\leq Q_\epsilon(x)$).

Thus, if the limit exists then it is independent of $V_{-n}^u$.

For every $n\geq m$: $$dist(\mathcal{F}_u^n(V^u_{-n}),\mathcal{F}_u^m(V^u_{-m}))\leq\sum\limits_{k=0}^{n-m}dist(\mathcal{F}_u^{m+k}(V^u_{-m-k}),\mathcal{F}_u^{m+k+1}(V^u_{-m-k-1}))\leq$$
$$\leq\sum_{k=0}^{n-m}dist(\mathcal{F}_u^m(V_{-m}^u),\mathcal{F}_u^{m+1}(V_{-m-1}^u))e^{-\frac{1}{2}\chi k}<\sum_{k=0}^{n-m}e^{-\frac{1}{2}\chi m}dist(V_{0}^u,\mathcal{F}_u(V_{-1}^u))e^{-\frac{1}{2}\chi k}\leq$$
$$\leq2e^{-m\frac{\chi}{2}}\frac{1}{1-e^{-\chi/2}}\underset{m,n\rightarrow\infty}{\longrightarrow}0 \Big(\because dist(V_0,\mathcal{F}_u(V_{-1}^u))\leq2Q_\epsilon(x)<2\Big).$$
Hence, $\mathcal{F}_u^n(V_{-n}^u)$ is a Cauchy sequence in a complete space, and therefore converges.

\medskip
\textit{Part 2}: Admissibility of the limit: Write $v_0=\psi_x^{p^s,p^u}$ and let $F_n$ denote the functions which represent $\mathcal{F}_u^n(V^u_{-n})$ in $v_0$. Since $\mathcal{F}_u^n(V^u_{-n})$ are $u$-admissible in $v_0$, for every $n$:

\begin{itemize}
\item $\|d_\cdot F_n\|_{\beta/3}\leq\frac{1}{2}$
\item $\|d_0F_n\|\leq\frac{1}{2}(p^s\wedge p^u)^{\beta/3}$
\item $|F_n(0)|_\infty\leq10^{-3}(p^s\wedge p^u)$
\end{itemize}
Since $\mathcal{F}_u^n(V_{-n}^u)\xrightarrow[n\rightarrow\infty]{}, V^u((v_i)_{i\leq0})$: $F_n\xrightarrow[n\rightarrow\infty]{}F$ uniformly, where $F$ represents $V^u((v_i)_{i\leq0})$. By the Arzela-Ascoli theorem, $\exists n_k\uparrow\infty$ s.t. $d_\cdot F_{n_k}\xrightarrow[k\rightarrow\infty]{}L$ uniformly where $\|L\|_{\beta/3}\leq\frac{1}{2}$. For each term in the sequences, the following identity holds:$$\forall t\in R_{q^u}(0):\text{ }\vec{F}_{n_k}(t)=\vec{F}_{n_k}(0)+\int\limits_{0}^{1}d_{\lambda t}\vec{F}_{n_k}\cdot t\text{ }d\lambda.$$

Since $d_\cdot F_{n_k}$ converge uniformly, we get:
$$\forall t\in R_{q^u}(0):\text{ }\vec{F}(t)=\vec{F}(0)+\int\limits_{0}^{1}L(\lambda t)\cdot t\text{ }d\lambda.$$
The same calculations give us that $\forall t_0\in R_{q^u}(0):$
$$\forall t\in R_{q^u}(0):\text{ }\vec{F}(t)=\vec{F}(t_0)+\int\limits_{0}^{1}L(\lambda (t-t_0)+t_0)\cdot (t-t_0)\text{ }d\lambda=$$
$$=\vec{F}(t_0)+\int\limits_{0}^{1}[L(\lambda (t-t_0)+t_0)-L(t_0)]\cdot (t-t_0)\text{ }d\lambda+L(t_0)\cdot(t-t_0).$$
Since $L$ is $\beta/3$-H\"older on a compact set, and in particular uniformly continuous, the second and third summands are $o(|t-t_0|)$, and hence $F$ is differentiable and $ d_tF=L(t)$. We also see that $\{d_\cdot F_n\}$ can only have one limit point. Consequently $d_\cdot F_n\xrightarrow[n\rightarrow\infty]{uniformly}d_\cdot F$ . So $\|d_\cdot F\|_{\beta/3}\leq\frac{1}{2},\|d_0F\|\leq\frac{1}{2}(p^s\wedge p^u)^{\beta/3}$ and $|F(0)|_\infty\leq10^{-3}(p^s\wedge p^u)$, whence the $u$-admissibility of $V^u((v_i)_{i\leq0})$.

\medskip
\textit{Part 3:} Let $V^u:=V^u((v_i)_{i\leq0})=\lim\mathcal{F}_u^n(V^u_{-n})$, and $W^u:=V^u((v_{i-1})_{i\leq0})=\lim\mathcal{F}_u^n(V^u_{-n-1})$.
\begin{align*}dist(V^u,\mathcal{F}_u(W^u))\leq& dist(V^u,\mathcal{F}_u^n(V_{-n}^u))+dist(\mathcal{F}_u^n(V_{-n}^u),\mathcal{F}_u^{n+1}(V_{-n-1}^u))+dist(\mathcal{F}_u^{n+1}(V_{-n-1}^u),\mathcal{F}_u(W^u))\\
\leq &dist(V^u,\mathcal{F}_u^n(V_{-n}^u))+e^{-\frac{1}{2}n\chi}dist(V^u_{-n},\mathcal{F}_u(V^u_{-n-1}))+e^{-\frac{1}{2}\chi}dist(\mathcal{F}_u^{n}(V_{-n-1}^u),W^u).\end{align*}
The first and third summands tend to zero by definition, and the second goes to zero since

$dist(V^u_{-n},\mathcal{F}_u(V^u_{-n-1}))<1$. So $V^u=\mathcal{F}_u(W^u)\subset f[W^u]$.

\medskip
\textit{Part 4}: The inclusion $\subset$ is simple: Every $u$-admissible manifold $W_i^u$ in $\psi_{x_i}^{p_i^s,p_i^u}$ is contained in $\psi_{x_i}[R_{p_i^u}(0)]$, because if $W_i^u$ is represented by the function $F$ then any $p=\psi_{x_i}(v,w)$ in $W_i^u$ satisfies $|w|_\infty\leq p_i^u$ and $$|v|_\infty=|F(w)|_\infty\leq|F(0)|_\infty+\max\|d_\cdot F\|\cdot|w|_\infty\leq \varphi+\epsilon\cdot|w|_\infty\leq(10^{-3}+\epsilon)p_i^u<p_i^u$$ Applying this to $V^u:=V^u[(v_i)_{i\leq0}]$, we see that for every $p\in V^u$, $p\in \psi_{x_0}[R_{p_0^u}(0)]$. Hence by part 3 for every $k\geq0$:$$f^{-k}(p)\in f^{-k}[V^u]\subset V^u[(v_{i-k})_{i\leq 0}]\subset\psi_{x_{-k}}[R_{p_{-k}^u}(0)]\subset\psi_{x_{-k}}[R_{10Q_\epsilon(x_{-k})}(0)].$$

So now to prove $\supset$: Suppose $z\in\psi_{x_0}[R_{p_0^u}(0)]$ and $f^{-k}(z)\in\psi_{x_{-k}}[R_{10Q_\epsilon(x_{-k})}(0)]$ for all $k\geq0$. Write $z=\psi_{x_0}(v_0,w_0)$. We show that $z\in V^u$ by proving that $v_0=F(w_0)$ where $F$ is the representing function for $V^u$. Set $z':=\psi_{x_0}(F(w_0),w_0)=:\psi_{x_0}(v',w')$. For $k\geq0$, $f^{-k}(z),f^{-k}(z')\in\psi_{x_{-k}}[R_{10Q_\epsilon(x_{-k})}(0)]$ (the first point by assumption, and the second since $f^{-k}(z')\in f^{-k}[V^u]\subset V^u((v_{i-k})_{i\leq0}$). It is therefore possible to write: $$f^{-k}(z)=\psi_{x_{-k}}(v_{-k},w_{-k})\text{ and }f^{-k}(z')=\psi_{x_{-k}}(v'_{-k},w'_{-k})\text{  }(k\geq0)$$ where $|v_{-k}|_\infty,|w_{-k}|_\infty,|v'_{-k}|_\infty,|w'_{-k}|_\infty\leq10Q_\epsilon(x_{-k})$ for all $k\geq0$.

By proposition  \ref{3.4inomris} for $f^{-1}$, $\forall k\geq0$, $f^{-1}_{x_{-k-1}x_{-k}}=\psi_{x_{-k-1}}^{-1}\circ f^{-1}\circ\psi_{x_{-k}}$ can be written as $$f_{x_{-k-1}x_{-k}}^{-1}(v,w)=(D_{s,k}^{-1}v+h_{s,k}(v,w),D_{u,k}^{-1}w+h_{u,k}(v,w)),$$ where $\|D_{s,k}\|^{-1}\geq e^{\chi/2},\|D_{u,k}^{-1}\|\leq e^{-\chi/2}$, and $\max\limits_{R_{10Q_\epsilon(x_{-k})}(0)}\|d_\cdot h_{s/u,k}\|<\epsilon$ (provided $\epsilon$ is small enough). Define $\Delta v_{-k}:=v_{-k}-v'_{-k}$, and $\Delta w_{-k}:=w_{-k}-w'_{-k}$. Since for every $k\leq 0$, $(v_{-k-1},w_{-k-1})=f_{x_{-k-1}x_{-k}}^{-1}(v_{-k},w_{-k})$ and $(v'_{-k-1},w'_{-k-1})=f_{x_{-k-1}x_{-k}}^{-1}(v'_{-k},w'_{-k})$ we get the two following bounds:
$$|\Delta v_{-k-1}|_\infty\geq\|D_{s,k}\|^{-1}\cdot|\Delta v_{-k}|_\infty-\max\|d_\cdot h_{s,k}\|\cdot(|\Delta v_{-k}|_\infty+|\Delta w_{-k}|_\infty)\geq(e^{\chi/2}-\epsilon)|v_{-k}|_\infty-\epsilon|\Delta w_{-k}|_\infty,$$
$$|\Delta w_{-k-1}|_\infty\leq\|D_{u,k}^{-1}\|\cdot|\Delta w_{-k}|_\infty+\max\|d_\cdot h_{u,k}\|\cdot(|\Delta w_{-k}|_\infty+|\Delta v_{-k}|_\infty)\leq(e^{-\chi/2}+\epsilon)|w_{-k}|_\infty+\epsilon|\Delta v_{-k}|_\infty.$$
Denote $a_k:=|\Delta v_{-k}|_\infty,b_k:=|\Delta w_{-k}|_\infty$ and assume that $\epsilon$ is small enough that $e^{-\chi/2}+\epsilon\leq e^{-\chi/3},e^{\chi/2}-\epsilon\geq e^{\chi/3}$ then we get (since $b_0\equiv0$): $$a_{k+1}\geq e^{\chi/3}a_k-\epsilon b_k,b_{k+1}\leq e^{-\chi/3}b_k+\epsilon a_k.$$

Assume $\epsilon$ is so small that $e^{\chi/3}-\epsilon>1,e^{-\chi/3}+\epsilon<1$, then we claim $a_k\leq a_{k+1},b_k\leq a_k \forall k\geq0$: 
Proof by induction: This is true for $k=0$ since $b_0=0$. Assume by induction that $a_k\leq a_{k+1},b_k\leq a_k$
, then:
$$b_{k+1}\leq e^{-\chi/3}b_k+\epsilon a_k\leq(e^{-\chi/3}+\epsilon)a_k\leq a_k\leq a_{k+1},$$
$$a_{k+2}\geq e^{\chi/3}a_{k+1}-\epsilon b_{k+1}\geq(e^{\chi/3}-\epsilon)a_{k+1}\geq a_{k+1}.$$

We also see that $a_{k+1}\geq(e^{\chi/3}-\epsilon)a_k$ for all $k\geq0$, hence $a_k\geq(e^{\chi/3}-\epsilon)^k a_0$. Either $a_k\rightarrow\infty$ or $a_0=0$. But $a_k=|v_{-k}-v'_{-k}|_\infty\leq20Q_\epsilon(x_{-k})<20\epsilon$, so $a_0=0$, and therefore $v_0=v'_0$, whence $F(w'_0)=F(w_0)$. So

$$z=\psi_x(F(w_0),w_0)\in V^u,$$
which concludes the proof of part 4.

\medskip
\textit{Part 5}: If two chains $v=(v_i)_{i\in\mathbb{Z}},w=(w_i)_{i\in\mathbb{Z}}$ satisfy $v_i=w_i$ for $|i|\leq N$ then $dist(V^u((v_i)_{i\leq0}),V^u((w_i)_{i\leq0}))\leq e^{-\frac{1}{2}N\chi}$: Given $n>N$ let $V_{-n}^u$ be a $u$-admissible manifold in $v_{-n}$, and let $W_{-n}^u$ be a $u$-admissible manifold in $w_{-n}$. $\mathcal{F}_u^{n-N}(V_{-n}^u)$ and $\mathcal{F}_u^{n-N}(W_{-n}^u)$ are $u$-admissible manifolds in $v_{-N}(=w_{-N})$. Let $F_N,G_N$ be their representing functions. Admissibility implies that $$\|F_N-G_N\|_\infty\leq\|F_N\|_\infty+\|G_N\|_\infty<2Q_\epsilon<1$$ $$\|d_\cdot F_N-d_\cdot G_N\|_\infty\leq\|d_\cdot F_N\|_\infty+\|d_\cdot G_N\|_\infty<2\epsilon<1$$
Represent $\mathcal{F}_u^{n-k}(V_{-n}^u)$ and $\mathcal{F}_u^{n-k}(W_{-n}^u)$ by $F_k$ and $G_k$. By the previous proposition, $$\|F_{k-1}-G_{k-1}\|_\infty\leq e^{-\chi/2}\|F_k-G_k\|_\infty,$$
$$\|d_\cdot F_{k-1}-d_\cdot G_{k-1}\|_\infty\leq e^{-\chi/2}(\|d_\cdot F_k-d_\cdot G_k\|_\infty+2\|F_k-G_k\|_\infty^{\beta/3}).$$
Iterating the first inequality, from $N$, going down, we get: $\|F_k-G_k\|_\infty\leq e^{-\frac{1}{2}(N-k)\chi}$. So: $$dist(\mathcal{F}_u^n(V^u_{-n}),\mathcal{F}_u^n(W^u_{-n}))\leq e^{-\frac{1}{2}N\chi}.$$ Passing to the limit as $n\rightarrow\infty$ we get: $$dist(V^u((v_i)_{i\leq0}),V^u((w_i)_{i\leq0}))\leq e^{-\frac{1}{2}N\chi}.$$
Now plug $\|F_k-G_k\|_\infty\leq e^{-\frac{1}{2}\chi(N-k)}$ in $\|d_\cdot F_{k-1}-d_\cdot G_{k-1}\|_\infty\leq e^{-\chi/2}(\|d_\cdot F_k-d_\cdot G_k\|_\infty+2\|F_k-G_k\|_\infty^{\beta/3})$, and set $c_k:=\|d_\cdot F_k-d_\cdot G_k\|_\infty,\theta_1=e^{-\frac{1}{2}\chi},\theta_2=e^{-\frac{1}{6}\beta\chi}$ then: $$c_{k-1}\leq \theta_1(c_k+2\theta_2^{N-k})$$ hence (by induction): 

$\forall 0\leq k\leq N$: $c_0\leq\theta_1^kc_k+2(\theta_1^k\theta_2^{N-k}+\theta_1^{k-1}\theta_2^{N-k+1}+...+\theta_1\theta_2^{N-1})$. 
Now substitute $k=N$. Since $\theta_1<\theta_2,c_N\leq1$, $c_0\leq\theta_1^N+2N\theta_2^N<(2N+1)\theta_2^N$. 

It follows that $dist_{C^1}(\mathcal{F}_u^n(V^u_{-n}),\mathcal{F}_u^n(W^u_{-n}))\leq2(N+1)\theta_2^N$. In part 2 we saw that $\mathcal{F}_u^n(V^u_{-n})$ and $\mathcal{F}_u^n(W^u_{-n})$ converge to  $V^u((v_i)_{i\leq0})$ and $V^u((w_i)_{i\leq0})$ (resp.) in $C^1$. Therefore if we pass to the limit as $n\rightarrow\infty$ we get that $dist_{C^1}(V^u((v_i)_{i\leq0}),V^u((w_i)_{i\leq0}))\leq2(N+1)\theta_2^N$. Now pick $\theta\in(\theta_2,1),K>0$ s.t. $2(N+1)\theta_2^N\leq K\theta^N$ $\forall N$.
\end{proof}
\begin{theorem}\label{DefOfPi}
Given a chain of double charts $(v_i)_{i\in\mathbb{Z}}$, let $\pi((\underline v)):=$unique intersection point of $V^u((v_i)_{i\leq0})$ and $V^s((v_i)_{i\geq0})$. Then \begin{enumerate}
\item $\pi$ is well-defined and $\pi\circ\sigma=f\circ\pi$.
\item $\pi:\Sigma\rightarrow M$ is H\"older continuous.
\item $\pi[\Sigma]\supset\pi[\Sigma^\#]\supset NUH_\chi^\#(f)$, therefore $\pi[\Sigma]$ and $\pi[\Sigma^\#]$ have full probability w.r.t any hyperbolic invariant measure with Lyapunov exponents bounded away from 0 by $\chi$.
\end{enumerate}
\end{theorem}
\begin{proof}
\textit{Part 1}: $\pi$ is well-defined thanks to proposition \ref{firstbefore}. Now, write $v_i=\psi_{x_i}^{p_i^s,p_i^u}$ and $z=\pi((\underline v))$. We claim that $f^k(z)\in\psi_{x_k}[R_{Q_\epsilon(x_k)}(0)]\text{  }(k\in\mathbb{Z})$. For $k=0$ this is true because $z\in V^s((v_i)_{i\geq0})$, and this is an $s$-admissible manifold in $\psi_{x_0}^{p_0^s,p_0^u}$. For $k>0$ we use the previous proposition (part 3) to see that $$f^k(z)\in f^k[V^s((v_i)_{i\geq0})]\subset V^s((v_{i+k})_{i\geq0}).$$
Since $V^s((v_{i+k})_{i\geq0})$ is an $s$-admissible manifold in $\psi_{x_k}^{p_k^s,p_k^u}$,  $f^k(z)\in\psi_{x_k}[R_{Q_\epsilon(x_k)}(0)]$. The case $k<0$ can be handled the same way, using $V^u((v_i)_{i\leq0})$. Thus $z$ satisfies: $$f^k(z)\in\psi_{x_k}[R_{Q_\epsilon(x_k)}(0)]\text{  }\text{ for all }k\in\mathbb{Z}.$$
Any point which satisfies this must be $z$, since by the previous proposition (part 4) it must lie in $V^u((v_i)_{i\leq0})\cap V^s((v_i)_{i\geq0})$, and this equation characterizes $\pi(\underline v)=z$. Hence: $$f^k(f(\pi(\underline v)))=f^{k+1}(\pi(\underline v))\in\psi_{x_{k+1}}^{p_{k+1}^s,p_{k+1}^u}\text{  }\forall k\in\mathbb{Z}$$ and this is the condition that characterizes $\pi(\sigma\underline v)$.

\medskip
\textit{Part 2}: We saw that $\underline v\mapsto V^u((v_i)_{i\leq0})$ and $\underline v\mapsto V^s((v_i)_{i\geq0})$ are H\"older continuous (previous proposition). Since the intersection of an $s$-admissible manifold and a $u$-admissible manifold is a Lipschitz function of those manifolds (proposition \ref{firstbefore}), $\pi$ is also H\"older continuous.

\medskip
\textit{Part 3}: We prove $\pi[\Sigma]\supset NUH_\chi^\#(f)$: Suppose $x\in NUH_\chi^\#(f)$, then by proposition \ref{prop131} there exist $\psi_{x_k}^{p_k^s,p_k^u}\in\mathcal{V}$ s.t. $\psi_{x_k}^{p_k^s,p_k^u}\rightarrow\psi_{x_{k+1}}^{p_{k+1}^s,p_{k+1}^u}$ for all $k$, and s.t. $\psi_{x_k}^{p_k^s\wedge p_k^u}$ $\epsilon-$overlaps $\psi_{f^k(x)}^{p_k^s\wedge p_k^u}$ for all $k\in\mathbb{Z}$. This implies $$f^k(x)=\psi_{f^k(x)}(0)\in\psi_{f^k(x)}[R_{p_k^s\wedge p_k^u}(0)]\subset\psi_{f^k(x)}[R_{Q_\epsilon(x_k)}(0)].$$
Then $x$ satisfies $f^k(x)\in\psi_{x_k}[R_{Q_\epsilon(x_k)}(0)]\text{  }(k\in\mathbb{Z})$. It follows that $x=\pi(\underline v)$ with $\underline v=(\psi_{x_i}^{p_i^s,p_i^u})_{i\in\mathbb{Z}}$. By lemma \ref{subordinatedchain} applied to $\{Q_\epsilon(f^n(x))\}_{n\in\mathbb{Z}}$ and $\{q_\epsilon(f^n(x))\}_{n\in\mathbb{Z}}$, there exists a sequence $\{(p^s_i,p^u_i)\}_{i\in\mathbb{Z}}$ $\epsilon$-subordinated to $\{Q_\epsilon(f^i(x))\}_{i\in\mathbb{Z}}$ s.t. $p^s_i\wedge p^u_i\geq q_\epsilon(f^i(x))$ for all $i\in\mathbb{Z}$
.
By the definition of $NUH_\chi^\#(f)$ there exists a sequence $i_k,j_k\uparrow\infty$ for which $p_{i_k}^s\wedge p_{i_k}^u\geq e^{-\frac{\epsilon}{3}}q_\epsilon(f^{i_k}(x))$ and $p_{-j_k}^s\wedge p_{-j_k}^u\geq e^{-\frac{\epsilon}{3}}q_\epsilon(f^{-j_k}(x))$ are bounded away from zero (see proposition \ref{subordinatedchain}). By the discreteness property of $\mathcal{A}$ (proposition  \ref{discreteness}), $\psi_{x_i}^{p_i^s,p_i^u}$ must repeat some symbol infinitely often in the past, and some symbol (possibly a different one) infinitely often in the future. Thus the above actually proves that $$\pi[\Sigma^\#]\supset NUH_\chi^\#(f),$$
where $\Sigma^\#=\{\underline v\in\Sigma: \exists v,w\in\mathcal{V},n_k,m_k\uparrow\infty\text{ s.t. }v_{n_k}=v,v_{-m_k}=w\}$.
\end{proof}
\subsubsection{The relevant part of the extension}
\begin{definition} A double chart $v=\psi_x^{p^s,p^u}$ is called {\em relevant} iff there exists a chain $(v_i)_{i\in\mathbb{Z}}$ s.t. $v_0=v$ and $\pi(\underline v)\in NUH_\chi^\#(f)$. A double chart which is not relevant is called {\em irrelevant}.
\end{definition}
\begin{definition}\label{sigmarel} {\em The relevant part of $\Sigma$} is $\Sigma_{rel}:=\{\underline v\in\Sigma: v_i\text{ is relevant for all }i\}$.

A chain $\underline{u}\in \Sigma_{rel}$ is called {\em ``regular"} if $\underline{u}$ also belongs to $\Sigma^\#$.
\end{definition}
\begin{prop}\label{sigmarelprop}
The previous theorem holds with $\Sigma_{rel}$ replacing $\Sigma$.
\end{prop}
\begin{proof} All the properties of $\pi:\Sigma_{rel}\rightarrow M$ are obvious, besides the statement that $\pi[\Sigma_{rel}^\#]\supset NUH_\chi^\#(f)$, where $\Sigma_{rel}^\#=\Sigma^\#\cap\Sigma_{rel}$: Suppose $p\in NUH_\chi^\#(f)$, then the proof of the last theorem shows that $\exists\underline v\in\Sigma^\#$ s.t. $\pi(\underline v)=p$. Since $NUH_\chi^\#(f)$ is $f$-invariant and $f\circ\pi=\pi\circ\sigma$ then $\pi(\sigma^i(\underline v))=f^i(p)\in NUH_\chi^\#(f)$. So $v_i$ is relevant for all $i\in\mathbb{Z}$. It follows that $v\in\Sigma_{rel}^\#$.
\end{proof}
\noindent\textbf{Summary up to here}
\begin{itemize}
    \item We defined $\epsilon$-chains (definition \ref{epsilonchains}).
    \item We proved the shadowing lemma: for every chain $v=(\psi^{p_i^s,p_i^u}_{x_i})_{i\in\mathbb{Z}}$ there exists $x\in M$ s.t. $f^i(x)\in \psi_{x_i}[R_{p^s_i\wedge p^u_i}(0)]$ $(i\in\mathbb{Z})$ [proof of theorem \ref{DefOfPi}].
    \item We found a countable set of charts $\mathcal{A}$ s.t. the countable Markov shift $$\Sigma=\{v\in \mathcal{A}^\mathbb{Z}: v\text{ is an }\epsilon-\text{chain}\}$$ is a Markov extension of a set containing $NUH_\chi^\#$ (theorem \ref{DefOfPi});
    But the coding $\pi:\Sigma\rightarrow M$ is not finite-to-one.
\end{itemize}
Along the way we saw that if $x$ is shadowed by $v=(v_i)_{i\in\mathbb{Z}}$ then $V^s((v_i)_{i\geq0}),V^u((v_i)_{i\leq0})$ are local stable/unstable manifolds of $x$. (proposition \ref{firstofchapter}).

\section{Regular chains which shadow the same orbit are close}\label{chapter333}
\subsection{The inverse problem for regular chains}
Let $\Sigma_{rel}^\#:=\Sigma^\#\cap\Sigma_{rel}$ (see definition \ref{sigmarel} and the proof of  proposition \ref{sigmarelprop}). The aim of this part is to show that the map $\pi:\Sigma_{rel}^\#\rightarrow NUH_\chi$ from theorem \ref{DefOfPi}, is ``almost invertible": if $\pi((\psi_{x_i}^{p_i^s,p_i^u})_{i\in\mathbb{Z}})=\pi((\psi_{y_i}^{q_i^s,q_i^u})_{i\in\mathbb{Z}})$, then $\forall i$: $x_i\approx y_i$, $p^{s/u}_i\approx q^{s/u}_i$ and $C_\chi(x_i)\approx C_\chi(y_i)$, where $\approx$ means that the respective parameters belong to the same compact sets. The compact sets can be as small as we wish by choosing small enough $\epsilon$. The discretization we constructed in \textsection 1.2.3 in fact allows these compact sets to contain only a finite number of such possible parameters for our charts. This does not mean that $\pi$ is finite-to-one, but it will allow us to construct a finite-to-one coding later. 
\subsubsection{Comparing orbits of tangent vectors}
\begin{lemma}\label{firstchapter2}
Let $(\psi_{x_i}^{p_i^s,p_i^u})_{i\in\mathbb{Z}},(\psi_{y_i}^{q_i^s,q_i^u})_{i\in\mathbb{Z}}$ be two chains s.t. $\pi((\psi_{x_i}^{p_i^s,p_i^u})_{i\in\mathbb{Z}})=\pi((\psi_{y_i}^{q_i^s,q_i^u})_{i\in\mathbb{Z}})=p$, then $d(x_i,y_i)<\frac{\sqrt{d}}{50}\max\{p_i^s\wedge p_i^u,q_i^s\wedge q_i^u\}$.
\end{lemma}
\begin{proof}$f^i(p)$ is the intersection of a $u$-admissible manifold and an $s$-admissible manifold in $\psi_{x_i}^{p_i^s,p_i^u}$, therefore (by proposition \ref{firstbefore}) $f^i(p)=\psi_{x_i}(v_i,w_i)$ where $|v_i|_\infty,|w_i|_\infty<10^{-2}(p_i^s\wedge p_i^u)$. $d(f^i(p),x_i)=d(\exp_{x_i}(C_\chi(x_i)(v_i,w_i)),x_i)=|C_\chi(x_i)(v_i,w_i)|_2\leq \sqrt{d}|(v_i,w_i)|_\infty\leq\frac{\sqrt{d}}{100}(p_i^s\wedge p_i^u)$. Similarly $d(f^i(x),y_i)<\frac{\sqrt{d}}{100}(q_i^s\wedge q_i^u)$. It follows that $d(x_i,y_i)<\frac{\sqrt{d}}{50}\max\{q_i^s\wedge q_i^u,p_i^s\wedge p_i^u\}$.
\end{proof}
The following definitions are taken from \cite{Sarig}.
\begin{definition}
Let $V^u$ be a $u$-admissible manifold in the double chart $\psi_x^{p^s,p^u}$. We say that $V^u$ {\em stays in windows} if there exists a negative chain $(\psi_{x_i}^{p_i^s,p_i^u})_{i\leq0}$ with $\psi_{x_0}^{p_0^s,p_0^u}=\psi_x^{p^s,p^u}$ and $u$-admissible manifolds $W_i^u$ in $\psi_{x_i}^{p_i^s,p_i^u}$ s.t. $f^{-|i|}[V^u]\subset W_i^u$ for all $i\leq0$.
\end{definition}
\begin{definition}
Let $V^s$ be an $s$-admissible manifold in the double chart $\psi_x^{p^s,p^u}$. We say that $V^s$ {\em stays in windows} if there exists a positive chain $(\psi_{x_i}^{p_i^s,p_i^u})_{i\geq0}$ with $\psi_{x_0}^{p_0^s,p_0^u}=\psi_x^{p^s,p^u}$ and $s$-admissible manifolds $W_i^s$ in $\psi_{x_i}^{p_i^s,p_i^u}$ s.t. $f^i[V_i^u]\subset W_i^s$ for all $i\geq0$.
\end{definition}

Remark: If $\underline v$ is a chain, then $V_i^u:=V^u((v_k)_{k\leq i})$ and $V_i^s:=V^s((v_k)_{k\geq i})$ stay in windows, since $f^{-k}[V_i^u]\subset V_{i-k}^u$ and $f^{k}[V_i^s]\subset V_{i+k}^s$ for all $k\geq0$ (by proposition  \ref{firstofchapter}).
\begin{prop}\label{Lambda} The following holds for all $\epsilon$ small enough: Let $V^s$ be an admissible $s$-manifold in $\psi_x^{p^s,p^u}$, and suppose $V^s$ stays in windows.
\begin{enumerate}
\item For every $y,z\in V^s$: $d(f^k(y),f^k(z))<4e^{-\frac{1}{2}k\chi}$ for all $k\geq0$.
\item For every $y\in V^s$ and $u\in T_yV^s(1)$: $|d_yf^ku|\leq4e^{-k\frac{2\chi}{3}}\|C_\chi(x)^{-1}\|$ for all $k\geq0$.

\noindent Recall: $T_yV^s(1)$ is the unit ball in $T_yV^s$, w.r.t the Riemannian norm.
\item For $y\in V^s$, let $\omega_s(y)$ be a normalized volume form of $T_yV^s$. For all $y,z\in V^s$
,  $\forall k\geq0$:
$$|\log|d_yf^k\omega_s(y)|-\log|d_zf^k \omega_s(z)||<\epsilon d(y,z)^{\beta/4},$$
\end{enumerate}
 
\noindent where $|d_yf^k\omega_s(y)|=\bigl|\det d_yf^k|_{T_yV^s}\bigr|,|d_zf^k\omega_s(z)|=\bigl|\det d_zf^k|_{T_zV^s}\bigr|$. The symmetric statement holds for $u$-admissible manifolds which stay in windows, for ``$s$" replaced by ``$u$" and $f$ by $f^{-1}$.
\end{prop}
\noindent\textbf{Remark:} Part (3) is not used in this work, but we plan to use it in future projects.
\begin{proof}
For the proofs of parts (1) and (2), see \cite[proposition~6.3(1),(2)]{Sarig}.

Part(3): Suppose $V^s$ is an $s$-admissible manifold in $\psi_{x}^{p^s,p^u}$, which stays in windows, then there is a positive chain $(\psi_{x_i}^{p_i^s,p_i^u})_{i\geq0}$ s.t. $\psi_{x_0}^{p_0^s,p_0^u}=\psi_x^{p^s,p^u}$, and there are $s$-admissible manifolds $W_i^s$ in $\psi_{x_i}^{p_i^s,p_i^u}$ s.t. $f^i[V^s]\subset W_i^s$ for all $i\geq0$. We write: \begin{itemize}
\item $V^s=\{(t,F_0(t)): |t|_\infty\leq p^s\}$
\item $W_i^s=\psi_{x_i}[\{(t,F_i(t)):|t|_\infty\leq p_i^s\}]$
\end{itemize} Admissibility means that $\|d_\cdot F_i\|_{\beta/3}\leq\frac{1}{2},\|d_0F_i\|\leq\frac{1}{2}(p_i^s\wedge p_i^u)^{\beta/3}$, $|F_i(0)|_\infty\leq10^{-3}(p_i^s\wedge p_i^u)$ and $\mathrm{Lip}(F_i)<\epsilon$ ($\because$ definition \ref{admissible}). 


Since $V^s$ stays in windows: $f^k[V^s]\subset\psi_{x_k}[R_{Q_\epsilon(x_k)}(0)]$ for all $k\geq0$. Therefore $\forall y,z\in V^s$ one can write $f^k(y)=\psi_{x_k}(y_k')$ and $f^k(z)=\psi_{x_k}(z_k')$, where $y_k':=(y_k,F_k(y_k)),z_k':=(z_k,F_k(z_k))$ belong to $R_{Q_\epsilon(x_k)}(0)$. For every $k$: $y_{k+1}'=f_{x_kx_{k+1}}(y_k')$ and $z_{k+1}'=f_{x_kx_{k+1}}(z_k')$. (Reminder: $f_{x_kx_{k+1}}=\psi_{x_{k+1}}^{-1}\circ f\circ\psi_{x_k}$).
By a small calculation which appears in \cite[proposition~6.3(1)]{Sarig},

\begin{align}\label{calcin6.3.1}  
|y_k'-z_k'|_\infty\leq&e^{-\frac{1}{2}k\chi}|y_0'-z_0'|_\infty=e^{-\frac{1}{2}k\chi}|\psi_x^{-1}(y)-\psi_x^{-1}(z)|\\
\leq&e^{-\frac{1}{2}k\chi}2\|C_\chi^{-1}(x)\|d(y,z)=2\|C_\chi^{-1}(x)\|e^{-\frac{1}{2}k\chi}\cdot d(y,z), \text{ }\forall k\geq0. \nonumber
\end{align}
In addition $f^k(y)=\psi_{x_k}(y_k'),f^k(z)=\psi_{x_k}(z_k')$, whence \begin{equation}\label{chromecastsonic}d(f^k(y),f^k(z))\leq \mathrm{Lip}(\psi_{x_k})|y_k'-z_k'|\leq4\|C_\chi^{-1}(x)\|e^{-\frac{1}{2}k\chi}\cdot d(y,z).\end{equation}
Fix $n\in\mathbb{N}$, and denote $A:=\Big|\log|d_yf^n\omega_s(y)|-\log|d_zf^n\omega_s(z)|\Big|$. For every $p\in V^s
$: $$d_pf^n\omega_s(p)=d_{f(p)}f^{n-1}d_pf(\omega_s(p))=|d_pf\omega_s(p)|\cdot d_{f(p)}f^{n-1}\omega_s(f(p))=\cdots=\prod_{k=1}^{n}|d_{f^k(p)}f\omega_s(f^k(p))|\cdot \omega_s(f^{n}(p)),$$ where $\omega_s(f^k(p)):=\frac{d_pf^k\omega_s(p)}{|d_pf^k\omega_s(p)|}$ is a normalized volume form of $T_{f^k(p)}f^k[V^s]$. Thus: $$A:=\Big|\log\frac{|d_yf^n\omega_s(y)|}{|d_zf^n\omega_s(z)|}\Big|\leq\sum_{k=1}^{n}\Big|\log|d_{f^k(y)}f\omega_s(f^k(y))|-\log|d_{f^k(z)}f\omega_s(f^k(z))|\Big|.$$

In section 1.2 we have covered $M$ by a finite collection $\mathcal{D}$ of open sets $D$, equipped with a smooth map $\theta_D:TD\rightarrow\mathbb{R}^d$ s.t. $\theta_D|_{T_xM}:T_xM\rightarrow\mathbb{R}^d$ is an isometry and $\nu_x:=\theta_D^{-1}:\mathbb{R}^d\rightarrow TM$ has the property that $(x,v)\mapsto\nu_x(v)$ is Lipschitz on $D\times B_1(0)$. Since $f$ is $C^{1+\beta}$ and $M$ is compact, $d_pfv$ depends in a $\beta$-H\"older way on $p$, and in a Lipschitz way on $v$. It follows that there exists a constant $H_0>1$ s.t. for every $D,D'\in\mathcal{D}; y',z'\in D,\text{ }f(y'),f(z')\in D'; u,v\in\mathbb{R}^d(1)$: $\Big|\Theta_{D'}d_{y'}f\nu_{y'}(u)-\Theta_{D'}d_{z'}f\nu_{z'}(v)\Big|<H_0(d(y',z')^\beta+|u-v|_2)$. Choose $D_k\in\mathcal{D}$ s.t. $D_k\ni f^k(y),f^k(z)$. Such sets exist provided $\epsilon$ is much smaller than the Lebesgue number of $\mathcal{D}$, since by part 1 $d(f^k(y),f^k(z))<4\epsilon$. By writing $Id=\Theta_{D_k}\circ\nu_{f^k(y)}$ and $Id=\Theta_{D_k}\circ\nu_{f^k(z)}$, and recalling that $\Theta_{k}$ are isometries, we see that:
\begin{align}\label{NewA}
A\leq&\sum_{k=1}^n\Big|\log|\Theta_{D_{k+1}}d_{f^k(y)}f\nu_{f^k(y)}\Theta_{D_k}\omega_s(f^k(y))|-\log|\Theta_{D_{k+1}}d_{f^k(z)}f\nu_{f^k(z)}\Theta_{D_k}\omega_s(f^k(z))|\Big|\\
\leq& \sum_{k=1}^nM_f^d\Big|\Theta_{D_{k+1}}d_{f^k(y)}f\nu_{f^k(y)}\Theta_{D_k}\omega_s(f^k(y))-\Theta_{D_{k+1}}d_{f^k(z)}f\nu_{f^k(z)}\Theta_{D_k}\omega_s(f^k(z))\Big|,\nonumber
\end{align}
by the inequalities $|\log a-\log b|\leq\frac{|a-b|}{\min\{a,b\}}$ and the following estimate:
\begin{align*}\inf_{p\in\{y,z\},k\geq0}|\Theta_{D_{k+1}}d_{f^k(p)}f\nu_{f^k(p)}\Theta_{D_k}\omega_s(f^k(p))|=&\inf_{p\in\{y,z\}}|\mathrm{Jac}(d_pf|_{T_pV^s})|\\
=&\inf_{p\in\{y,z\}}\mathrm{Jac}(\sqrt{(d_pf|_{T_pV^s})^td_pf|_{T_pV^s}})\geq M_f^{-d}.
\end{align*}
Let $\omega_1$ be the standard volume form for $\mathbb{R}^{s(x)}$, and consider the map $t\overset{G_k}{\mapsto}(t,F_k(t))$, $G_k:R_{p^s_k}(0)(\subset \mathbb{R}^{s(x)})\rightarrow \mathbb{R}^d$. Then $$\omega_s(f^k(y))=\pm\frac{d_{y_k}(\psi_{x_k}\circ G_k)\omega_1}{|d_{y_k}(\psi_{x_k}\circ G_k)\omega_1|},\text{ }\omega_s(f^k(z))=\pm\frac{d_{z_k}(\psi_{x_k}\circ G_k)\omega_1}{|d_{z_k}(\psi_{x_k}\circ G_k)\omega_1|},$$
where the $\pm$ notation means the equality holds for one of the signs. Since we are only interested in the absolute value of the volume forms, we can assume WLOG that the sign is identical for both of the volume forms. $\forall u\in \mathbb{R}^d(1)$, $|d_{y_k/z_k}(\psi_{x_k}\circ G_k)u|\geq\frac{1}{2\cdot\|C_\chi^{-1}(x_k)\|}$, then $|d_{y_k/z_k}(\psi_{x_k}\circ G_k)\omega_1|\geq\frac{1}{(2\|C_\chi^{-1}(x_k)\|)^{s(x)}}$. 
Recalling that $\Theta_{D_k}$ are isometries, we get
\begin{align}\label{shezifhadash}
    |\Theta_{D_k}\omega_s(f^k(y))-\Theta_{D_k}\omega_s(f^k(z))|&=\Big|\frac{\Theta_{D_k}d_{y_k}(\psi_{x_k}\circ G_k)\omega_1}{|\Theta_{D_k}d_{y_k}(\psi_{x_k}\circ G_k)\omega_1|}-\frac{\Theta_{D_k}d_{z_k}(\psi_{x_k}\circ G_k)\omega_1}{|\Theta_{D_k}d_{z_k}(\psi_{x_k}\circ G_k)\omega_1|}\Big|\\
    &\leq(2\|C_\chi^{-1}(x_k)\|)^{s(x)}\Big|\Theta_{D_k}d_{y_k}(\psi_{x_k}\circ G_k)\omega_1-\Theta_{D_k}d_{z_k}(\psi_{x_k}\circ G_k)\omega_1\Big|.\nonumber
\end{align}
Substituting this back in \eqref{NewA} yields $\Big($ recall Hadamard's inequality, which asserts that if a volume form $\widetilde{\omega}$ is the wedge product $\widetilde{\omega}=u_1\wedge...\wedge u_l$, then $|\widetilde{\omega}|\leq\prod_{i=1}^l|u_i|_2 \Big)$,
\begin{align}\label{wubbalubba}
    A\leq & M_f^d\sum_{k=1}^n\Big|\Theta_{D_{k+1}}d_{f^k(y)}f\nu_{f^k(y)}\Theta_{D_k}\omega_s(f^k(y))-\Theta_{D_{k+1}}d_{f^k(z)}f\nu_{f^k(z)}\Theta_{D_k}\omega_s(f^k(y))\\
    + & \Theta_{D_{k+1}}d_{f^k(z)}f\nu_{f^k(z)}\Theta_{D_k}\omega_s(f^k(y))-\Theta_{D_{k+1}}d_{f^k(z)}f\nu_{f^k(z)}\Theta_{D_k}\omega_s(f^k(z))\Big|\nonumber\\
    \leq & M_f^d\sum_{k=1}^n\|d_{f^k(z)}f\|^{s(x)}\cdot|\Theta_{D_k}\omega_s(f^k(y))-\Theta_{D_k}\omega_s(f^k(z))|+\|\Theta_{D_{k+1}}d_{f^k(y)}f\nu_{f^k(y)}-\Theta_{D_{k+1}}d_{f^k(z)}f\nu_{f^k(z)}\|\nonumber\\
    \leq & M_f^d\cdot \sum_{k=1}^n\Big( M_f^d\cdot (2\|C_\chi^{-1}(x_k)\|)^{s(x)}\cdot\Big|\Theta_{D_k}d_{y_k}(\psi_{x_k}\circ G_k)\omega_1-\Theta_{D_k}d_{z_k}(\psi_{x_k}\circ G_k)\omega_1\Big|+H_0d(f^k(y),f^k(z))^\beta\Big)\nonumber\\
    \leq& M_f^d\cdot \sum_{k=1}^n\Big( M_f^d\cdot (2\|C_\chi^{-1}(x_k)\|)^{s(x)}\cdot[\text{H\"ol}_{\frac{\beta}{3}}(d_\cdot (\psi_{x_k}\circ G_k))|y_k-z_k|^\frac{\beta}{3}]^{s(x)}+H_0d(f^k(y),f^k(z))^\beta\Big)\nonumber\\
    \leq& M_f^d\cdot \sum_{k=1}^n\Big( M_f^d\cdot (2\|C_\chi^{-1}(x_k)\|)^{s(x)}\cdot(2\cdot\frac{1}{2})^{s(x)}(|y_k-z_k|^\frac{\beta}{3.5}\cdot|y_k'-z_k'|^{\frac{\beta}{21}\cdot s(x)})+H_0d(f^k(y),f^k(z))^\beta\Big)\nonumber\\
    \leq&\sum_{k=1}^n M_f^{2d}2^{2d}(\|C_\chi^{-1}(x_k)\|\cdot (p_k^s)^\frac{\beta}{21})^{s(x)}\cdot [4\|C_\chi^{-1}(x)\|e^{\frac{-\chi}{2}k}d(y,z)]^\frac{\beta}{3.5}+M_f^dH_0d(f^k(y),f^k(z))^\beta\text{  }(\because\eqref{calcin6.3.1}),\nonumber
\end{align}
where H\"ol$_\frac{\beta}{3}(d_\cdot G_k$)=H\"ol$_\frac{\beta}{3}(d_\cdot F_k)\leq\frac{1}{2}$, and Lip($d_\cdot\psi_{x_k})\leq2$ since $\|C_\chi(x_k)\|\leq1$. The second inequality in \eqref{wubbalubba} uses the fact that $M_f>1$, and the penultimate inequality uses the fact that $|y_k-z_k|<1$. 

By the choice of $Q_\epsilon(\cdot)$ (definition \ref{balls}) and the inequality $p^s_k\leq Q_\epsilon(x_k)$, we have $\|C_\chi^{-1}(x_k)\|\cdot (p_k^s)^\frac{\beta}{21}<\epsilon^2$ for all $k\geq0$. So for small enough $\epsilon$,
\begin{align*}
    A\leq& \sum_{k=1}^n\epsilon^\frac{3}{2}\|C_\chi^{-1}(x)\|^\frac{\beta}{3.5}d(y,z)^\frac{\beta}{3.5}e^{-\frac{\chi\beta}{8}k}+M_f^dH_0d(f^k(y),f^k(z))^\beta\\
    \leq&\sum_{k=1}^n\|C_\chi^{-1}(x)\|\Big(\epsilon^\frac{3}{2}d(y,z)^\frac{\beta}{3.5}e^{-\frac{\chi\beta}{8}k}+M_f^dH_0\cdot4 e^{-\frac{\chi\beta}{2}k}d(y,z)^\beta\Big) \text{ (}\because\eqref{chromecastsonic})\\
    \leq&2\epsilon^\frac{3}{2}\|C_\chi^{-1}(x)\|d(y,z)^\frac{\beta}{3.5}\sum_{k=1}^ne^{-\frac{\chi\beta}{8}k}\leq2\epsilon^\frac{3}{2}\Big(\sum_{k=1}^\infty e^{-\frac{\chi\beta}{8}k}\Big)\|C_\chi^{-1}(x)\|d(y,z)^\frac{\beta}{3.5}\\
    \leq& \epsilon \|C_\chi^{-1}(x)\|d(y,z)^\frac{\beta}{3.5}
    \leq \epsilon d(y,z)^\frac{\beta}{4}\text{ (for }\epsilon\text{ small enough)}.
\end{align*}
This estimate is uniform in $n$, and part (3) is proven.

\end{proof}

\begin{lemma}\label{RegularOnVs}
Assume $W^s=V^s((v_i)_{i\geq0})$, and $\exists z\in W^s\cap NUH_\chi^\#$, then $\exists C_z<\infty$ s.t. $\forall y\in W^s, \eta\in T_yW^s(1)$, $\sqrt{2\sum_{m=0}^\infty |d_yf^m\eta|^2e^{2\chi m}}\equiv S(y,\eta)\leq C_z$.
\end{lemma}
\begin{proof}
$z\in NUH_\chi^\#$, therefore $\chi_z=\min\{|\chi_i(z)|\}>\chi$ (see definition \ref{chi_z}); 
Consider the Lyapunov change of coordinates $\widetilde{C}_{\chi}(z)$ (definition \ref{chi_z}). By the definition of $NUH_\chi^* (\supset NUH_\chi^\#)$ (definition \ref{chi_z}, claim \ref{tempered}), $\lim_{n\rightarrow\infty}\frac{1}{n}\log\|\widetilde{C}_\chi^{-1}(f^n(z))\|=0$. 
Let $\overline{Q}_\epsilon(x):=\max\{Q\in I_\epsilon: Q\leq \frac{1}{3^\frac{6}{\beta}}\epsilon^\frac{90}{\beta}\|\widetilde{C}_\chi^{-1}(x)\|^\frac{-48}{\beta}\}$ (this is the same as definition \ref{balls}(e), but with $\widetilde{C}_\chi$ replacing $C_\chi$). Then $\lim\limits_{|n|\rightarrow\infty}\frac{1}{n}\log \overline{Q}_\epsilon(f^n(x))=0$ for every $z\in NUH^*$.
Similarly, define $\frac{1}{\widetilde{q}_\epsilon(f^n(z))}:=\frac{1}{\epsilon}\sum_{k\in\mathbb{Z}}e^{-|k|\epsilon}\frac{1}{\overline{Q}_\epsilon(f^{k+n}(z))}$. As in lemma \ref{omega0}(5), $\frac{\widetilde{q}_\epsilon(f^{n+1}(z))}{\widetilde{q}_\epsilon(f^n(z))}=e^{\pm\epsilon}$, and $0<\widetilde{q}_\epsilon(f^n(z))\leq \epsilon \overline{Q}_\epsilon(f^n(z))$ for all $n\geq0$.

Define $\widetilde{\psi}_{f^m(z)}:=\exp_{f^m(z)}\circ\widetilde{C}_\chi(f^m(x))$ ($m\geq0)$. Define also $\widetilde{f}_{f^n(z)}:=\widetilde{\psi}_{f^{n+1}(z)} \circ f\circ\ \widetilde{\psi}_{f^n(z)}$ and recall $\widetilde{D}_\chi(f^n(z))=\widetilde{C}^{-1}_\chi(f^{n+1}(z))\circ d_{f^n(z)} f\circ \widetilde{C}_\chi(f^n(z))$ (definition \ref{chi_z}).
Define $\eta_n:=\delta_z\cdot\widetilde{q}_\epsilon(f^n(z))$ where $\delta_z:=(e^{-\frac{\chi_z+2\chi}{3}}-e^{-\frac{\chi_z+\chi}{2}})^\frac{3}{\beta}$.
By proposition \ref{Lambda}(1), $\forall t\in W^s$ $d(f^n(t),f^n(z))\leq4e^{-n\frac{\chi}{2}}$. Then
\begin{align*}
    |\widetilde{\psi}_{f^n(z)}^{-1}(f^n(t))|&\leq\mathrm{Lip}(\widetilde{\psi}_{f^n(z)}^{-1})\cdot d(f^n(t),f^n(z))\leq 2\|\widetilde{C}_\chi^{-1}(f^n(z))\|\cdot4e^{-n\frac{\chi}{2}}\\
    &\leq\eta_n^2\eta_0^{-2}e^{2n\epsilon}8\|\widetilde{C}_\chi^{-1}(f^n(z))\|e^{-n\frac{\chi}{2}}
    \leq\eta_n\cdot e^{-(\frac{\chi}{2}+2\epsilon)n}\cdot[(8\eta_0^{-2})\cdot (\eta_n\cdot\|\widetilde{C}_\chi^{-1}(f^n(z))\|)]\\
    &\leq\eta_n\cdot e^{-(\frac{\chi}{2}+2\epsilon)n}\cdot[8\eta_0^{-2}],\text{ because }\eta_n\ll\overline{Q}_\epsilon(f^n(z))\ll \|\widetilde{C}_\chi^{-1}(f^n(z))\|^{-1}.
\end{align*}
Hence, for $n_z$ large enough s.t. $e^{-(\frac{\chi}{2}+2\epsilon)n_z}\cdot[8\eta_0^{-2}]<1$, we get $\forall n>n_z$ $|\widetilde{\psi}_{f^n(z)}^{-1}(f^n(t))|\leq \eta_n$. So $\forall n>n_z$, $f^{n}[W^s]\subset \widetilde{\psi}_{f^n(z)}[R_{\eta_n}(0)]$. 

\medskip
\noindent Notice that $\widetilde{C}_\chi(z)=C_{\frac{\chi+\chi_z}{2}}(z)$; then the process of the Graph Transform can be done for the positive chain $\{\psi_{f^k(z)}^{\eta_k,\eta_k}\}_{k\geq0}$, as if $\chi$ had the alternative value $\frac{\chi_z+\chi}{2}$ (inequalities in the calculations require $\epsilon$ to be small enough w.r.t $\chi$, whence in particular also w.r.t $\frac{\chi_z+\chi}{2}>\chi$). This process admits stable manifolds $\widetilde{V}^s_n$ in $\widetilde{\psi}_{f^n(z)}^{\eta_n,\eta_n}$ $\forall n\geq0$. Then, $f^{n}[W^s]\subset \widetilde{V}^s_n, n>n_z$. Denote the representing function of $\widetilde{V}^s_n$ by $F_n$.

\medskip
\noindent\underline{Claim:} For any $k\geq0$, $y\in \widetilde{V}^s_{n_z}$, $u\in T_y\widetilde{V}^s_{n_z}(1)$:
$$|d_yf^ku|\leq2e^{-k\frac{2\chi+\chi_z}{3}}\|\widetilde{C}_\chi(f^{n_z}(z))^{-1}\|.$$
This claim is a subtle but crucial improvement of the bounds achieved in proposition \ref{Lambda}(2), as its uniform bound for the exponential rate of contraction is greater than $\chi$.

\medskip
\noindent\underline{Proof:} To simplify notations, we assume WLOG that $n_z=0$ (the proof remains the same), and we write $\widetilde{V}^s\equiv \widetilde{V}^s_{n_z}$ for short. Let $y\in \widetilde{V}^s$ and let $u\in T_y\widetilde{V}^s(1)$ be some unit vector tangent to $\widetilde{V}^s$. $f^k(y)\in \widetilde{W}_k^s\subset\widetilde{\psi}_{f^k(z)}[R_{\eta_k}(0)]$, so $d_yf^ku=d_{y_k'}\widetilde{\psi}_{f^k(z)}\Big(\begin{array}{c}
\xi_k\\
\eta_k\\
\end{array}\Big)$ where $\Big(\begin{array}{c}
\xi_k\\
\eta_k\\
\end{array}\Big)$ is tangent to the graph of $F_k$. Since $Lip(F_k)<\epsilon$, $|\eta_k|\leq\epsilon|\xi_k|$ for all $k$. $$\Big(\begin{array}{c}
\xi_{k+1}\\
\eta_{k+1}\\
\end{array}\Big)=\Big[\begin{pmatrix}D_{s,k}  &   \\  & D_{u,k}
\end{pmatrix}+\begin{pmatrix}d_{y_k'}h_s^{(k)} \\ d_{y_k'}h_u^{(k)}
\end{pmatrix}\Big]\Big(\begin{array}{c}
\xi_k\\
\eta_k\\
\end{array}\Big),$$
where the left terms on the RHS are block matrices as denoted (recall proposition \ref{3.4inomris}). By proposition \ref{3.4inomris} (for the ``new" $\chi\mapsto\frac{\chi+\chi_z}{2}$) we get,
$$|\xi_{k+1}|\leq(e^{-\frac{\chi+\chi_z}{2}}+2\cdot\epsilon\eta_k^\frac{\beta}{3})|\xi_k|\leq (e^{-\frac{\chi+\chi_z}{2}\chi}+\delta_z)|\xi_k|\leq e^{-k\frac{2\chi+\chi_z}{3}}|\xi_0|.$$
Returning to the defining relation $d_yf^ku=d_{y_k'}\widetilde{\psi}_{f^k(z)}\Big(\begin{array}{c}
\xi_k\\
\eta_k\\
\end{array}\Big)$, and recalling that $\|d_\cdot\widetilde{\psi}_{f^k(z)}\|\leq2$  we get
$|d_yf^ku|\leq2e^{-k\frac{2\chi+\chi_z}{3}}\|\widetilde{C}_\chi(z)^{-1}\|$ as wished. QED

This claim concludes our lemma in the following way: $$C_z^2:=\max_{y\in W^s,\eta\in T_yW^s(1)}2\sum_{j=0}^{n_z}|d_yf^j\eta|^2e^{2\chi j}+\sum_{j=n_z+1}^\infty4\|\widetilde{C}_\chi^{-1}(f^{n_z}(z))\|^2M_f^{2n_z}e^{-2\frac{2\chi+\chi_z}{3}j}e^{2\chi j}<\infty.$$
\end{proof}

\noindent\underline{Notice}: In a setup where we are given two double charts $\psi_x^{p^s,p^u}\rightarrow \psi_y^{q^s,q^u}$, if $V^s$ is an $s$-admissible manifold in $\psi_y^{q^s,q^u}$ which stays in windows, then also $\mathcal{F}(V^s)$ is an $s$-admissible manifold in $\psi_x^{p^s,p^u}$ which stays in windows.
For the next lemma, recall definition \ref{scalingfuncs} for $S(\cdot,\cdot)$.

\begin{lemma}\label{boundimprove} Let $\psi_x^{p^s,p^u}\rightarrow \psi_y^{q^s,q^u}$. Let $V^s$ be an $s$-admissible manifold in $\psi_y^{q^s,q^u}$ which stays in windows. Let $\mathcal{F}(V^s)$ be the Graph Transform of $V^s$ w.r.t to the edge $\psi_x^{p^s,p^u}\rightarrow \psi_y^{q^s,q^u}$. Let $q\in f[\mathcal{F}(V^s)]\subset V^s$.
Then
$\exists\pi_x:T_{f^{-1}(q)}\mathcal{F}(V^s)\rightarrow H^s(x),\pi_y:T_qV^s\rightarrow H^s(y)$ invertible linear transformations s.t. $\forall \rho\geq e^{\sqrt{\epsilon}}$: if $\frac{S(q,\lambda)}{S(y,\pi_y \lambda)}\in [\rho^{-1},\rho]$ then
$$\frac{S(f^{-1}(q),\xi)}{S(x,\pi_x\xi)}\in [e^{\frac{1}{6}Q_\epsilon(x)^{\beta/6}}\rho^{-1},e^{-\frac{1}{6}Q_\epsilon(x)^{\beta/6}}\rho],$$
where $\lambda=d_{f^{-1}(q)}f\xi$ and $\xi\in T_{f^{-1}(q)}\mathcal{F}(V^s)$ is arbitrary.
A similar statement holds for the unstable case. We will also see that $\frac{|\pi_x\xi|}{|\xi|}=e^{\pm2 Q_\epsilon(x)^{\beta/4}}$ for all $\xi\in T_{f^{-1}(q)}\mathcal{F}(V^s)$ (similarly for $\pi_y$).
\end{lemma}
\begin{figure}[h]
\caption{
}
\centerline{
\includegraphics[scale=0.25]{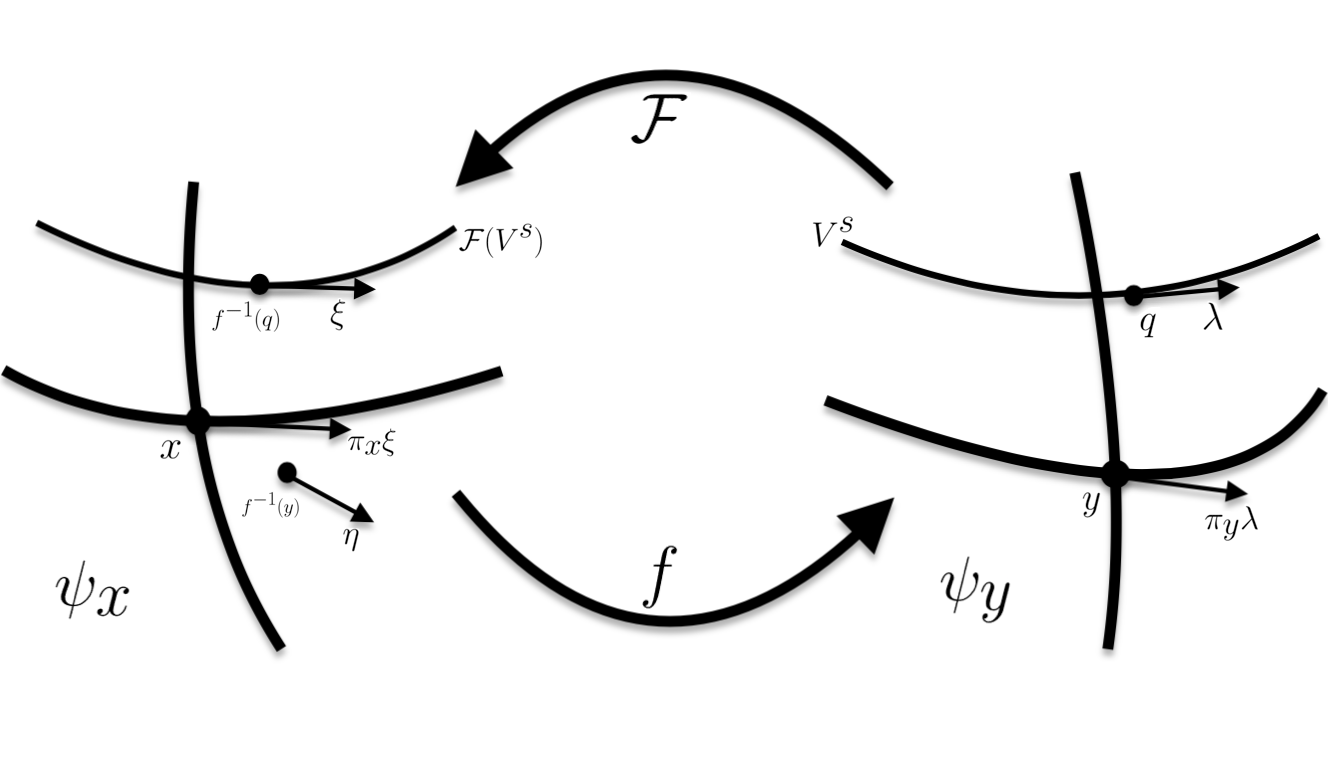}
}
\end{figure}

\begin{proof}  Recall the disks $\mathcal{D}=\{D\}$ and the isometries $\Theta_D:TD\rightarrow\mathbb{R}^d,\nu_a=\Theta_D^{-1}|_{T_aM}, a\in D\subset M$ from definition \ref{isometries}. If $\epsilon$ is smaller than the Lebesgue number of $\mathcal{D}$, then $\exists D$ which contains $x, f^{-1}(q), f^{-1}(y)$. Let $\Theta=\Theta_D$, and set $C_1:=\Theta\circ C_\chi(x),C_2:=\Theta\circ C_\chi(f^{-1}(y))$.
Let $F,G$ denote the representing functions of $V^s,\mathcal{F}(V^s)$, respectively. 

Let $\xi\in T_{f^{-1}(q)}\mathcal{F}(V^s)$. Write $t=\psi_x^{-1}(f^{-1}(q))=(t_s,G(t_s))$, and
$\xi=d_{t}\psi_x\Big(\begin{array}{c}
v\\
d_{t_s}Gv\\
\end{array}\Big)$ for a vector $v\in \mathbb{R}^{s(x)}$. We define
$\pi_x\xi:=d_{(0,0)}\psi_x\Big(\begin{array}{c}
v\\
0\\
\end{array}\Big)$ (
$\pi_y$ is defined analogously). Denote by $E_0$ the $|\cdot|_\infty$-Lipschitz constant of $(x,\underline{u},\underline{v})\mapsto[\Theta_D\circ d_{\underline{v}}\exp][\nu_x \underline{u}]$.

We begin with bounding the distortion of $\pi_{x}$ (assume WLOG\footnote{Otherwise cancel its size from both the denominator and numerator of the LHS of eq. \eqref{eq9}.} $|\xi|=1$):

\begin{align}\label{step1new}
|\Theta\xi-\Theta\pi_x\xi|=&|\Theta d_{(t_s,G(t_s))}\psi_x\Big(\begin{array}{c}
v\\
d_{t_s}Gv\\
\end{array}\Big)-\Theta d_{(0,0)}\psi_x\Big(\begin{array}{c}
v\\
0\\
\end{array}\Big)|\\
=&|\Theta \Big(d_{C_\chi(x)(t_s,G(t_s))}\exp_x\Big)\nu_x\circ C_1\Big(\begin{array}{c}
v\\
d_{t_s}Gv\\
\end{array}\Big)-\Theta \Big(d_{C_\chi(x)(0,0)}\exp_x\Big)\nu_x\circ C_1\Big(\begin{array}{c}
v\\
0\\
\end{array}\Big)|\nonumber\\
\leq& E_0\Big(\Big| C_1\Big( \Big(\begin{array}{c}
v\\
d_{t_s}Gv\\
\end{array}\Big) -\Big(\begin{array}{c}
v\\
0\\
\end{array}\Big)\Big)\Big|+\Big|C_\chi(x)\Big((t_s,G(t_s))-(0,0)\Big)\Big|\Big)\nonumber\\ \leq&E_0\Big(|v|\cdot\|d_{t_s}G\|+ |\exp_x^{-1}f^{-1}(q)|\Big)= E_0\Big(|v|\cdot\|d_{t_s}G\|+ d(f^{-1}(q),x)\Big)(\because C_1\text{ is a contraction})\nonumber\\
\leq&4E_0 Q_\epsilon(x)^{\beta/3}\|C_\chi^{-1}(x)\|\leq Q_\epsilon(x)^{\beta/4}\text{ }(\because |v|\leq 2\|C_\chi^{-1}(x)\|,\|d_{t_s}G\|\leq \frac{3}{2}Q_\epsilon(x)^{\beta/3}).\nonumber
\end{align} 
Therefore $\frac{|\pi_x\xi|}{|\xi|}=e^{\pm2Q_\epsilon(x)^\frac{\beta}{4}}$  (similarly, $\frac{|\pi_y\xi'|}{|\xi'|}=e^{\pm2Q_\epsilon(y)^\frac{\beta}{4}}, \xi'\in T_q V^s\setminus\{0\}$). We continue to decompose:
\begin{equation}\label{multipliers}\frac{S(x,\pi_x\xi)}{S(f^{-1}(q),\xi)}=\frac{S(f^{-1}(y),\eta)}{S(f^{-1}(q),\xi)}\cdot\frac{S(x,\pi_x \xi)}{S(f^{-1}(y),\eta)},\end{equation}
where $\eta:= d_y (f^{-1})\pi_y d_{f^{-1}(q)}f\xi\in H^s(f^{-1}(y))$  ($\|\pi_y\|\leq e^{2Q_\epsilon^\frac{\beta}{4}(y)}$, hence $|\eta|\leq2M_f^2$).
We study the second factor. Denote
$$u_1:=C_\chi^{-1}(x)\xi,u_2:=C_\chi^{-1}(f^{-1}(y))\eta.$$
By the definitions of $S(\cdot,\cdot)$ and $C_\chi(\cdot)$ (definition \ref{scalingfuncs}, theorem \ref{pesinreduction}), the inverse of the second factor satisfies:
\begin{align}
\Big|\frac{S(f^{-1}(y),\eta)}{S(x,\pi_x\xi)}-1\Big|=&\Big|\frac{|C_\chi^{-1}(f^{-1}(y))\eta|}{|C_\chi^{-1}(x)\pi_x\xi|}-1\Big|=\Big|\frac{|C_2^{-1}\Theta\eta|}{|C_1^{-1}\Theta\pi_x\xi|}-1\Big|\label{eq9}\\ \leq&\Big|\frac{|C_2^{-1}\Theta\eta-C_1^{-1}\Theta\eta|+|C_1^{-1}\Theta\eta-C_1^{-1}\Theta\pi_x\xi|+|C_1^{-1}\Theta\pi_x\xi|}{|C_1^{-1}\Theta\pi_x\xi|}-1\Big|\nonumber\\
\leq&\frac{1}{|C_\chi^{-1}(x)\pi_x\xi|}\Big(\|C_2^{-1}-C_1^{-1}\|\cdot|\eta|+\|C_1^{-1}\|\cdot|\Theta\pi_x\xi-\Theta \eta|\Big)\nonumber\\
\leq&2\Big(2\epsilon \cdot2M_f^2Q_\epsilon(x)Q_\epsilon(f^{-1}(y))+\|C_\chi^{-1}(x)\|\cdot|\Theta\pi_x\xi-\Theta\eta|\nonumber\Big)\\
(\because&\|C_\chi(x)\|\leq1\text{, lemma \ref{overlap}},|\eta|\leq 2M_f^2,|\pi_x\xi|\geq e^{-2Q_\epsilon^\frac{\beta}{4}(x)}\geq\frac{1}{2}).\nonumber
\end{align}
By the definition of $\omega_0$ (lemma \ref{omega0}), $\frac{Q_\epsilon(f^{-1}(y))}{Q_\epsilon(y)}\leq \omega_0$. Substituting this gives:
\begin{align}\label{thirdterm}
\Big|\frac{S(f^{-1}(y),\eta)}{S(x,\pi_x\xi)}-1\Big|\leq&2\cdot\epsilon\cdot4M_f^2\omega_0Q_\epsilon(x)Q_\epsilon(y)+2\cdot\|C_\chi^{-1}(x)\|\cdot|\Theta\pi_x\xi-\Theta\eta|\nonumber\\
\leq&\frac{1}{4}Q_\epsilon(x)Q_\epsilon(y)+2\|C_\chi^{-1}(x)\|\cdot|\Theta\pi_x\xi-\Theta\eta|,\end{align}
for small enough $\epsilon$. Next we bound $|\Theta\pi_x\xi-\Theta\eta|$:

\medskip
\textit{Step 1}: $|\Theta\xi-\Theta\pi_x\xi|\leq Q_\epsilon(x)^{\beta/4}$ by \eqref{step1new}.


\medskip
\textit{Step 2}: Estimation of $|\Theta\xi-\Theta \eta|$. Let $\lambda:=d_{f^{-1}(q)}f\xi$ and $\hat{\lambda}=\frac{\lambda}{|\lambda|}$ \text{ }\footnote{$\hat{\lambda}$ is well defined since $f$ is a diffeomorphism and $|\xi|\neq0$.}.
Then
\begin{align*}
|\Theta\eta-\Theta\xi|\equiv&|\Theta d_y(f^{-1})\pi_yd_{f^{-1}(q)}f\xi
-\Theta\xi|=|\Theta d_y(f^{-1})\pi_yd_{f^{-1}(q)}f\xi-\Theta d_q(f^{-1})d_{f^{-1}(q)}f\xi|\\
=&|\Theta d_y(f^{-1})\pi_y\lambda-\Theta d_q(f^{-1})\lambda|=|\Theta d_y(f^{-1})\nu_y\Theta'\pi_y\lambda-\Theta d_q(f^{-1})\nu_q\Theta'\lambda|\\
\leq&|\Theta d_y(f^{-1})\nu_y\Theta'\pi_y\lambda-\Theta d_q(f^{-1})\nu_q\Theta'\pi_y\lambda|+|\Theta d_q(f^{-1})\nu_q\Theta'\pi_y\lambda-\Theta d_q(f^{-1})\nu_q\Theta'\lambda|\\
\leq &\|\Theta d_y(f^{-1})\nu_y-\Theta d_q(f^{-1})\nu_q\|\cdot|\pi_y\lambda|+M_f|\Theta'\pi_y\lambda-\Theta'\lambda|\\
\leq &H_0\cdot d(q,y)^\beta|\pi_y\lambda|+M_f|\Theta'\pi_y\hat{\lambda}-\Theta'\hat{\lambda}|\cdot|\lambda|,    
\end{align*}
where $\Theta'$ is the linear isometry for the neighborhood containing $y$ and $q$, as in definition \ref{isometries}, and $H_0$ is the bound for H\"ol$(d_\cdot f^{-1})$ as in the proof of proposition \ref{Lambda}. 
$|\lambda|\leq M_f$. As in \eqref{step1new}, $|\Theta'\pi_y\hat{\lambda}-\Theta'\hat{\lambda}|\leq Q_\epsilon(y)^{\beta/4}$, and hence also $|\pi_y\lambda|=|\pi_y\hat{\lambda}|\cdot|\lambda|\leq  (1+Q_\epsilon(y)^{\beta/4})\cdot M_f<2M_f$. Therefore,
$$|\Theta \eta
-\Theta\xi|\leq M_f\cdot(H_0\cdot Q_\epsilon(y)^{\beta}\cdot2+M_f\cdot Q_\epsilon(y)^{\beta/4}\cdot2)<2M_f^2(H_0+1)Q_\epsilon(y)^{\beta/4}.$$
Since $\psi_x^{p^s,p^u}\rightarrow \psi_y^{q^s,q^u}$,  $\psi_{f(x)}^{p^s\wedge p^u}$ and $\psi_y^{q^s\wedge q^u}$ $\epsilon$-overlap. Hence lemma \ref{overlap} gives us the bound \begin{align*}\frac{Q_\epsilon(y)}{Q_\epsilon(x)}=&\frac{Q_\epsilon(f(x))}{Q_\epsilon(x)}\cdot\frac{Q_\epsilon(y)}{Q_\epsilon(f(x))}\leq\omega_0 \frac{Q_\epsilon(y)}{Q_\epsilon(f(x))}\leq\omega_0 \Big(\frac{\|C_\chi^{-1}(y)\|}{\|C_\chi^{-1}(f(x))\|}\Big)^\frac{-48}{\beta}\\
\leq&\omega_0 (e^{Q_\epsilon(f(x))Q_\epsilon(y)})^\frac{-48}{\beta}\leq\omega_0 e^\epsilon\text{ (for small enough }\epsilon\text{)}.
\end{align*}
Thus
$$|\Theta\eta
-\Theta\xi|\leq2(H_0+1)(e^\epsilon\omega_0)^\frac{\beta}{4}M_f^2Q_\epsilon(x)^{\beta/4}\leq3(H_0+1)\omega_0M_f^2Q_\epsilon(x)^{\beta/4}.$$

\medskip
\textit{Step 3}: Adding the estimates in steps 1 and 2, we obtain
\begin{align*}|\Theta\pi_x\xi-\Theta\eta|\leq&|\Theta\pi_x\xi-\Theta\xi|+|\Theta\xi-\Theta\eta|\\
\leq &Q_\epsilon(x)^{\beta/4}(1+3(H_0+1)\omega_0M_f^2)\leq4(H_0+1)\omega_0M_f^2Q_\epsilon(x)^{\beta/4}.\end{align*}

Substituting this in \eqref{thirdterm} gives (for small enough $\epsilon$):
$$\Big|\frac{S(f^{-1}(y),\eta)}{S(x,\pi_x\xi)}-1\Big|\leq\frac{1}{4}Q_\epsilon(x)Q_\epsilon(y)+2\|C_\chi^{-1}(x)\|\cdot4(H_0+1)\omega_0M_f^2Q_\epsilon(x)^\frac{\beta}{4}\leq\frac{1}{6}Q_\epsilon(x)^\frac{\beta}{6}
\text{ and hence, }\frac{S(f^{-1}(y),\eta)}{S(x,\pi_x\xi)}= e^{\pm\frac{1}{3}Q_\epsilon(x)^\frac{\beta}{6}}.$$ This concludes the study of the second factor of the RHS of \eqref{multipliers}.

\medskip
It remains to study the first factor of the RHS of \eqref{multipliers}. We begin with the identity

\begin{align*}S^2(f^{-1}(y),\eta)=&2\sum_{m=0}^{\infty}|d_{f^{-1}(y)}f^m\eta|^2e^{2\chi m}
=2\Big(|\eta|^2+\sum_{m=1}^\infty|d_yf^{m-1}d_{f^{-1}(y)}f\eta|^2e^{2\chi m}\Big)
=2|\eta|^2+e^{2\chi}S^2(y,d_{f^{-1}(y)}f\eta).\end{align*}
Similarly, and by the assumption of the lemma 
:
\begin{align*}S^2(f^{-1}(q),\xi)=&2|\xi|^2+e^{2\chi}S^2(q,d_{f^{-1}(q)}f\xi)\leq2|\xi|^2+\rho^2 e^{2\chi}S^2(y,\pi_yd_{f^{-1}(q)}f\xi).\end{align*}
By the choice of $\eta$, $d_{f^{-1}(y)}f\eta=\pi_yd_{f^{-1}(q)}f\xi$, so
$$S^2(f^{-1}(q),\xi)\leq2|\xi|^2+\rho^2e^{2\chi}S^2(y,d_{f^{-1}(y)}f\eta)$$
and so,
\begin{align}\frac{S^2(f^{-1}(q),\xi)}{S^2(f^{-1}(y),\eta)}\leq&\frac{2|\xi|^2+\rho^2e^{2\chi}S^2(y,d_{f^{-1}(y)}f\eta)}{2|\eta|^2+e^{2\chi}S^2(y,d_{f^{-1}(y)}f\eta)}=
\rho^2-\frac{2(\rho^2-1)|\eta|^2+2(|\eta|^2-|\xi|^2)}{2|\eta|^2+e^{2\chi}S^2(y,d_{f^{-1}(y)}f\eta)}\nonumber\\=&\rho^2-\frac{2(\rho^2-1)|\eta|^2+2(|\eta|^2-|\xi|^2)}{S^2(f^{-1}(y),\eta)}\label{kofiko}.\end{align}

By the bound achieved in step 2, and since $|\xi|=1$, $|\eta|=e^{\pm Q_\epsilon(x)^{\beta/5}}$ (for small enough $\epsilon$). Hence $\Big||\xi|^2-|\eta|^2\Big|\leq2(e^{ Q_\epsilon(x)^{\beta/5}}-1)\leq4Q_\epsilon(x)^{\beta/5}$. We get that for $\epsilon$ small enough, since $Q_\epsilon(x)\ll \epsilon$,
 \begin{align*}\rho^2-1\geq e^{2\sqrt{\epsilon}}-1\geq2\sqrt{\epsilon}\Rightarrow&
 (\rho^2-1)|\eta|^2+(|\xi|^2-|\eta|^2)\geq(\rho^2-1)|\eta|^2-4Q_\epsilon(x)^{\beta/5}=(\rho^2-1)[|\eta|^2-\frac{4Q_\epsilon(x)^{\beta/5}}{\rho^2-1}]\\
 \geq&(\rho^2-1)[e^{-Q_\epsilon(x)^\frac{\beta}{5}}-\frac{4Q_\epsilon(x)^{\beta/5}}{2\sqrt{\epsilon}}]\geq(\rho^2-1)[1-\frac{1}{20}Q_\epsilon(x)^{\beta/6}]
 \geq (\rho^2-1)e^{-\frac{1}{10}Q_\epsilon(x)^{\beta/6}}.
 \end{align*}
 Plugging this in \eqref{kofiko} yields
\begin{align*}\frac{S^2(f^{-1}(q),\xi)}{S^2(f^{-1}(y),\eta)}\leq\rho^2-\frac{2(\rho^2-1)e^{-\frac{1}{10}Q_\epsilon(x)^{\frac{\beta}{6}}}}{S^2(f^{-1}(y),\eta)}\leq\rho^2-\frac{2(\rho^2-1)e^{-\frac{1}{10}Q_\epsilon(x)^{\frac{\beta}{6}}}e^{-\frac{2}{3}Q_\epsilon(x)^{\frac{\beta}{6}}}}{S^2(x,\pi_x\xi)},
\end{align*}
where the last inequality is because, as we saw above, $\frac{S(f^{-1}(y),\eta)}{S(x,\pi_x\xi)}=e^{\pm\frac{1}{3}Q_\epsilon(x)^\frac{\beta}{6}}$ 
. So
$$\frac{S^2(f^{-1}(q),\xi)}{S^2(f^{-1}(y),\eta)}\leq\rho^2-\frac{2(\rho^2-1)e^{-\frac{2}{3}Q_\epsilon(x)^{\beta/6}}e^{-\frac{1}{10}Q_\epsilon(x)^{\frac{\beta}{6}}}}{\|C_\chi^{-1}(x)\|^2\cdot|\pi_x\xi|^2}.$$
Substituting $|\xi|=1$, $\frac{|\pi_x\xi|}{|\xi|}= e^{\pm2 Q_\epsilon(x)^{\beta/4}}$ gives:
$$\frac{S^2(f^{-1}(q),\xi)}{S^2(f^{-1}(y),\eta)}\leq\rho^2-\frac{2(\rho^2-1)e^{-\frac{2}{3}Q_\epsilon(x)^{\beta/6}}e^{-\frac{1}{10}Q_\epsilon(x)^{\beta/6}}}{\|C_\chi^{-1}(x)\|^2\cdot e^{4 Q_\epsilon(x)^{\beta/4}}}\leq\rho^2\Big(1-\frac{2(1-\frac{1}{\rho^2})e^{-Q_\epsilon(x)^{\beta/6}}}{\|C_\chi^{-1}(x)\|^2}\Big)$$
$\Big(\because$ $4Q_\epsilon(x)^\frac{\beta}{4}\leq(1-\frac{2}{3}-\frac{1}{10})Q_\epsilon(x)^\frac{\beta}{6}$, for $\epsilon$ small enough$\Big)$. For $\rho\geq e^{\sqrt{\epsilon}}$, $1-\frac{1}{\rho^2}\geq1-e^{-2\sqrt{\epsilon}}\geq\sqrt{\epsilon}$, so:
$$\frac{S^2(f^{-1}(q),\xi)}{S^2(f^{-1}(y),\eta)}\leq\rho^2\Big(1-\frac{2\sqrt{\epsilon}e^{-Q_\epsilon(x)^{\beta/6}}}{\|C_\chi^{-1}(x)\|^2}\Big)\leq\rho^2\Big(1-\frac{\epsilon^\frac{1}{2}}{\|C_\chi^{-1}(x)\|^{2}}\Big).$$
By definition, $Q_\epsilon(x)<\frac{1}{3^{6/\beta}}(\frac{\epsilon^{90}}{\|C_\chi^{-1}(x)\|^{48}})^\frac{1}{\beta}<\frac{1}{3^{6/\beta}}(\frac{\epsilon^3}{\|C_\chi^{-1}(x)\|^{12}})^\frac{1}{\beta}\leq\frac{1}{3^{6/\beta}}(\frac{\epsilon^\frac{1}{2}}{\|C_\chi^{-1}(x)\|^2})^{6/\beta}$. Hence,
$$\frac{S^2(f^{-1}(q),\xi)}{S^2(f^{-1}(y),\eta)}\leq\rho^2(1-3Q_\epsilon(x)^{\beta/6})\leq\rho^2 e^{-3Q_\epsilon(x)^{\beta/6}}.$$
Similarly, $\frac{S^2(f^{-1}(q),\xi)}{S^2(f^{-1}(y),\eta)}\geq\rho^{-2}e^{3Q_\epsilon(x)^\frac{\beta}{6}}$.

\medskip
In summary, $\frac{S^2(x,\pi_x\xi)}{S^2(f^{-1}(q),\xi)}$ is the product of two factors (see \eqref{multipliers}). The second factor is bounded by $e^{\pm\frac{1}{3}Q_\epsilon(x)^\frac{\beta}{6}}$, 
and the first factor takes values in $[\rho^{-2}e^{3Q_\epsilon(x)^\frac{\beta}{6}},\rho^2e^{-3Q_\epsilon(x)^\frac{\beta}{6}}]$.
Therefore, $\frac{S^2(f^{-1}(q),\xi)}{S^2(f^{-1}(y),\eta)}\leq\rho^2 e^{-3Q_\epsilon(x)^{\frac{\beta}{6}}+\frac{1}{3}Q_\epsilon(x)^{\frac{\beta}{6}}}$ $\leq\rho^2 \Big(e^{-(\frac{3}{2}-\frac{1}{6})Q_\epsilon(x)^{\frac{\beta}{6}}}\Big)^2$ $\leq \rho^2\Big(e^{-\frac{1}{6}Q_\epsilon(x)^\frac{\beta}{6}}\Big)^2$ and $\frac{S^2(f^{-1}(q),\xi)}{S^2(f^{-1}(y),\eta)}\geq\rho^{-2}\Big(e^{\frac{1}{6}Q_\epsilon(x)^\frac{\beta}{6}}\Big)^2$ for small enough $\epsilon$. This concludes the proof.


\end{proof}
\begin{lemma}\label{chainbounds}
The following holds for any $\epsilon$ small enough. For any two regular chains $(\psi_{x_i}^{p^s_i,p^u_i})_{i\in\mathbb{Z}},(\psi_{y_i}^{q^s_i,q^u_i})_{i\in\mathbb{Z}}$, if $\pi[(\psi_{x_i}^{p^s_i,p^u_i})_{i\in\mathbb{Z}}]=\pi[(\psi_{y_i}^{q^s_i,q^u_i})_{i\in\mathbb{Z}}]=p$ then $\forall k\in\mathbb{Z}$, and for any $\xi\in T_{f^k(p)}V^s_k$,
$$\frac{S(x_k,\pi_{x,k}\xi)}{S(y_k,\pi_{y,k}\xi)}= e^{\pm4\epsilon^{1/2}}$$
 where $\pi_{x,k},\pi_{y,k}$ denote respectively the maps $\pi_{x_k}:T_{f^k(p)}V^s_k\rightarrow H^s(x_k),\pi_{y_k}:T_{f^k(p)}U^s_k\rightarrow H^s(y_k)$ as in lemma \ref{boundimprove},
and $V^s_k=V^s((\psi_{x_i}^{p_i^s,p_i^u})_{i\geq k}),U^s_k=V^s((\psi_{y_i}^{q_i^s,q_i^u})_{i\geq k})$.
A similar statement holds for the unstable manifold.
\end{lemma}
\noindent\textbf{Remark}: Notice that $T_{f^k(p)}V^s_k=T_{f^k(p)}U^s_k$, which makes the choice and use of $\xi$ well defined. 
This can be seen by the facts that the orbit $(f^k(p))_{k\in\mathbb{Z}}$ stays on the intersection of stable and unstable manifolds, and demonstrates hyperbolic behavior (proposition \ref{Lambda}(2) gives lower bounds for the contraction/expansion by the differential on the respective manifolds' tangent spaces). Thus, the 
stable space of $f^k(p)$ (which exists by proposition \ref{Lambda}(2))
coincides with $T_{f^k(p)}V^s_k$, and $T_{f^k(p)}V^s_k=\text{stable space of }f^k(p)=T_{f^k(p)}U^s_k$.

\begin{proof} (For a proof in the two dimensional case, see proposition 7.3 in \cite{Sarig}) Denote $\vec{v}=(\psi_{x_i}^{p^s_i,p^u_i})_{i\in\mathbb{Z}}$ and $\vec{u}=(\psi_{y_i}^{q^s_i,q^u_i})_{i\in\mathbb{Z}}$. $V^s_k$ and $U^s_k$ are $s$-admissible manifolds which stay in windows. We will prove that:
\begin{equation}\label{thatexpression}\forall\text{ } \xi\in T_{f^k(p)}V^s_k=T_{f^k(p)}U^s_k:\text{ }\frac{S(f^k(p),\xi)}{S(x_k,\pi_{x,k}\xi)},\frac{S(f^k(p),\xi)}{S(y_k,\pi_{y,k}\xi)}= e^{\pm\sqrt{\epsilon}}.
\end{equation}
This is sufficient, since 
$$\frac{S(x_k,\pi_{x,k}\xi)}{S(y_k,\pi_{y,k}\xi)}=\frac{S(x_k,\pi_{x,k}\xi)}{S(f^k(p),\xi)}\cdot\frac{S(f^k(p),\xi)}{S(y_k,\pi_{y,k}\xi)}
.$$
We show $\frac{S(f^k(p),\xi)}{S(x_k,\pi_{x,k}\xi)}=e^{\pm\sqrt{\epsilon}}$; the case of $\frac{S(f^k(p),\xi)}{S(y_k,\pi_{y,k}\xi)}$ is identical. Since $\vec{v}$ is regular, there exists a relevant double chart $v$ and a sequence $n_k\uparrow \infty$ s.t. $v_{n_k}=v$ for all $k$. Write $v=\psi_x^{p^s,p^u}$.

\medskip
\noindent\textit{Claim 1}:  $$\exists\rho\geq e^{\sqrt{\epsilon}}\text{ s.t. }\frac{S(f^{n_k}(p),\xi)}{S(x_{n_k},\pi_{x,n_k}\xi)}\in[\rho^{-1},\rho]\text{ for all }k\in\mathbb{Z},\xi\in T_{f^{n_k}(p)}V^s(1).$$
\textit{Proof}: 
Choose a chain $\vec{w}$ s.t. $w_0=v$ and $z:=\pi(\vec{w})\in NUH_\chi^\#$. Let $W^s:=V^s((w_i)_{i\geq0})$.
$W^s$ is an admissible manifold in $v_{n_k}$, so by taking $W^s$ in $v_{n_{k+l}}\equiv v$ and applying  $\mathcal{F}_s$ $(n_{k+l}-n_k)$ times, the resulting manifolds $W^s_l:=\mathcal{F}_s^{n_{k+l}-n_k}(W^s)$ are $s$-admissible manifolds in $v_{n_k}$ which converge to $V^s_{n_k}$ in $C^1$ (proposition \ref{firstofchapter}, proof of item 2).
The convergence of $W^s_l$ to $V^s_{n_k}$ means that if $W^s_l$ is represented in $v_{n_k}$ by $F_l$, and $V^s_{n_k}$ is represented by $F$, then $\|F_l-F\|_{C^1}\rightarrow0$. Write $f^{n_k}(p)=\psi_x((r,F(r)))$, and define $t_l:=\psi_x((r,F_l(r)))\in W^s_l$; then $t_l\rightarrow f^{n_k}(p)$.
Since $f^{n_{k+l}-n_k}(t_l)\in W^s$ $\forall l$, also $\sup_l\max_{\eta\in T_{f^{n_{k+l}-n_k}(t_l)}W^s(1)}S(f^{n_{k+l}-n_k}(t_l),\eta)\leq C_z<\infty$, by lemma \ref{RegularOnVs}. Therefore the following quantity is finite and well defined:
$$\rho_1:=\max\{\sup_l\max_{\eta\in T_{f^{n_{k+l}-n_k}(t_l)}W^s(1)}\frac{S(f^{n_{k+l}-n_k}(t_l),\eta)}{S(x,\pi_x\eta)},\sup_l\max_{\eta\in T_{f^{n_{k+l}-n_k}(t_l)}W^s(1)}\frac{S(x,\pi_x\eta)}{S(f^{n_{k+l}-n_k}(t_l),\eta)},e^{\sqrt{\epsilon}}\},$$
where $\pi_x$ is the linear transformation defined in lemma \ref{boundimprove} respective to $W^s$. By lemma \ref{boundimprove}, these bounds can only improve, then the following quantity is also finite and well defined:
$$\rho_0:=\max\{\sup_l\max_{\eta\in T_{t_l}W^s_l(1)}\frac{S(t_l,\eta)}{S(x,\pi_x^{(l)}\eta)},\sup_l\max_{\eta\in T_{t_l}W^s_l(1)}\frac{S(x,\pi_x^{(l)}\eta)}{S(t_l,\eta)},e^{\sqrt{\epsilon}}\},$$
where $\pi_x^{(l)}:T_{t_l}W_l^s\rightarrow H^s(x)$ is as in the proof of lemma \ref{boundimprove}. 
 Also, since $F_l\xrightarrow[]{C^1} F$, and $\|d_\cdot F_l\|_{\beta/3}\leq\frac{1}{2}$ uniformly in $l$: $\|d_\cdot F_l-d_\cdot F\|_\infty\rightarrow 0$. Then 
\begin{align*}
&\sup_{|u|=1}|d_{(r,F(r))}\psi_x((u,d_rFu))-d_{(r,F_l(r))}\psi_x((u,d_rF_lu))|\\
&\leq\mathrm{Lip}(\psi_x)\cdot|F(r)-F_l(r)|\cdot |(u,d_rFu)| 
+\|d_{(r,F_l(r))}\psi_x\|\cdot\|\|d_\cdot F-d_\cdot F_l\|_\infty\\
&\leq3\|F-F_l\|_\infty+2\cdot\|d_\cdot F-d_\cdot F_l\|_\infty\rightarrow 0.
\end{align*}

\noindent For any $\xi\in T_{f^{n_k}(p)}V^s_{n_k}$ there exists $u\in\mathbb{R}^{s(x)}$ s.t. $\xi=d_{(r,F(r))}\psi_x((u,d_rFu))$.

\noindent 
Define
$\xi_l:=d_{(r,F_l(r))}\psi_x((u,d_rF_lu))\in T_{t_l}W^s_l$.
Then $\xi_l\rightarrow \xi$ (in the metric of $TM$), and similarly: $\pi_x^{(l)}\xi_l\rightarrow \pi_x \xi$ (in the Riemannian metric on $T_xM$). Fix some large $N\in\mathbb{N}$, and some small $\delta>0$. Since $df$ is continuous, $\exists l$ s.t.
$$\sqrt{2}\Big(\sum_{j=0}^Ne^{2j\chi}|d_{f^{n_k}(p)}f^j\xi|^2\Big)^{\frac{1}{2}}\leq e^\delta\sqrt{2}\Big(\sum_{j=0}^Ne^{2j\chi}|d_{t_l}f^j\xi_l|^2\Big)^{\frac{1}{2}}.$$
The RHS is $\leq e^\delta S(t_l,\xi_l)\leq e^\delta S(x,\pi_x^{(l)}\xi_l)\rho_0$ (by lemma \ref{boundimprove} and the choice of $\rho_0$). In theorem \ref{SandU} we have seen that $S(x,\cdot)$ is continuous, so: $\lim_{l\rightarrow \infty}e^{\delta}S(x,\pi_x^{(l)}\xi_l)\rho_0=e^{\delta}S(x,\pi_x\xi)\rho_0$. Since the inequality would be true for any $l'\geq l$ we get, $\sqrt{2}\Big(\sum_{j=0}^Ne^{2j\chi}|d_{f^{n_k}(p)}f^j\xi|^2\Big)^{\frac{1}{2}}\leq e^{\delta}S(x,\xi)\rho_0$. Thus, since this is true for any small $\delta>0$ and large $N\in\mathbb{N}$:
$$\frac{S(f^{n_k}(p),\xi)}{S(x,\pi_x\xi)}\leq \rho_0.$$
Recalling $x_{n_k}=x$ and that $S(f^{n_k}(p),\xi)\geq\sqrt{2}|\xi|$ $\forall\xi\neq0$, we can assume WLOG $|\xi|=1$, otherwise we cancel its size in the numerator and the denominator of the fraction:
$$\frac{S(f^{n_k}(p),\xi)}{S(x_{n_k},\pi_{x,n_k}\xi)}\in[\frac{\sqrt{2}}{\max\limits_{|\xi|=1}S(x,\pi_{x,n_k}\xi)},\rho_0].$$
Since we have seen in the previous lemma that $\frac{|\pi_x\xi|}{|\xi|}\leq e^{2 Q_\epsilon(x_k)^{\beta/4}}\leq e^{\sqrt{\epsilon}}$, and $\pi_x$ is invertible:
$$\frac{S(f^{n_k}(p),\xi)}{S(x_{n_k},\pi_{x,n_k}\xi)}\in[\frac{\sqrt{2}}{\max\limits_{|\eta|=1}e^{\sqrt{\epsilon}}S(x,\eta)},\rho_0]\Rightarrow\frac{\sqrt{2}}{\max\limits_{|\eta|=1}e^{\sqrt{\epsilon}}S(x,\eta)}\leq\min_\xi\frac{S(f^{n_k}(p),\xi)}{S(x_{n_k},\pi_{x,n_k}\xi)}\leq\max\limits_\xi\frac{S(f^{n_k}(p),\xi)}{S(x_{n_k},\pi_{x,n_k}\xi)}\leq\rho_0.$$
Hence, the claim would work with $\rho:=\rho_0\cdot e^{\sqrt{\epsilon}}\max\limits_{|\eta|=1}S(x,\eta)$. This concludes the proof of claim 1.

\medskip
\noindent\textit{Claim 2}: $$\forall \xi\in T_{p}V^s_0(1):\frac{S(p,\xi)}{S(x_0,\pi_{x,0}\xi)}= e^{\pm\sqrt{\epsilon}}.$$
\textit{Proof}: Fix $k$ large. By claim 1, $\forall \xi\in T_{f^{n_k}(p)}V^s_{n_k}(1):\frac{S(f^{n_k}(p),\xi)}{S(x_{n_k},\pi_{x,n_k}\xi)}\in[\rho^{-1},\rho]$. By proposition \ref{firstofchapter} $\mathcal{F}_s(V^s_{n_k})=V^s_{n_k-1}$; and thus by lemma \ref{boundimprove}: the bounds for $\max\setminus\min_{\xi\in T_{f^{n_{k-1}}(p)}V^s_{n_k}(1)}\frac{S(f^{n_k-1}(p),\xi)}{S(x_{n_k-1},\pi_{x,n_k-1}\xi)}$ improve or fall in $[e^{-\sqrt{\epsilon}},e^{\sqrt{\epsilon}}]$
. We ignore those improvements, and write $\frac{S(f^{n_k-1}(p),\xi)}{S(x_{n_k},\pi_{x,n_k-1}\xi)}\in[\rho^{-1},\rho]$. We continue this way until we get $\frac{S(f^{n_{k-1}+1}(p),\xi)}{S(x_{n_{k-1}+1},\pi_{x,n_{k-1}+1}\xi)}\in[\rho^{-1},\rho]$. Since $x_{n_k}=x$, the next application of $\mathcal{F}_s$ improves the ratio bound by at least $e^{\frac{1}{6}Q_\epsilon(x)^{\beta/6}}$:
$$\frac{S(f^{n_{k-1}}(p),\xi)}{S(x_{n_{k-1}},\pi_{x,n_{k-1}}\xi)}\in[e^{\frac{1}{6}Q_\epsilon(x)^{\beta/6}}\rho^{-1},e^{-\frac{1}{6}Q_\epsilon(x)^{\beta/6}}\rho],\forall \xi\in T_{f^{n_{k-1}}(p)}V^s_{n_{k-1}}(1).$$
We repeat the procedure by applying $\mathcal{F}_s$ $(n_{k-1}-n_{k-2}+1)$ times, whilst ignoring the potential improvements of the error bounds. Then we apply $\mathcal{F}_s$ once more and arrive at 
$$\frac{S(f^{n_{k-2}}(p),\xi)}{S(x_{n_{k-2}},\pi_{x,n_{k-2}}\xi)}\in[e^{2\frac{1}{6}Q_\epsilon(x)^{\beta/6}}\rho^{-1},e^{-2\frac{1}{6}Q_\epsilon(x)^{\beta/6}}\rho],\forall \xi\in T_{f^{n_{k-2}}(p)}V^s_{n_{k-2}}(1).$$
By repeating this process $k_0$ times, where $k_0$ is large enough so $e^{k_0\frac{1}{6}Q_\epsilon(x)^{\beta/6}}>\rho e^{-\sqrt{\epsilon}}$,  
we get:
$$\frac{S(f^{n_{k-k_0}}(p),\xi)}{S(x_{n_{k-k_0}},\pi_{x,n_{k_0}}\xi)}\in[e^{-\sqrt{\epsilon}},e^{\sqrt{\epsilon}}],\forall \xi\in T_{f^{n_{k-k_0}}(p)}V^s_{n_k}(1).$$
This is the threshold of applicability of the previous lemma. 
From this point onward, the ratio bound may not improve, but it is guaranteed to remain at least as good as $[e^{-\sqrt{\epsilon}},e^{\sqrt{\epsilon}}]$. We are free to choose $k$ as large as we want, therefore 
we obtain $\frac{S(p,\xi)}{S(x_0,\pi_{x,0}\xi)}= e^{\pm\sqrt{\epsilon}},\forall \xi\in T_{p}V^s_{0}(1)$. This proves claim 2. Claim 2 gives the lemma for $k=0$. For other $k$, apply claim 2 to $(\psi_{x_{i+k}}^{p^s_{i+k},p^u_{i+k}})_{i\in\mathbb{Z}}$,$(\psi_{y_{i+k}}^{q^s_{i+k},q^u_{i+k}})_{i\in\mathbb{Z}}$.
\end{proof}

\medskip
Our next task is to show that if $\pi((\psi_{x_{i}}^{p^s_{i},p^u_{i}})_{i\in\mathbb{Z}})=\pi((\psi_{y_{i}}^{q^s_{i},q^u_{i}})_{i\in\mathbb{Z}})$ then $C_\chi(x_i)^{-1}:T_{x_i}M\rightarrow\mathbb{R}^d,C_\chi(y_i)^{-1}:T_{y_i}M\rightarrow\mathbb{R}^d$ are ``approximately the same." There are two issues in expressing this formally:
\begin{enumerate}[label=(\alph*)]
\item $C_\chi(x_i)^{-1},C_\chi(y_i)^{-1}$ are linear maps with different domains.
\item $C_\chi(\cdot)$ is only determined up to orthogonal self maps of $H^s(\cdot),H^u(\cdot)$. In the two-dimensional case this means choosing a number from $\{\pm1\}$, which commutes with composition of linear maps, but the higher dimensional case is more subtle, because of non-commutativity.
\end{enumerate}
We address these issues in the following proposition.
\begin{prop}\label{Xi} Under the assumptions of the previous lemma:
$$\exists\Xi_i:T_{x_i}M\rightarrow T_{y_i}M\text{ s.t. }\forall\xi\in T_{x_i}M(1):\frac{|C_\chi^{-1}(x_i)\xi|}{|C_\chi^{-1}(y_i)\Xi_i\xi|}=e^{\pm4\sqrt{\epsilon}},$$
where $\Xi_i$ is an invertible linear transformation, and $\|\Xi_i\|,\|\Xi_i^{-1}\|\leq \exp(\epsilon)$.
\end{prop}
\begin{proof} 

Denote the following operators as defined in lemma \ref{boundimprove}
: $\pi_{x,k}^s:T_{f^k(p)}V_k^s\rightarrow H^s(x_k)$, $\pi_{x,k}^u:T_{f^k(p)}V^u_k\rightarrow H^u(x_k)$, $\pi_{y,k}^s:T_{f^k(p)}U_k^s\rightarrow H^s(y_k)$, $\pi_{y,k}^u:T_{f^k(p)}U_k^u\rightarrow H^u(x_k)$.
Using the notations from the previous lemma, define $\pi_{x,k}^s:=\pi_{x,k}:T_{f^k(p)}V^s_k\rightarrow H^s(x_k),\pi_{y,k}^s:=\pi_{y,k}:T_{f^k(p)}U^s_k\rightarrow H^s(y_k)$. Let $V^u_k:=V^u((\psi_{x_i}^{p_i^s,p_i^u})_{i\leq k}),U^u_k:=V^u((\psi_{y_i}^{q_i^s,q_i^u})_{i\leq k})$, and define $\pi_{x,k}^u:=\pi_{x,k}:T_{f^k(p)}V^u_k\rightarrow H^u(x_k),\pi_{y,k}^u:=\pi_{y,k}:T_{f^k(p)}U^u_k\rightarrow H^u(y_k)$ as in the unstable case of lemma \ref{boundimprove}. Define $\Xi_k^s:H^s(x_k)\rightarrow H^s(y_k)$, $\Xi_k^u:H^u(x_k)\rightarrow H^u(y_k)$ the following way:
$$\Xi_k^s:=\pi_{y,k}^s\circ
(\pi_{x,k}^s)^{-1}\text{ , }\Xi_k^u:=\pi_{y,k}^u\circ
(\pi_{x,k}^u)^{-1}.$$
The composition of $\pi_{y,k}^{s/u}\circ(\pi_{x,k}^{s/u})^{-1}$ is well defined by the remark after lemma \ref{chainbounds}. Define $\Xi_k:T_{x_k}M\rightarrow T_{y_k}M$:
\begin{equation}\label{defxi}\forall \xi\in T_{x_k}M\text{:  }\Xi_k\xi:=\Xi_k^s\xi_s+\Xi_k^u\xi_u\text{, where }\xi=\xi_s+\xi_u, \xi_{s/u}\in H^{s/u}(x_k).\end{equation}
This is well defined since $T_{x_k}M=H^s(x_k)\oplus H^u(x_k)$. Notice that $\Xi_k[H^{s/u}(x_k)]=H^{s/u}(y_k)$. 
Hence, by lemma \ref{chainbounds}:
\begin{equation}\label{shlukit}
\frac{|C_\chi^{-1}(x_k)\xi|}{|C_\chi^{-1}(y_k)\Xi_k\xi|}=\sqrt{\frac{S^2(x_k,\xi_s)+U^2(x_k,\xi_u)}{S^2(y_k,\Xi_k\xi_s)+U^2(y_k,\Xi_k\xi_u)}}=e^{\pm4\sqrt{\epsilon}}  .  
\end{equation}
It remains to bound $\|\Xi_k\|$ and $\|\Xi_k^{-1}\|$. We begin by showing bounds on the norms of $\Xi_k^{s/u}$. By the bounds for  $\|\pi_{x,k}^{s/u}\|$,$\|(\pi_{x,k}^{s/u})^{-1}\|$, $\|\pi_{y,k}^{s/u}\|$,$\|(\pi_{y,k}^{s/u})^{-1}\|$
from lemma \ref{boundimprove}
,
\begin{equation}\label{shluk}
\|\Xi_k^{s/u}\|\leq \exp\Big( 2Q_\epsilon(x_k)^{\frac{\beta}{4}}+
2Q_\epsilon(y_k)^{\frac{\beta}{4}}
\Big)=\exp(2(Q_\epsilon(x_k)^\frac{\beta}{4}+Q_\epsilon(y_k)^\frac{\beta}{4})).
\end{equation}

We continue to bound $\|\Xi_k\|$. To ease notations we will omit the '$k$' subscripts. Recall $\rho(M)$ from \textsection\ref{forrho}. By lemma \ref{firstchapter2}, $d(x,f^k(p))<\frac{\sqrt{d}}{100}Q_\epsilon(x)<\rho(M)$ for $\epsilon$ small enough.  Similarly, $d(y,f^k(p))<\rho(M)$. Therefore the following is well defined: $\psi_x^{-1}(f^k(p))=:z_x'$, $\psi_y^{-1}(f^k(p))=:z_y'$. The vectors of the first $s(x)$ coordinates of $z_{x}'/z_y'$ will be called $z_{x}/z_y$ respectively.

\medskip
$\textit{Part 1}$: Let $\xi\in T_xM$, $\xi=\xi_s+\xi_u$, $\xi_{s/u}\in H^{s/u}(x)$, then $|\xi_s|,|\xi_{u}|\leq \|C_\chi^{-1}(x)\|\cdot|\xi|$.

\textit{Proof}: WLOG $|\xi|=1$. Notice that the size of $\xi_{s/u}$ can be very big, even when $|\xi_s+\xi_u|=|\xi|=1$. This can happen when 
$\sphericalangle(H^s(x),H^u(x)):=\inf\limits_{\eta_s\in H^s(x),\eta_u\in H^u(x)}|\sphericalangle(\eta_s,\eta_u)|$ is very small, and $\xi_s$ and $\xi_u$ are almost parallel, of the same size, and are pointing to almost opposite directions. Consider the triangle created by the tips of $\xi_s$ and $-\xi_u$, and the origin. Denote its angle at the origin by $\alpha$ (the angle between $\xi_s$ and $-\xi_u$). The size of the edge in front of $\alpha$ has length 1 (since $|\xi_s-(-\xi_u)|=|\xi_s+\xi_u|=1$). Angles between two non-parallel vectors are in $(0,\pi)$, and their sine is always positive. By the sine theorem,
\begin{equation}\label{darksalt}
|\xi_s|,|\xi_u|\leq \frac{1}{|\sin\alpha|}\leq\frac{1}{\inf\limits_{\eta_s\in H^s(x),\eta_u\in H^u(x)}|\sin\sphericalangle(\eta_s,\eta_u)|}=\frac{1}{\sin\sphericalangle(H^s(x),H^u(x))}.
\end{equation}
Notice that $\frac{1}{|\sin\sphericalangle(\eta_s,\eta_u)|}\leq\|C_\chi^{-1}(x)\|$, $\forall\eta_{s/u}\in H^{s/u}(x)$. To see this, notice first the identity
$$\|C_\chi^{-1}(x)\|^2=\sup_{\eta_s\in H^s(x),\eta_u\in H^u(x),|\eta_s+\eta_u|=1}\{S^2(x,\eta_s)+U^2(x,\eta_u)\}.$$

Choose $\eta_s\in H^s(x)(1),\eta_u\in H^u(x)(1)$ which minimize $\sin\sphericalangle(\eta_s,\eta_u)$. WLOG $\cos\sphericalangle(\eta_s,\eta_u)>0$. Let $\zeta:=\frac{\eta_s-\eta_u}{|\eta_s-\eta_u|}$. Since $C_\chi^{-1}(x)\eta_s\perp C_\chi^{-1}(x)\eta_u$ and $\|C_\chi(x)\|\leq1$,
$$|C_\chi^{-1}(x)\zeta|^2\geq \frac{|\eta_s|^2}{|\eta_s-\eta_u|^2}+\frac{|\eta_u|^2}{|\eta_s-\eta_u|^2}=\frac{2}{2-2\cos\sphericalangle(\eta_s,\eta_u))}\geq\frac{1}{1-\cos^2\sphericalangle(\eta_s,\eta_u)}=\frac{1}{\sin^2\sphericalangle(\eta_s,\eta_u)}.$$
So $\|C_\chi^{-1}(x)\|\geq\frac{1}{|\sin\sphericalangle(\eta_s,\eta_u)|}$. By \eqref{darksalt}, $|\xi_s|,|\xi_u|\leq\|C_\chi^{-1}(x)\|$.

\medskip
\noindent QED

$\medskip$
$\textit{Part 2}$: $\Xi=d_{C_\chi(x)z_x'}(\exp_y^{-1}\exp_x)+E_1$, where $E_1:T_xM\rightarrow T_yM$ is linear and $\|E_1\|\leq Q_\epsilon(x)^\frac{\beta}{5}+Q_\epsilon(y)^\frac{\beta}{5}<\frac{\epsilon}{2}$.

\textit{Proof}: We will begin by showing that the expression above is well defined for $\epsilon<\frac{\rho(M)}{2}$. This is true because $\exp_x[B_\epsilon(0)]\subset B_{\epsilon+d(x,y)}(y)\subset B_{2\epsilon}(y)$; $|C_\chi(x)z_x'|=d(x,f^k(p))<\epsilon$; so $\exp_y^{-1}\exp_x$ is defined on an open neighborhood of $C_\chi(x)z_x'$, and hence is differentiable on it.

By definition, $\Xi^s=\pi_{y,k}^s\circ(\pi_{x,k}^s)^{-1}$.
\begin{itemize}
\item $\xi_s\in H^s(x)$ is mapped by $(\pi_{x,k}^s)^{-1}$ to $\xi_s'$, a tangent vector at $f^k(p)$
given by $\xi_s'=d_{z_x'}\psi_x\Big(\begin{array}{c}
v\\
d_{z_x}Gv\\ \end{array}\Big)$, where $G$ represents $V^s_k$ in $\psi_{x_k}^{p_k^s,p^u_k}$ and $v=C_\chi^{-1}(x)\xi_s$. We write $\xi_s'$ as follows:
$$\xi_s'=d_{z_x'}\psi_x\Big(\begin{array}{c}
v\\
d_{z_x}Gv\\ \end{array}\Big)=(d_{z_y'}\psi_y)(d_{z_y'}\psi_y)^{-1}d_{z_x'}\psi_x\Big(\begin{array}{c}
v\\
d_{z_x}Gv\\ \end{array}\Big)=:d_{z_y'}\psi_y\Big(\begin{array}{c}
w\\
d_{z_y}Fw\\ \end{array}\Big),$$
for some $w$, where $F$ is the representing function of $U^s_k$ in $\psi_{y_k}^{q^s_k,q^u_k}$. This representation is possible because $\xi_s'$ is tangent to the stable manifold $U^s_k$ in $\psi_{y_k}^{q^s_k,q^u_k}$ at $f^k(p)=\psi_y(z_y')$.
\item $\xi_s'=(\psi_{x,k}^s)^{-1}\xi_s$ is mapped by $\pi_{y,k}^s$ to
$$\Xi^s\xi_s=\pi_{y,k}^s\xi_s'=d_0\psi_y\Big(\begin{array}{c}
w\\
0\\ \end{array}\Big).$$
\end{itemize}

By construction,
$$\Big(\begin{array}{c}
w\\
d_{z_y}Fw\\ \end{array}\Big)=(d_{z_y'}\psi_y)^{-1}(d_{z_x'}\psi_x)\Big(\begin{array}{c}
v\\
d_{z_x}Gv\\ \end{array}\Big).$$

\noindent We abuse notation and use $w,v$ to represent both vectors in $\mathbb{R}^{s(x)}\times \{0\}
$ and vectors in $\mathbb{R}^{s(x)}
$. Then, $$w=(d_{z_y'}\psi_y)^{-1}d_{z_x'}\psi_xv+\Big[ (d_{z_y'}\psi_y)^{-1}d_{z_x'}\psi_x\Big(\begin{array}{c}
0\\
d_{z_x}Gv\\ \end{array}\Big)-\Big(\begin{array}{c}
0\\
d_{z_y}Fw\\ \end{array}\Big)\Big].$$
Call the term in the brackets $E_s\xi_s$, then
$\Xi_s=C_\chi(y)(d_{z_y'}\psi_y)^{-1}d_{z_x'}\psi_xC_\chi^{-1}(x)+C_\chi(y)E_s$.
\begin{itemize}
  \item  $\Xi^s\xi_s=d_0\psi_y\Big(\begin{array}{c}
w\\
0\\ \end{array}\Big)\Rightarrow |w|=|(d_0\psi_y)^{-1}\Xi^s\xi_s|=|C_\chi^{-1}(y)\Xi^s \xi_s|\leq\|C_\chi^{-1}(y)\|\cdot\|\Xi^s\|\cdot|\xi_s|\leq$

$\leq\|C_\chi^{-1}(y)\|\cdot e^{2(Q_\epsilon(x)^\frac{\beta}{4}+Q_\epsilon(y)^\frac{\beta}{4})}|\xi_s|\leq 2\|C_\chi^{-1}(y)\|\cdot |\xi_s|(\because \text{ \eqref{shluk} })$.
\item $\mathrm{Lip}(G)\leq Q_\epsilon(x)^\frac{\beta}{3},\mathrm{Lip}(F)\leq Q_\epsilon(y)^\frac{\beta}{3}$ (by the remark after definition \ref{admissible}).
\item $\xi_s=C_\chi(x)v\Rightarrow|v|\leq \|C_\chi^{-1}(x)\|\cdot |\xi_s|$.
\end{itemize} Hence:

\begin{align}\label{Es}
|C_\chi(y)E_s\xi_s|\leq&\|d_{z_y'}\exp_y^{-1}\|\cdot\|d_{z_x'}\psi_x\|\cdot\|d_{z_x}G\|\cdot |v|+\|d_{z_y}F\|\cdot|w|\nonumber\\
\leq&2\cdot2Q_\epsilon(x)^\frac{\beta}{3}\cdot\|C_\chi^{-1}(x)\|\cdot|\xi_s|+Q_\epsilon(y)^\frac{\beta}{3}\cdot2\|C_\chi^{-1}(y)\|\cdot|\xi_s|
\leq(Q_\epsilon(x)^\frac{\beta}{4}+Q_\epsilon(y)^\frac{\beta}{4})|\xi_s|.
\end{align}
Similarly, $\Xi_u=C_\chi(y)(d_{z_y'}\psi_y)^{-1}d_{z_x'}\psi_xC_\chi^{-1}(x)+C_\chi(y)E_u$, where $E_u:H^u(x)\rightarrow T_yM$ and 
\begin{equation}\label{Eu}
\|C_\chi(y)E_u\|
\leq (Q_\epsilon(x)^\frac{\beta}{4}+Q_\epsilon(y)^\frac{\beta}{4})
.
\end{equation}
For a general tangent $\xi\in T_xM$, write $\xi=\xi_s+\xi_u,\xi_{s/u}\in H^{s/u}(x)$, then

\begin{align*}\Xi\xi=&\Xi(\xi_s+\xi_u)=\Xi_s\xi_s+\Xi_u\xi_u \\
=&\Big(C_\chi(y)(d_{z_y'}\psi_y)^{-1}d_{z_x'}\psi_xC_\chi^{-1}(x)+C_\chi(y)E_s\Big)\xi_s+\Big(C_\chi(y)(d_{z_y'}\psi_y)^{-1}d_{z_x'}\psi_xC_\chi^{-1}(x)+C_\chi(y)E_u\Big)\xi_u\\
=&C_\chi(y)(d_{z_y'}\psi_y)^{-1}d_{z_x'}\psi_xC_\chi^{-1}(x)(\xi_s+\xi_u)+C_\chi(y)(E_s\xi_s+E_u\xi_u)\\
=&C_\chi(y)(d_{z_y'}\psi_y)^{-1}d_{z_x'}\psi_xC_\chi^{-1}(x)\xi+C_\chi(y)(E_s\xi_s+E_u\xi_u).
\end{align*}

Hence, since $
d_{z'_{x}}\psi_{x}=d_{C_\chi(x)z_{x}'}\exp_{x}\circ C_\chi(x),d_{z'_{y}}\psi_{y}=d_{C_\chi(y)z_{y}'}\exp_{y}\circ C_\chi(y)$, we get
$$\Xi\xi=d_{f^k(p)}\exp_y^{-1}d_{C_\chi(x)z_x'}\exp_x\xi+C_\chi(y)(E_s\xi_s+E_u\xi_u)
\equiv d_{C_\chi(x)z_x'}(\exp_y^{-1}\exp_x)\xi+E_1\xi,$$
where
\begin{align*}|E_1\xi|\leq& |C_\chi(y)E_s\xi_s|+|C_\chi(y)E_u\xi_u|\leq(Q_\epsilon(x)^\frac{\beta}{4}+Q_\epsilon(y)^\frac{\beta}{4})(|\xi_s|+|\xi_u|)\text{ } (\because\text{ \eqref{Es},\eqref{Eu}})\\
\leq& (Q_\epsilon(x)^\frac{\beta}{4}+Q_\epsilon(y)^\frac{\beta}{4})(\|C_\chi^{-1}(x)\|\cdot|\xi|+\|C_\chi^{-1}(x)\|\cdot|\xi|)\text{ }(\because\text{Part 1})\leq2\|C_\chi^{-1}(x)\|(Q_\epsilon(x)^\frac{\beta}{4}+Q_\epsilon(y)^\frac{\beta}{4})|\xi|.\end{align*}
Therefore,
$$\Xi=d_{C_\chi(x)z_x'}(\exp_y^{-1}\exp_x)+E_1\Rightarrow\|\Xi\|\leq \|d_\cdot\exp_y^{-1}\|\cdot\|d_\cdot\exp_x\|+\|E_1\|\leq4+2\|C_\chi^{-1}(x)\|(Q_\epsilon(x)^\frac{\beta}{4}+Q_\epsilon(y)^\frac{\beta}{4})$$
Now using this again in \eqref{shlukit}:
\begin{align*}
\|C_\chi^{-1}(x)\|\leq& e^{4\sqrt{\epsilon}}\|C_\chi^{-1}(y)\Xi\|\leq e^{4\sqrt{\epsilon}}\|C_\chi^{-1}(y)\|\cdot\|\Xi\|\leq e^{4\sqrt{\epsilon}}\|C_\chi^{-1}(y)\|\Big(4+2\|C_\chi^{-1}(x)\|(Q_\epsilon(x)^\frac{\beta}{4}+Q_\epsilon(y)^\frac{\beta}{4})\Big)\\
=& \|C_\chi^{-1}(y)\|\cdot e^{4\sqrt{\epsilon}}\Big(4+2\|C_\chi^{-1}(x)\|Q_\epsilon(x)^\frac{\beta}{4}\Big)+\|C_\chi^{-1}(x)\|\cdot2e^{4\sqrt{\epsilon}}\Big(\|C_\chi^{-1}(y)\|Q_\epsilon(y)^\frac{\beta}{4}\Big)\\
\leq&5(1-\epsilon)\|C_\chi^{-1}(y)\|+\epsilon\|C_\chi^{-1}(x)\|
\Rightarrow (1-\epsilon)\|C_\chi^{-1}(x)\|\leq (1-\epsilon)5\|C_\chi^{-1}(y)\|\Rightarrow\|C_\chi^{-1}(x)\|\leq 5\|C_\chi^{-1}(y)\|.
\end{align*}
By symmetry we will also get $\|C_\chi^{-1}(y)\|\leq 5\|C_\chi^{-1}(x)\|$. Substituting these in the bounds for $\|E_1\|$:
$$\|E_1\|\leq2\|C_\chi^{-1}(x)\|Q_\epsilon(x)^\frac{\beta}{4}+2\|C_\chi^{-1}(x)\|Q_\epsilon(y)^\frac{\beta}{4}\leq2\|C_\chi^{-1}(x)\|Q_\epsilon(x)^\frac{\beta}{4}+2\cdot5\|C_\chi^{-1}(y)\|Q_\epsilon(y)^\frac{\beta}{4}\leq$$ $$\leq Q_\epsilon(x)^\frac{\beta}{5}+Q_\epsilon(y)^\frac{\beta}{5}\leq\frac{1}{2}\epsilon.$$
QED

\medskip
$\textit{Part 3}$: $\forall u\in R_\epsilon(0)$: $\exists D\in \mathcal{D}$ s.t. $\|\Theta_D d_{C_\chi(x)u}(\exp_y^{-1}\exp_x)-\Theta_D\|\leq \frac{1}{2}\epsilon$ (see definition \ref{isometries}).

\textit{Proof}: Begin by choosing $\epsilon$ smaller than $\frac{\varpi(\mathcal{D})}{2\sqrt{d}}$ (as in definition \ref{isometries}). Then $\exists D\in\mathcal{D}$ such that $\exp_x[R_\epsilon(0)]\subset B_{\frac{\varpi(\mathcal{D})}{2}}(x)$, because $d(x,\exp_xC_\chi(x)v)=|C_\chi(x)v|_2\leq\sqrt{d}|v|_\infty<\epsilon\sqrt{d}$, so $\exp_x[R_\epsilon(0)]$ has diameter less than $2\sqrt{d}\epsilon<\varpi(\mathcal{D})$. Hence the choice of $D$ and of $\Theta_D$ is proper. Since

$$\Theta_D\exp_y^{-1}\exp_x=(\Theta_D\exp_y^{-1}-\Theta_D\exp_x^{-1})\exp_x+\Theta_D,$$
we get that for any $u\in R_\epsilon(0)$:
$$\Theta_Dd_{C_\chi(x)u}(\exp_y^{-1}\exp_x)=d_{C_\chi(x)u}(\Theta_D\exp_y^{-1}\exp_x)=d_{\exp_xC_\chi(x)u}(\Theta_D\exp_y^{-1}-\Theta_D\exp_x^{-1})d_{C_\chi(x)u}\exp_x+\Theta_D.$$
Then
$$\|\Theta_Dd_{C_\chi(x)u}(\exp_y^{-1}\exp_x)-\Theta_D\|\leq \|\Theta_D\exp_y^{-1}-\Theta_D\exp_x^{-1}\|_{C^2}\cdot \|d_{C_\chi(x)u}\exp_x\|\leq L_2d(x,y)\|d_{C_\chi(x)u}\exp_x\|,$$
where $L_2$ is the uniform Lipschitz const. of $x\mapsto\nu_x^{-1}\exp_x^{-1}$ (proposition \ref{chartsofboxes}). This is due to the fact that $d(x,\exp_xC_\chi(x)u)<\sqrt{d}\epsilon\Rightarrow\exp_xC_\chi(x)u\in D$. In addition, $\|d_{C_\chi(x)u}\exp_x\|\leq\|Id_{T_xM}\|+E_0|C_\chi(x)u-0|\leq1+E_0\epsilon$ ($E_0$ is a constant introduced in lemma \ref{boundimprove}). Therefore, in total,
\begin{equation}\label{explicitcalculation}\|\Theta_Dd_{C_\chi(x)u}(\exp_y^{-1}\exp_x)-\Theta_D\|\leq L_2(1+\epsilon E_0)\cdot d(x,y)<Q_\epsilon(y)^\frac{\beta}{5}<\frac{1}{2}\epsilon.\end{equation}
QED

Adding the estimates of part 2 and part 3, we get \begin{align*}\|\Xi\|=&\|\Theta_Dd_{C_\chi(x)v}(\exp_y^{-1}\exp_x)-\Theta_D+\Theta_D+\Theta_D E_1\|\\ \leq&\|\Theta_D d_{C_\chi(x)v}(\exp_y^{-1}\exp_x)-\Theta_D\|+\|\Theta_D\|+\|E_1\|\leq \frac{1}{2}\epsilon+1+\frac{1}{2}\epsilon\leq e^\epsilon.\end{align*}

If we exchange the roles of $x_k,y_k$ then we get $\Xi_k^{-1}$, so $\|\Xi_k^{-1}\|$ is also less than $e^\epsilon$.
\end{proof}

\begin{cor} Under the assumptions and notations of the previous lemma:
$$\forall i\in\mathbb{Z}: \frac{\|C_\chi^{-1}(x_i)\|}{\|C_\chi^{-1}(y_i)\|}=e^{\pm5\epsilon^\frac{1}{2}}.$$

\end{cor}
\begin{proof}
By proposition \ref{Xi}:
$$\|C_\chi^{-1}(x_i)\|\leq e^{4\sqrt{\epsilon}}\|C_\chi^{-1}(y_i)\Xi_i\|\leq e^{\epsilon+4\sqrt{\epsilon}}\|C_\chi^{-1}(y_i)\|\leq e^{5\sqrt{\epsilon}}\|C_\chi^{-1}(y_i)\|,$$
and by symmetry we get the other inequality, and we are done.
\end{proof}

\subsubsection{Comparing frame parameters}

\begin{definition}
(See \cite{Sarig})
\medskip
\begin{enumerate}
\item A positive or negative chain is called {\em regular} if it can be completed to a regular chain (i.e., every coordinate is relevant, and some double chart appears infinitely many times).
\item If $v$ is a double chart, then $p^{u/s}(v)$ means the $p^{u/s}$ in $\psi_x^{p^s,p^u}=v$.
\item A negative chain $(v_i)_{i\leq0}$ is called {\em $\epsilon$-maximal} if it is regular, and $$p^u(v_0)\geq e^{-\sqrt[3]{\epsilon}}p^u(u_0)$$ for every regular chain $(u_i)_{i\in\mathbb{Z}}$ for which there is a positive regular chain $(v_i)_{i\geq0}$ s.t. $\pi((u_i)_{i\in\mathbb{Z}})=\pi((v_i)_{i\in\mathbb{Z}})$.
\item A positive chain $(v_i)_{i\geq0}$ is called {\em $\epsilon$-maximal} if it is regular, and $$p^s(v_0)\geq e^{-\sqrt[3]{\epsilon}}p^s(u_0)$$ for every regular chain $(u_i)_{i\in\mathbb{Z}}$ for which there is a negative regular chain $(v_i)_{i\leq0}$ s.t. $\pi((u_i)_{i\in\mathbb{Z}})=\pi((v_i)_{i\in\mathbb{Z}})$.
\end{enumerate}
\end{definition}

\begin{prop}\label{prop213}
The following holds for all $\epsilon$ small enough. For every regular chain $(v_i)_{i\in\mathbb{Z}}$, $(u_i)_{i\leq0}$,$(u_i)_{i\geq0}$ are $\epsilon$-maximal.
\end{prop}
\begin{proof} The following holds for all $\epsilon$ small enough. Let $u$ and $v$ be two regular chains s.t. $\pi(u)=\pi(v)$. If $u_0=\psi_x^{p^s,p^u}$ and $v_0=\psi_y^{q^s,q^u}$, then $\frac{Q_\epsilon(x)}{Q_\epsilon(y)}=e^{\pm\sqrt[3]{\epsilon}}.$ In order to see that, recall that $$\|C_\chi^{-1}(x)\|^{48}=\widetilde{Q}_\epsilon(x)^{-\beta}\cdot\epsilon^{90}\frac{1}{3^{6/\beta}},Q_\epsilon(x)\leq\widetilde{Q}_\epsilon(x)\leq Q_\epsilon(x)\cdot e^{\frac{1}{3}\epsilon}.$$
So, $\frac{Q_\epsilon(x)}{Q_\epsilon(y)}\leq e^{\frac{1}{3}\epsilon} \frac{\|C_\chi^{-1}(x)\|^{48}}{\|C_\chi^{-1}(y)\|^{48}}$, and by the previous corollary this is less than $e^{\frac{1}{3}\epsilon}e^{240\epsilon^\frac{1}{2}}$. For $\epsilon$ small enough this less than $e^{\sqrt[3]{\epsilon}}$. The other inequality is derived identically.

From here onward, the proof is exactly as in \cite[proposition~8.3]{Sarig}.

\end{proof}
\begin{lemma}\label{prop214}
Let $(\psi_{x_i}^{p_i^s,p_i^u})_{i\in\mathbb{Z}}$, and $(\psi_{y_i}^{q_i^s,q_i^u})_{i\in\mathbb{Z}}$ be two regular chains s.t. $\pi((\psi_{x_i}^{p_i^s,p_i^u})_{i\in\mathbb{Z}})=\pi((\psi_{y_i}^{q_i^s,q_i^u})_{i\in\mathbb{Z}})=p$ then for all $i\in\mathbb{Z}$, $\frac{p_i^u}{q_i^u},\frac{p_i^s}{q_i^s}\in[e^{-\sqrt[3]{\epsilon}},e^{\sqrt[3]{\epsilon}}]$.
\end{lemma}
\begin{proof}
By the previous proposition, $(\psi_{x_i}^{p_i^s,p_i^u})_{i\leq0}$ is $\epsilon$-maximal, so $p_0^u\geq e^{-\sqrt[3]{\epsilon}}q_0^u$. $(\psi_{x_i}^{p_i^s,p_i^u})_{i\leq0}$ is also $\epsilon$-maximal, so $q_0^u\geq e^{-\epsilon^\frac{1}{3}}p_0^u$. It follows that $\frac{p_0^u}{q_0^u}\in[e^{-\epsilon^\frac{1}{3}},e^{\epsilon^\frac{1}{3}}]$. Similarly $\frac{p_0^s}{q_0^s}\in[e^{-\epsilon^\frac{1}{3}},e^{\epsilon^\frac{1}{3}}]$. Working with the shifted sequences $(\psi_{x_{i+k}}^{p_{i+k}^s,p_{i+k}^u})_{i\in\mathbb{Z}}$, and $(\psi_{y_{i+k}}^{q_{i+k}^s,q_{i+k}^u})_{i\in\mathbb{Z}}$, we obtain $\frac{p_i^u}{q_i^u},\frac{p_i^s}{q_i^s}\in[e^{-\epsilon^\frac{1}{3}},e^{\epsilon^\frac{1}{3}}]$ for all $i\in\mathbb{Z}$.
\end{proof}

\subsubsection{Comparing Pesin charts}
The following theorem,  the ``solution to the inverse problem", is a multidimensional generalization of theorem 5.2 in \cite{Sarig}. Our proof is different from Sarig's: in the higher dimensional case there is no  explicit formula for the operator $C_\chi(x)$, because of the lack of a canonical basis compatible with the Lyapunov inner product (recall the definition in  \eqref{LIP},\eqref{omrisuggestnumber}). We will use the map $\Xi$ from proposition \ref{Xi} to overcome this problem.
\begin{theorem}\label{beforefinal}
The following holds for all $\epsilon$ small enough. Suppose $(\psi_{x_i}^{p_i^s,p_i^u})_{i\in\mathbb{Z}},(\psi_{y_i}^{q_i^s,q_i^u})_{i\in\mathbb{Z}}$ are regular chains s.t. $\pi((\psi_{x_i}^{p_i^s,p_i^u})_{i\in\mathbb{Z}})=p=\pi((\psi_{y_i}^{q_i^s,q_i^u})_{i\in\mathbb{Z}})$ then for all i:
\begin{enumerate}
\item $d(x_i,y_i)<\epsilon$,
\item $\frac{p_i^s}{q_i^s},\frac{p_i^u}{q_i^u}\in[e^{-\epsilon^\frac{1}{3}},e^{\epsilon^\frac{1}{3}}]$,
\item $(\psi_{y_i}^{-1}\circ\psi_{x_i})(u)=O_iu+a_i+\Delta_i(u)$ for all $u\in R_\epsilon(0)$, where $O_i\in O(d)$ is a $d\times d$ orthogonal matrix, $a_i$ is a constant vector s.t. $|a_i|_\infty\leq10^{-1}(q_i^u\wedge q_i^s)$, and $\Delta_i$ is a vector field s.t. $\Delta_i(0)=0,\|d_v\Delta_i\|<\frac{1}{2}\epsilon^\frac{1}{3}$. $O_i$ preserves $\mathbb{R}^{s(x)}$ and $\mathbb{R}^{d-s(x)}$. \footnote{Recall, in our notations $\mathbb{R}^{s(x)}:=C_\chi^{-1}(x_i)[H^s(x_i)]=C_\chi^{-1}(y_i)[H^s(y_i)]\subset \mathbb{R}^d$ and $\mathbb{R}^{d-s(x)}:=(\mathbb{R}^{s(x)})^\perp=C_\chi^{-1}(x_i)[H^u(x_i)]=C_\chi^{-1}(y_i)[H^u(y_i)]$.}
\end{enumerate}
\end{theorem}
\begin{proof}
Parts $(1)$ and $(2)$ are the content of lemma \ref{firstchapter2} and lemma \ref{prop214} respectively; hence we are left with proving $(3)$.

First we show that $\psi_{y_i}^{-1}\psi_{x_i}$ is well defined on $R_\epsilon(0)$. Recall the definition of $r=r(M)$ from \textsection1.1.3. Then
$$\psi_{y_i}^{-1}\psi_{x_i}=C_\chi^{-1}(y_i)\exp_{y_i}^{-1}\exp_{x_i}C_\chi(x_i)$$
is well defined and smooth for all $v\in \mathbb{R}^d$ s.t. $|C_\chi(x_i)v|_2<2r$ and $d(y_i,\exp_{x_i}C_\chi(x_i)v)<\rho$. So for $v\in R_\epsilon(0)$ and $\epsilon$ small enough:
$$|C_\chi(x_i)v|_2\leq|v|_2\leq\sqrt{d}|v|_\infty<\epsilon\sqrt{d}<2r,$$
and $d(y_i,\exp_{x_i}C_\chi(x_i)v)\leq d(y_i,x_i)+d(x_i,\exp_{x_i}C_\chi(x_i)v)=d(x_i,y_i)+|C_\chi(x_i)v|_2\leq \epsilon+\sqrt{d}\epsilon$.
Now we can begin to show (3). We start by pointing out that proposition \ref{Xi} can be restated as follows:
$$\forall v\in\mathbb{R}^d: |C_\chi^{-1}(y)\Xi C_\chi(x)v|=e^{\pm4\sqrt{\epsilon}}|C_\chi^{-1}(x)C_\chi(x)v|=|v|e^{\pm4\sqrt{\epsilon}}.$$
By the polar decomposition for real matrices
, there exists an orthogonal matrix $O_i$, and a positive symmetric matrix $R_i$ s.t. 
$$C_\chi^{-1}(y_i)\Xi_i C_\chi(x_i)=O_i\cdot R_i, \forall v |R_iv|=e^{\pm4\sqrt{\epsilon}}|v|.$$
$$\Big(R_i=\sqrt{(C_\chi^{-1}(y_i)\Xi_i C_\chi(x_i))^t(C_\chi^{-1}(y_i)\Xi_i C_\chi(x_i))},O_i=C_\chi^{-1}(y_i)\Xi_i C_\chi(x_i)\cdot R_i^{-1}\Big).$$
$C_\chi^{-1}(y_i)\Xi_i C_\chi(x_i)$ preserves $\mathbb{R}^{s(x)}\times \{0\}=C_\chi^{-1}(x_i)[H^s(x_i)]=C_\chi^{-1}(y_i)[H^s(y_i)]$; $\{0\}\times\mathbb{R}^{d-s(x)}=C_\chi^{-1}(x_i)[H^u(x_i)]=C_\chi^{-1}(y_i)[H^u(y_i)]$. Therefore $R_i$ and $O_i$ preserve $\mathbb{R}^{s(x)}\times \{0\},\{0\}\times\mathbb{R}^{s(x)}$.

We now define a transformation $\Delta_i$ using the following identity:
$$\psi_{y_i}^{-1}\circ\psi_{x_i}=O_i+\psi_{y_i}^{-1}\circ\psi_{x_i}(0)+\Delta_i.$$
Call $\psi_{y_i}^{-1}\circ\psi_{x_i}(0)=:a_i$. Since $\psi_z=\exp_z\circ C_\chi(z)$ ($z=x_i,y_i$),
$$\Delta_i=[C_\chi^{-1}(y_i)\exp_{y_i}^{-1}\exp_{x_i}C_\chi(x_i)]-a_i-O_i.$$
We omit the subscript $i$ for easier notation. We get
$$\Delta=C_\chi^{-1}(y)[\exp_y^{-1}\exp_x-\Xi]C_\chi(x)+(C_\chi^{-1}(y)\Xi C_\chi(x)-O)-a.$$
We bound the differentials of these summands, beginning with the second one:
$$\|d_\cdot(C_\chi^{-1}(y)\Xi C_\chi(x)-O)\|=\|C_\chi^{-1}(y)\Xi C_\chi(x)-O\|=\|OR-O\|=\|R-Id\|.$$
Since $R$ is positive and symmetric, it can be orthogonally diagonalized. Let $v_i$ be the orthonormal basis of eigenvectors, and let $\lambda_i$ be the eigenvalues. Necessarily $\lambda_i=e^{\pm4\sqrt{\epsilon}}$. So $|(R-Id)v_i|\leq(e^{4\sqrt{\epsilon}}-1)|v_i|\leq8\sqrt{\epsilon}$, so $\|C_\chi^{-1}(y)\Xi C_\chi(x)-O\|=\|R-Id\|\leq8\sqrt{\epsilon}$.

$\medskip$
Now to bound the differential of the first summand $(\exp_y^{-1}\exp_x-\Xi)$. Let $z_x':=\psi_{x}^{-1}(f^i(p))$ and choose $D\ni x,y$ as in definition \ref{isometries}. Then for every $v\in R_\epsilon(0)$,
\begin{align*}\Theta_Dd_{C_\chi(x)v}(\exp_y^{-1}\exp_x-\Xi)=&\Theta_Dd_{C_\chi(x)v}(\exp_y^{-1}\exp_x)-\Theta_Dd_{C_\chi(x)v}\Xi=\Theta_Dd_{C_\chi(x)v}(\exp_y^{-1}\exp_x)-\Theta_D\Xi\\
=&\Theta_Dd_{C_\chi(x)v}(\exp_y^{-1}\exp_x)-\Theta_Dd_{C_\chi(x)z_x'}(\exp_y^{-1}\exp_x)-\Theta_DE_1,\end{align*}
where $E_1=\Xi-d_{C_\chi(x)z_x'}(\exp_y^{-1}\exp_x)$ as in part 2 of proposition \ref{Xi}. Hence, for any $v\in R_\epsilon(0)$:
$$\|\Theta_Dd_{C_\chi(x)v}(\exp_y^{-1}\exp_x-\Xi)\|\leq\|\Theta_Dd_{C_\chi(x)v}(\exp_y^{-1}\exp_x)-\Theta_Dd_{C_\chi(x)z_x'}(\exp_y^{-1}\exp_x)\|+\|E_1\|\leq$$
$$\leq\|\Theta_Dd_{C_\chi(x)v}(\exp_y^{-1}\exp_x)-\Theta_D\|+\|\Theta_D-\Theta_Dd_{C_\chi(x)z_x'}(\exp_y^{-1}\exp_x)\|+\|E_1\|.$$
 Therefore, by part 2 of proposition \ref{Xi}, by eq. \eqref{explicitcalculation}, and since $z_x'\in R_\epsilon(0)$, 
\begin{align*}\|d_\cdot (\exp_y^{-1}\exp_x-\Xi)\|&\leq Q_\epsilon(y)^\frac{\beta}{5}+Q_\epsilon(y)^\frac{\beta}{5}+(Q_\epsilon(y)^\frac{\beta}{5}+Q_\epsilon(x)^\frac{\beta}{5})\\
&\leq(3+e^{\sqrt[3]{\epsilon}})Q_\epsilon(y)^\frac{\beta}{5}\text{ (by the proof of proposition \ref{prop213})}\leq 5Q_\epsilon(y)^\frac{\beta}{5}.
\end{align*}
Hence, for all $v\in R_\epsilon(0)$, if $C_2:=\Theta_D\circ C_\chi(y)$, then
$$\|d_v\Big(C_\chi^{-1}(y)(\exp_y^{-1}\exp_x-\Xi)C_\chi(x)\Big)\|=\|C_2^{-1}\Theta_Dd_{C_\chi(x)v}(\exp_y^{-1}\exp_x-\Xi)C_\chi(x)\|\leq$$
$$\leq\|C_2^{-1}\|\cdot\|\Theta_Dd_{C_\chi(x)v}(\exp_y^{-1}\exp_x-\Xi)\|\leq\|C_\chi^{-1}(y)\|5Q_\epsilon(y)^\frac{\beta}{5}\leq \sqrt{\epsilon}.$$
It is clear that $\|d_\cdot a\|=0$. Hence in total:
$$\|d_v\Delta_i\|\leq\sqrt{\epsilon}+8\sqrt{\epsilon}+0\leq\frac{1}{2}\sqrt[3]{\epsilon}\text{ , for }\epsilon\text{ small enough and }v\in R_\epsilon(0).$$

It is also clear that $\Delta(0)=0$ by definition.

Finally, if $z_y'=\psi_y^{-1}(f^i(p))$ then $$z'_y=\psi_y^{-1}\psi_x(z'_x)=Oz'_x+\Delta(z'_x)+a.$$
We obtain
$$|a|\leq|z'_y|+|z'_x|+|\Delta(z'_x)|\leq|z'_y|+|z'_x|+\frac{1}{2}\epsilon^\frac{1}{3}|z'_x|\leq\frac{1}{10}q^u\wedge q^s.$$
\end{proof}
We can improve part $(3)$ in the previous theorem the following way: 
\begin{cor}\label{final} There is a finite $\frac{1}{2}\epsilon^\frac{1}{3}$-dense set $O_\epsilon(d)\subset O(d)$ s.t. for any $i\in\mathbb{Z}$ there is an $O'_i\in O_\epsilon(d)$, and $\Delta'_i$ as before s.t.:

$(\psi_{y_i}^{-1}\circ\psi_{x_i})(u)=O'_iu+a_i+\Delta'_i(u)$ for all $u\in R_\epsilon(0)$, where $\Delta'_i(0)=0,\|d_v\Delta'_i\|<\epsilon^\frac{1}{3}$
\end{cor}
\begin{proof} $s(x)$ belongs to a finite set of possible values $1,...,d-1$.\footnote{$s(x):=s(x_i)=s(y_i)$ for all $i$. This fact was pointed out and explained in the remark before the proof of lemma \ref{chainbounds}.} For each one of those values $s$,
choose once and for all some $\frac{1}{2}\epsilon^\frac{1}{3}$-net for $O(s)\times O(d-s)=\{\text{all orthogonal maps which preserve }\mathbb{R}^s\text{ and }\mathbb{R}^{d-s}\}$, which exists since $O(s), O(d-s)$ are compact groups. Take some $O'_i$ in the corresponding $O(s(x))\times O(d-s(x))$ to approximate $O_i$ from the theorem. Define $\Delta'_i:=\Delta_i+O_i-O'_i$. The rest follows immediately.
\end{proof}
\subsection{Similar charts have similar manifolds}
\subsubsection{Comparing manifolds}
The following proposition is the higher-dimensional version of \cite[proposition~6.4]{Sarig}.
\begin{prop}\label{prop221} The following holds for all $\epsilon$ small enough. Let $V^s$ (resp. $U^s$) be an $s$-admissible manifold in $\psi_x^{p^s,p^u}$ (resp. $\psi_y^{q^s,q^u}$). Suppose $V^s, U^s$ stay in windows. If $x=y$ then either $V^s,U^s$ are disjoint, or one contains the other. The same statement holds for $u$-admissible manifolds.
\end{prop}
\begin{proof}
Assume $V^s\cap U^s\neq\varnothing$. Since $V^s$ stays in windows there is a positive chain $(\psi_{x_i}^{p^s_i,p^u_i})_{i\geq0}$ such that $\psi_{x_i}^{p^s_i,p^u_i}=\psi_x^{p^s,p^u}$ and such that for all $i\geq0$: $f^i[V^s]\subset W_i^s$ where $W_i^s$ is an $s$-admissible manifold in $\psi_{x_i}^{p^s_i,p^u_i}$.

\noindent\textit{Claim 1}: The following holds for $\epsilon$ small enough: $f^n[V^s]\subset \psi_{x_n}[R_{\frac{1}{2}Q_\epsilon(x_n)}(0)]$ for all $n$ large enough.

\noindent\textit{Proof}: See claim 1 in \cite[proposition~6.4]{Sarig}.

\noindent\textit{Claim 2}: For all $\epsilon$ small enough $f^n[U^s]\subset\psi_{x_n}[R_{Q_\epsilon(x_n)}(0)]$ for all $n$ large enough.

\noindent\textit{Proof}: See claim 2 in \cite[proposition~6.4]{Sarig}.

\medskip
\noindent\textit{Claim 3}: Recall that $V^s$ is $s$-admissible in $\psi_x^{p^s,p^u}$ and $U^s$ is $s$-admissible in $\psi_y^{q^s,q^u}$. If $p^s\leq q^s$ then $V^s\subset U^s$ (and if $q^s\leq p^s$ then $U^s\subset V^s$).

\textit{Proof}: WLOG $p^s\leq q^s$. First assume $p^s<q^s$. Pick $n$ s.t. for all $m\geq n$: $f^m[U^s],f^m[V^s]\subset \psi_{x_m}[R_{Q_\epsilon(x_m)}(0)]$. Then: $f^{n}[V^s],f^{n}[U^s]\subset W^s:=V^s((\psi_{x_i}^{p_i^s,p_i^u})_{i\geq n})$ (proposition \ref{firstofchapter} section 4.). Let $G$ be the function which represents $W^s$ in $\psi_{x_{n}}$, then $\psi_{x_n}^{-1}[f^n[U^s]]$ and $\psi_{x_n}^{-1}[f^n[V^s]]$ are two connected subsets of $graph(G)$.
So there are $S_u,S_v\subset \mathbb{R}^{s(x)}$ s.t.
$$f^n[V^s]=\psi_{x_n}[\{(t,G(t))|t\in S_v\}],f^n[U^s]=\psi_{x_n}[\{(t,G(t))|t\in S_u\}].$$

\noindent $S_u$ and $S_v$ are closed sets, because $U^s, V^s$ are closed sets. $V^s$ and $U^s$ intersect, therefore $f^n[V^s]$ and $f^n[U^s]$ intersect, whence $S_v$ and $S_u$ intersect.


Recall that $x=y$. By definition, the map $t\mapsto \psi_x^{-1}f^{-n}\psi_{x_n}(t,G(t))$ maps $S_u$ homeomorphically onto $\psi_x^{-1}[U^s]$, and $S_v$ homeomorphically onto $\psi_x^{-1}[V^s]$. Since $V^s,U^s$ are admissible manifolds, $\psi_x^{-1}[V^s],\psi_x^{-1}[U^s]$ are homeomorphic to closed boxes in $\mathbb{R}^{s(x)}$, and as such to the closed unit ball in $\mathbb{R}^{s(x)}$. So, $S_v,S_u\subset\mathbb{R}^{s(x)}$ are homeomorphic to the closed unit ball in $\mathbb{R}^{s(x)}$, and $S_u\cap S_v\neq\varnothing$. Therefore, if $S_v\not\subset S_u$, then $\exists z\in(\partial S_u)\cap S_v$\footnote{Fix $a\in S_u\cap S_v$. Since $S_v\subset S_u$, $\exists y\in S_v\setminus S_u$. Let $\gamma:[0,1]\rightarrow S_v$ be a continuous curve such that $\gamma(0)=a,\gamma(1)=y$ (it exists since $S_v$ is homeomorphic to a closed ball). Let $t_0:=\inf\{t:\gamma(t)\notin S_u\}$, then $\gamma(t_0)\in \partial S_u\cap S_v$.}. 



The map $t\mapsto f^{-n}(\psi_{x_n}(t,G(t))):S_u\rightarrow U^s$ is continuous and onto $U^s$. Therefore $f^{-n}(\psi_{x_n}(z,G(z)))$ belongs to the boundary of $U^s$ and to $V^s$. We now use the assumption that $x=y$ again. $\psi_x^{-1}[V^s]$ and $\psi_x^{-1}[U^s]$ are submanifolds of the chart $\psi_x^{p^s,p^u}$. The boundary points of $\psi_x^{-1}[U^s]$ have $s$-coordinates with $|\cdot|_\infty$-norm equal to $q^s$. On the other hand, $\psi_x^{-1}\Big(f^{-n}(\psi_{x_n}(z,G(z)))\Big)\in\psi_x^{-1}[V^s]$, hence its $s$-coordinates have $|\cdot|_\infty$-norm less or equal to $p^s$. This yields $q^s\leq p^s<q^s$, a contradiction. Therefore $S_v\subset S_u$.

We now address the case of $q^s=p^s$: Denote with $F$ the representing function of $V^s$. Define $\delta_0:=\frac{q^s}{2}$, $\forall\delta\in(0,\delta_0)$: $V^s_\delta:=\psi_x[\{(t,F(t))|t\in R_{q^s-\delta}(0)\}]$. It follows that for all $\forall\delta\in(0,\delta_0)$ $V^s_\delta \subset V^s$.
The argument we used in the case $q^s<p^s$ also shows that $V^s_\delta\subset U^s$. $\bigcup\limits_{\delta\in(0,\delta_0)}V^s_\delta=\mathring{V^s}\Rightarrow \mathring{V^s}\subset U^s\Rightarrow \overline{\mathring{V^s}}\subset\overline{U^s}$. Since $V^s,U^s$ are homeomorphic to the closed unit ball in $\mathbb{R}^{s(x)}$, $V^s\subset U^s$.
\end{proof}

\section{Markov partitions and symbolic dynamics \cite{B3,B4,Sarig}}
In sections 1,2,3 of this paper we developed the shadowing theory needed to construct Markov partitions \`a la Bowen (\cite{B3,B4}). What remains is to carry out this construction. This is the content of this part. We follow \cite{B3} and \cite{Sarig} closely.
\subsection{A locally finite countable Markov partition}
\subsubsection{The cover}
In subsection 1.3 we constructed a countable Markov shift $\Sigma$ with countable alphabet $\mathcal{V}$ and a H\"older continuous map $\pi:\Sigma\rightarrow M$ which commutes with the left shift $\sigma:\Sigma\rightarrow\Sigma$, so that $\pi[\Sigma]$ has full measure with respect to all $\chi$-hyperbolic measures. Moreover, using the convention that any element of $\mathcal{V}$ is relevant, we get that if:
$$\Sigma^\#=\{u\in\Sigma: u \text{ is a regular chain}\}=\{u\in\Sigma: \exists v,w\in \mathcal{V}\exists n_k,m_k\uparrow\infty\text{ s.t.   }u_{n_k}=v,u_{-m_k}=w\}$$
then $\pi[\Sigma^{\#}]\supset NUH_\chi^{\#}$. Therefore $\pi[\Sigma^{\#}]$ has full probability w.r.t any $\chi$-hyperbolic measure. But $\pi$ is not finite-to-one. In this section we study the following countable cover of $NUH_\chi^\#$:
\begin{definition}\label{ZV} $\mathcal{Z}:=\{Z(v):v\in\mathcal{V}\}$, where $Z(v):=\{\pi(u):u\in\Sigma^\#,u_0=v\}$.
\end{definition}
\begin{theorem}\label{Zlocallyfinite}
For every $Z\in\mathcal{Z}$, $|\{Z'\in\mathcal{Z}:Z'\cap Z\neq \varnothing\}|<\infty$.
\end{theorem}
\begin{proof} The proof is the same as in \cite[theorem~10.2]{Sarig}, but we include it to show the purpose of the results of \textsection\ref{chapter333}.

Fix some $Z=Z(\psi_x^{p^s,p^u})$. If $Z'=Z(\psi_y^{q^s,q^u})$ intersects $Z$ then there must exist two chains $v,w\in\Sigma^\#$ s.t. $v_0=\psi_x^{p^s,p^u},w_0=\psi_y^{q^s,q^u}$ and $\pi(v)=\pi(w)$. Lemma \ref{prop214} says that in this case $q^u\geq e^{-\sqrt[3]{\epsilon}}p^u,q^s\geq e^{-\sqrt[3]{\epsilon}}p^s$.
It follows that $Z'$ belongs to $\{Z(\psi_y^{q^s,q^u}):\psi_y^{q^s,q^u}\in\mathcal{V},q^u\wedge q^s\geq e^{-\sqrt[3]{\epsilon}}(p^u\wedge p^s)\}$. By definition, the cardinality of this set is less than or equal to:
$$|\{\psi_y^\eta\in\mathcal{A}:\eta\geq e^{-\sqrt[3]{\epsilon}}(p^s\wedge p^u)\}|\times|\{(q^s,q^u)\in I_\epsilon\times I_\epsilon: q^s\wedge q^u\geq e^{-\epsilon^\frac{1}{3}}(p^s\wedge p^u)\}|.$$
This is a finite number because of the discreteness of $\mathcal{A}$ (proposition  \ref{discreteness}).
\end{proof}
\subsubsection{Product structure}
Suppose $x\in Z(v)\in \mathcal{Z}$, then $\exists u\in\Sigma^\#$ s.t. $u_0=v$ and $\pi(u)=x$. Associated to $u$ are two admissible manifolds in $v$: $V^u((u_i)_{i\leq 0})$ and $V^s((u_i)_{i\geq 0})$. These manifolds do not depend on the choice of $u$: If $w\in\Sigma^\#$ is another  chain s.t. $w_0=v$ and $\pi(w)=x$ then $$V^u((u_i)_{i\leq 0})=V^u((w_i)_{i\leq 0})\text{ and }V^s((u_i)_{i\geq 0})=V^s((w_i)_{i\geq 0}),$$
because of proposition \ref{prop221} and the equalities $p^{s/u}(w_0)=p^{s/u}(u_0)=p^{s/u}(v)$. We are therefore free to make the following definitions:
\begin{definition}
Suppose $Z=Z(v)\in\mathcal{Z}$. For any $x\in Z$:
\begin{enumerate}
    \item $V^u(x,Z):=V^u((u_i)_{i\leq 0})$ for some (any) $u\in\Sigma^\#$ s.t. $u_0=v$ and $\pi(u)=x$.
    
    We also set $W^u(x,Z):=V^u(x,Z)\cap Z$.
    \item $V^s(x,Z):=V^s((u_i)_{i\geq 0})$ for some (any) $u\in\Sigma^\#$ s.t. $u_0=v$ and $\pi(u)=x$.
    
    We also set $W^s(x,Z):=V^s(x,Z)\cap Z$.
\end{enumerate}
\end{definition}
\begin{prop}\label{forBuzzi1} Suppose $Z\in\mathcal{Z}$, then for any $x,y\in Z$: $V^u(x,Z),V^u(y,Z)$ are either equal or they are disjoint. Similarly for $V^s(x,Z),V^s(y,Z)$; and for $W^s(x,Z),W^s(y,Z)$ and for $W^u(x,Z),W^u(y,Z)$.
\end{prop}
\begin{proof}
The statement holds for $V^{s/u}(x,Z)$ because of proposition \ref{prop221}. The statement for $W^{s/u}$ is an immediate corollary.
\end{proof}
\begin{prop}\label{propForBracketZ} Suppose $Z\in\mathcal{Z}$, then for any $x,y\in Z$: $V^u(x,Z)$ and $V^s(y,Z)$ intersect at a unique point $z$, and $z\in Z$. Thus $W^u(x,Z)\cap W^s(y,Z)=\{z\}$.
\end{prop}
\begin{proof} See proposition 10.5 in \cite{Sarig}.
\end{proof}
\begin{definition}\label{bracketZ} The {\em Smale bracket} of two points $x,y\in Z\in\mathcal{Z}$ is the unique point $[x,y]_Z\in W^u(x,Z)\cap W^s(y,Z)$.
\end{definition}
\begin{lemma} Suppose $x,y\in Z(v_0)$, and $f(x),f(y)\in Z(v_1)$. If $v_0\rightarrow v_1$ then $f([x,y]_{Z(v_0)})=[f(x),f(y)]_{Z(v_1)}$.
\end{lemma}
\begin{proof} This is proved as in \cite[lemma~10.7]{Sarig} using theorem \ref{graphtransform}(2).
\end{proof}

\begin{lemma}
The following holds for all $\epsilon$ small enough: Suppose $Z,Z'\in\mathcal{Z}$. If $Z\cap Z'\neq\varnothing$ then for any $x\in Z,y\in Z'$: $V^s(x,Z)$ and $V^s(y,Z')$ intersect at a unique point.

We do not claim that this point is in $Z$ nor $Z'$.
\end{lemma}
\begin{proof}Suppose $Z=Z(\psi_{x_0}^{p_0^s,p_0^u}),Z'=Z(\psi_{y_0}^{q_0^s,q_0^u})$ intersect. We are asked to show that for every $x\in Z$ and $y\in Z'$: $V^u(x,Z)$ and $V^s(y,Z')$ intersect at a unique point.
Fix some $z\in Z\cap Z'$, then there are $v,w\in \Sigma^\#$ s.t. $v_0=\psi_{x_0}^{p_0^s,p_0^u},w_0=\psi_{y_0}^{q_0^s,q_0^u}$ and $z=\pi(v)=\pi(w)$. Write $p:=p_0^s\wedge p_0^u$ and $q:=q_0^s\wedge q_0^u$. By theorem \ref{beforefinal} $p_0^u/q_0^u,p_0^s/q_0^s,p_0^s\wedge p_0^u/q_0^s\wedge q_0^u\in [e^{-\sqrt[3]{\epsilon}},e^{\sqrt[3]{\epsilon}}]$ and by corollary \ref{final}:
$$\psi_{x_0}^{-1}\circ\psi_{y_0}=O+a+\Delta\text{ on }R_{\max\{Q_\epsilon(x_0),Q_\epsilon(y_0)\}}(0),$$
where $O,a,\Delta$ are as in the notations of corollary \ref{final}. Now suppose $x\in Z$. $V^u=V^u(x,Z)$ is a $u$-admissible manifold in $\psi_{x_0}^{p_0^s,p_0^u}$, therefore it can be put in the form $V^u(x,Z)=\{(F(t),t):|t|_\infty<p_0^u\}$, where $F$ is the representing function. We write $V^u(x,Z)$ in $\psi_{y_0}$-coordinates. Denote the standard orthonormal basis of $\mathbb{R}^d$ by $\{e_i\}_{i=1,...,d}$, and $\mathbb{R}_{s/u}:=\spn\{e_i\}_{i=1,..,s(x)}/\spn\{e_i\}_{i=s(x)+1,...,d}$. By theorem \ref{beforefinal}, $O[\mathbb{R}_s]=\mathbb{R}_s, O[\mathbb{R}_u]=\mathbb{R}_u$. 
Let $a=a_s+a_u$, $a_{s/u}\in \mathbb{R}_{s/u}$ and $\Delta=\Delta_s+\Delta_u$ where $\Delta_{s/u}:\mathbb{R}^d\rightarrow \mathbb{R}_{s/u}$, then
\begin{align*}V^u(x,Z)=&\Big(\psi_{y_0}\circ(\psi_{y_0}^{-1}\circ\psi_{x_0})\Big)[\{(F(t),t):|t|_\infty\leq p_0^u\}]\\
=&\psi_{y_0}[\{O(F(t),t)+a_s+\Delta_s((F(t),t))+a_u+\Delta_d((F(t),t)):|t|_\infty\leq p_0^u\}]\\
=&\psi_{y_0}[\{(\widetilde{F}(\theta)+a_s+\widetilde{\Delta}_s((\widetilde{F}(\theta),\theta))\text{ , }\theta+a_u+\widetilde{\Delta}_u((\widetilde{F}(\theta),\theta))):|\theta|_\infty\leq p_0^u\}],\end{align*}
where we have used the transformation $\theta:=O(0,t),\widetilde{F}(s)=O(F(O^{-1}s),0),\widetilde{\Delta}(\vec{v})=\Delta(O^{-1}\vec{v})$; and the norm $|\theta|_\infty$ is calculated w.r.t to the new orthonormal basis after $e_i\mapsto Oe_i$.

Let $\tau(\theta):=\theta+a_u+\widetilde{\Delta}_u((\widetilde{F}(\theta),\theta))$. Assuming $\epsilon$ is small enough we have:
\begin{itemize}
    \item $\|d_\cdot\tau-Id\|\leq \sqrt[3]{\epsilon}\Rightarrow\exists (d_\cdot\tau)^{-1}\text{ and }\|(d_\cdot\tau)^{-1}\|^{-1},\|d_\cdot\tau\|\leq e^{2\sqrt[3]{\epsilon}}$; and in particular $\tau$ is 1-1.
    \item $|\tau(0)|_\infty<|a_u|_\infty+|\widetilde{\Delta}_u((\widetilde{F}(0),0))|_\infty<10^{-1}q+\sqrt[3]{\epsilon}10^{-3}p<\frac{1}{6}\min\{p,q\} (\because p\leq e^{\sqrt[3]{\epsilon}}q\text{ , }q\leq e^{\sqrt[3]{\epsilon}}p)$.
    \item Hence $R':=\tau[R_q(0)]\supset R_{e^{-2\sqrt[3]{\epsilon}}q}(0)+\tau(0)\supset R_{\frac{2}{3}q}(0)$, and in particular $0\in R'$.
\end{itemize}
Since $\tau:R_q(0)\rightarrow R'$ is 1-1 (as seen in the first item in the list) and onto, it has a well-defined inverse function $\zeta:R'\rightarrow R_q(0)$ which is also onto its target.

Let $G(s):=\widetilde{F}(\zeta(s))+a_s+\widetilde{\Delta}_s((\widetilde{F}(\zeta(s)),\zeta(s)))$, then
$$V^u(x,Z)=\psi_{y_0}[\{(G(s),s):s\in R'\}].$$
Here, again, the decomposition into orthogonal components is by $O[\mathbb{R}_s]\perp O[\mathbb{R}_u]$.
Using the properties of $\tau$ it is not hard to check that
\begin{itemize}
    \item $\|d_\cdot\zeta^{-1}\|^{-1},\|d_\cdot\zeta\|\leq e^{2\sqrt[3]{\epsilon}}$
    \item $\mathrm{Rad}_{\|\cdot\|_\infty}(R'):=\max_{t\in R'}\{|t|_\infty\}\leq \|d_\cdot \tau\|\cdot \mathrm{Rad}_{\|\cdot\|_\infty}(R_q(0))+|\tau(0)|\leq e^{2\sqrt[3]{\epsilon}}q+\frac{1}{6}p\leq (1+\frac{1}{6})e^{2\sqrt[3]{\epsilon}}q<2q$
    \item $|\zeta(0)|_\infty\leq \|d_\cdot\zeta\|\cdot|\tau(0)-0|+|\zeta(\tau(0))|\leq e^{2\sqrt[3]{\epsilon}}\frac{1}{6}p+0\leq\frac{1}{5}p$ (this is well defined since $0\in R_q(0)\Rightarrow \tau(0)\in R'$ ; and by the third item in the previous list $0\in R'$)
    \item Since $|\zeta(0)|\leq\frac{1}{5}p$, $\widetilde{F}\circ\zeta$ is well defined on $0$, and thus:
    
    $|\widetilde{F}\circ\zeta(0)|\leq|\widetilde{F}\circ\zeta(\tau(0))|+\|d_\cdot (\widetilde{F}\circ\zeta)\|\cdot|\tau(0)-0|\leq|\widetilde{F}(0)|+\|d_\cdot \widetilde{F}\|\cdot\|d_\cdot\zeta\|\cdot|\tau(0)|=$
    
    $=|F(0)|+\|d_\cdot F\|\cdot\|d_\cdot\zeta\|\cdot|\tau(0)|\leq 10^{-3}p+\epsilon\cdot e^{2\sqrt[3]{\epsilon}}\cdot\frac{1}{6}p\leq 10^{-2}p$
\end{itemize}

Since $|\widetilde{F}(\zeta(0))|_\infty\leq 10^{-2}p$, we get $|G(0)|_\infty\leq10^{-2}p+10^{-1}q+\sqrt[3]{\epsilon}p<\min\{\frac{1}{6}p,\frac{1}{6}q\}(\because \frac{q}{p}\in[e^{-\sqrt[3]{\epsilon}},e^{\sqrt[3]{\epsilon}}]).$
We also know $\|d_\cdot G\|\leq \|d_\cdot \widetilde{F}\|\cdot\|d_\cdot \zeta\|+\epsilon^\frac{1}{3}\sqrt{1+\|d_\cdot \widetilde{F}\|^2}\cdot\|d_\cdot \zeta\|<2\sqrt[3]{\epsilon}.$

It follows that for all $\epsilon$ small enough: $G[R_{\frac{2}{3}p}(0)]\subset R_{2\epsilon^\frac{1}{3}\frac{2}{3}p+\frac{1}{6}p}(0)\subset R_{\frac{2}{3}p}(0)$. We can now show that $|V^u(x,Z)\cap V^s(y,Z')|\geq1$: $V^s(y,Z')$ can be represented as $\psi_{y_0}[\{(t,H(t)): |t|_\infty\leq q_0^s\}]$. By admissibility, $H(0)<10^{-3}q,\|d_\cdot H\|<\epsilon$. So as before: $H[R_{\frac{2}{3}p}(0)]\subset R_{\frac{2}{3}p}(0)$. It follows that $H\circ G$ is a contraction of $R_{\frac{2}{3}q}(0)$ to itself ($\mathrm{Lip}(H\circ G)\leq \mathrm{Lip}(H)\cdot \mathrm{Lip}(G)\leq 2\epsilon^\frac{1}{3}\cdot\epsilon$). Such a function has a unique fixed point $s_0$ s.t. $H\circ G(s_0)=s_0$. It is easy to see that $\psi_{y_0}((G(s_0),s_0))$ belongs to $V^u(x,Z)\cap V^s(y,Z')$.

Next we claim that $V^u(x,Z)\cap V^s(y,Z')$ contains at most one point: Extend $G$ and $H$ to $\epsilon$-Lipschitz functions $\widetilde{G}$ and $\widetilde{H}$ on $A$ using McShane's extension formula \cite{McShane} (as demonstrated in footnote \footref{lipextensions}),
where $A:=R'\cup R_{q_0^s}(0)$. We saw: $\mathrm{Rad}_{\|\cdot\|_\infty}(R')<2q<2q_0^s$, whence $A\subset R_{2q_0^s}(0)$. We also know $|\widetilde{G}(0)|_\infty<\frac{1}{6}q\leq\frac{1}{6}q_0^s$ and $\mathrm{Lip}(G)<2\epsilon^\frac{1}{3}$, thus $\widetilde{G}[A]\subset\widetilde{G}[R_{2q_0^s}(0)]\subset R_{2\epsilon^\frac{1}{3}\cdot2q_0^s+\frac{1}{6}q_0^s}(0)\subset R_{q_0^s}(0)\subset A$. Also $|\widetilde{H}(0)|_\infty<10^{-3}q_o^s$, so $\widetilde{H}:A\rightarrow A$ by similar calculations. It follows that $\widetilde{H}\circ\widetilde{G}$ is a contraction of $A$ into itself, and therefore it has a unique fixed point. Every point in $V^u(x,Z)\cap V^s(y,Z')$ takes the form $\psi_{y_0}((G(s),s))$ where $s\in R'$ and $s=H\circ G(s)\equiv \widetilde{H}\circ \widetilde{G}(s)$. Since the equation $\widetilde{H}\circ \widetilde{G}(s)=s$ has at most one solution in $A$, it has at most one solution in $R'$. It follows that $|V^u(x,Z)\cap V^s(y,Z')|\leq1$.
\end{proof}
\subsubsection{The symbolic Markov property}
\begin{prop}\label{symbolicmarkovproperty}
If $x=\pi((v_i)_{i\in\mathbb{Z}})$, where $v\in\Sigma^\#$, then $f[W^s(x,Z(v_0))]\subset W^s(f(x),Z(v_1))$ 

and $f^{-1}[W^u(f(x),Z(v_1))]\subset W^u(x,Z(v_0))$.
\end{prop}
\begin{proof} See proposition 10.9 in \cite{Sarig}.
%
%
%
%
\end{proof}
\begin{lemma}\label{lastOne} Suppose $Z,Z'\in\mathcal{Z}$ and $Z\cap Z'\neq\varnothing$, then:
\begin{enumerate}
    \item If $Z=Z(\psi_{x_0}^{p_0^s,p_0^u})$ and $Z'=Z(\psi_{y_0}^{q_0^s,q_0^u})$ then $Z\subset \psi_{y_0}[R_{q_0^s\wedge q_0^u}(0)]$
    \item For any $x\in Z\cap Z'$: $W^u(x,Z)\subset V^u(x,Z')$ and $W^s(x,Z)\subset V^s(x,Z')$
\end{enumerate}
\end{lemma}
\begin{proof}
Fix some $x\in Z\cap Z'$. Write $x=\pi(v),x=\pi(w)$ where $v,w\in\Sigma^\#$ satisfy $w_0=\psi_{x_0}^{p_0^s,p_0^u}$ and $w_0=\psi_{y_0}^{q_0^s,q_0^u}$. Write $p:=p_0^s\wedge p_0^u,q:=q_0^s\wedge q_0^u$. Since $\pi(v)=\pi(w)$ we have by theorem \ref{beforefinal} $\frac{p}{q}\in[e^{-\epsilon^\frac{1}{3}},e^{\epsilon^\frac{1}{3}}]$ and $$\psi_{y_0}^{-1}\circ\psi_{x_0}=O+a+\Delta\text{ on }R_\epsilon(0).$$ 
Here $O$ is an orthogonal transformation, $a$ is a constant vector s.t. $|a|<10^{-1}q$ and $\Delta:R_\epsilon(0)\rightarrow\mathbb{R}^d$ satisfies $\Delta(0)=0$, $\|d_\cdot \Delta\|<\epsilon^\frac{1}{3}$ (and hence also $|\Delta(u)|\leq\epsilon^\frac{1}{3}|u|$ for all $u\in R_\epsilon(0)$). Recall that all vector norms are infinity norms, and all operators norms are calculated w.r.t to infinity norms.

Every point in $Z$ is the intersection of a $u$-admissible and an $s$-admissible manifolds in $\psi_{x_0}^{p_0^s,p_0^u}$, therefore $Z$ is contained in $\psi_{x_0}[R_{10^{-2}p}(0)]$ (proposition \ref{firstbefore}). Thus
$$Z\subset\psi_{y_0}[(\psi_{y_0}^{-1}\circ\psi_{x_0})[R_{10^{-2}p}(0)]\subset \psi_{y_0}[R_{(1+\epsilon^\frac{1}{3})10^{-2}p}(a)]\subset$$
$$\subset \psi_{y_0}[R_{(1+\epsilon^\frac{1}{3})10^{-2}p+10^{-1}q}(0)]\subset\psi_{y_0}[B_{(1+\epsilon^\frac{1}{3})10^{-2}qe^{\epsilon^\frac{1}{3}}+10^{-1}q}(0)]\subset\psi_{y_0}[R_{q}(0)]\text{  }(\because \epsilon<1).$$
This proves the first statement of the lemma. Next we show $W^s(x,Z)\subset V^s(x,Z')$: Write $v_i=\psi_{x_i}^{p_i^s,p_i^u}$ and $w_i=\psi_{y_i}^{q_i^s,q_i^u}$. Since $\pi(v)=x$ and $Z=Z(v_0)$, the symbolic Markov property gives us:
$$f^k[W^s(x,Z)]\subset W^s(f^k(x),Z(v_k))\text{  }  (k\geq0).$$
The sets $Z(v_k)$ and $Z(w_k)$ intersect, because they both contain $f^k(x)$. By the first part of the lemma, $Z(v_k)\subset\psi_{y_k}[R_{q_k^s\wedge q_k^u}(0)]$. It follows that
$f^k[W^s(x,Z)]\subset \psi_{y_k}[R_{q_k^s\wedge q_k^u}(0)]\subset\psi_{y_k}[R_{Q_\epsilon(y_k)}(0)]$
for all $k\geq0$. By proposition \ref{firstofchapter}(4): $W^s(x,Z)\subset V^s((w_i)_{i\geq0})\equiv V^s(x,Z')$.
\end{proof}
\section{Proof of the main results}\label{chapter5}


We present a summary of the objects we have constructed so far, and their properties.

We have constructed a family of sets $\mathcal{Z}$ s.t.
\begin{enumerate}
    \item $\mathcal{Z}$ is \textbf{countable}  (definition \ref{ZV}; see also proposition \ref{discreteness} and definition \ref{graphosaurus}).
    \item $\mathcal{Z}$ covers $NUH_\chi^\#$, in particular $\mathcal{Z}$ covers a set of \textbf{full measure} for every ergodic $\chi$-hyperbolic $f$-invariant measure (theorem \ref{DefOfPi} and proposition \ref{sigmarelprop}).
    \item $\mathcal{Z}$ is \textbf{locally finite}: $\forall Z\in \mathcal{Z}, \#\{Z'\in\mathcal{Z}:Z\cap Z'\neq\varnothing\}<\infty$  (theorem \ref{Zlocallyfinite}).
    \item Every $Z\in\mathcal{Z}$ has \textbf{product structure}: There are subsets $W^{s}(x,Z),W^u(x,Z)\subset Z$ $(x\in Z)$
s.t. \begin{enumerate}
    \item $x\in W^{s}(x,Z)\cap W^u(x,Z)$.
    \item $\forall x,y\in Z \exists ! z\in Z$ s.t. $W^u(x,Z)\cap W^s(y,Z)=\{z\}$, we write $z=[x,y]_Z.$
    \item $\forall x,y\in Z$ $W^u(x,Z),W^u(y,Z)$ are equal or disjoint (similarly for $W^s(x,Z),W^s(y,z)$).
\end{enumerate}
 (proposition \ref{forBuzzi1}, proposition \ref{propForBracketZ}).
\item The \textbf{symbolic Markov property}: if $x\in Z(u_0), f(x)\in Z(u_1)$,  and $u_0\rightarrow u_1$, then $f[W^s(x,Z_0)]\subset W^s(f(x),Z_1)$ and $f^{-1}[W^u(f(x),Z_1)]\subset W^u(x,Z_0)$ (proposition \ref{symbolicmarkovproperty}).\footnote{To apply proposition \ref{symbolicmarkovproperty}, represent $x=\pi(v)$ with $v\in\Sigma^\#$, $v_0=u_0$, and $f(x)=\pi(w)$ with $w\in\Sigma^\#$, $w_0=u_1$; and note that $x=\pi(u), u\in \Sigma^\#$ where $u_i=v_i$ $(i\leq0)$, $u_i=w_{i-1}$ $(i\geq1)$.}
\item 
\textbf{Uniform expansion/contraction on $W^{s}(x,Z),W^{u}(x,Z)$}: 
$\exists \theta\in (0,1), C>0$ s.t. $\forall Z\in \mathcal{Z}, \forall z\in Z$, $\forall x\in W^s(z,Z),y\in W^u(z,Z)$, $\forall n\geq0$, $d(f^n(x),f^n(z)),d(f^{-n}(y), f^{-n}(z))\leq \frac{C}{2}\theta^n$ (proposition \ref{Lambda}(1)).

\item 
\textbf{Bowen's property with finite degree} (see \cite[definition~1.11]{BoyleBuzzi}): \begin{enumerate}
    \item 
    $\forall u,v\in \Sigma:$ $\Big(\forall i, Z(u_i)\cap Z(v_i)\neq\varnothing\Big)\Rightarrow \pi(u)=\pi(v)$  (proposition \ref{firstofchapter}(4), proof of proposition \ref{prop213}).
    \item $\forall u,v\in \Sigma^\#$: $\pi(u)=\pi(v)\Leftrightarrow \Big(Z(u_i)\cap Z(v_i)\neq\varnothing$ $\forall i\in\mathbb Z\Big)$ ((a), definition \ref{ZV}).
    \item $\forall Z\in \mathcal{Z}$:  $\#\{Z'\in\mathcal{Z}: Z'\cap Z\neq\emptyset\}<\infty$  (theorem \ref{Zlocallyfinite}).
\end{enumerate}
\end{enumerate}
We explain how to use these constructs to prove our main results, theorems \ref{t4.1.1}-\ref{t4.2.2}.

\noindent\textbf{\underline{Proof of theorem \ref{t4.1.1}:}}

\underline{Step 1:} $\mathcal{Z}$ can be refined into a countable \textbf{partition} $\mathcal{R}$ with the following properties:
\begin{enumerate}
    \item[(a)] Product structure.
    \item[(b)] Symbolic Markov property.
    \item[(c)] Every $Z\in\mathcal{Z}$ contains only finitely many $R\in\mathcal{R}$.
\end{enumerate}

\underline{Proof:} This is done exactly as in \cite[\textsection~11]{Sarig}, using the Bowen-Sinai refinement procedure \cite{B1,B4,B3}, the local-finiteness property (3), and the product structure (4) of $Z_i\in\mathcal{Z}$. The partition elements are the equivalence classes of the following equivalence relation on $\bigcup\mathcal{Z}$ (see \cite[proposition~11.2]{Sarig}):
$$x\sim y \iff \forall Z,Z'\in \mathcal{Z}, 
\begin{bmatrix}
    x\in Z\iff y\in Z  \\
    W^u(x,Z)\cap Z'\neq\varnothing \iff W^u(y,Z)\cap Z'\neq\varnothing   \\
     W^s(x,Z)\cap Z'\neq\varnothing \iff W^s(y,Z)\cap Z'\neq\varnothing
\end{bmatrix}.$$


\underline{Step 2:} For every finite chain $R_0\rightarrow\cdots\rightarrow R_m$,  $\bigcap_{j=0}^{m}f^{-j}[R_j]\neq\varnothing$.

\underline{Proof:} This follows from the fact that $\mathcal{R}$ is a Markov partition, using a classical argument due to Adler and Weiss, and Sinai. The proof is by induction on $m$. Assume the statement holds for $m$, and prove for $m+1$: Take $p\in f^m\Big[\bigcap_{j=0}^{m}f^{-j}[R_j]\Big]$, and take $p'\in R_{m}\cap f^{-1}R_{m+1}$. $p,p'\in R_{m}$, whence belong to some $Z\in\mathcal{Z}$; therefore $[p,p']_Z$ is well defined by definition \ref{bracketZ} and proposition \ref{propForBracketZ}. By the symbolic Markov property, $f^{-m}([p,p']_Z)\in \bigcap_{j=0}^{m+1}f^{-j}[R_{j}]$.

\underline{Step 3:} Let $\widehat{\mathcal{G}}$ be a directed graph with set of vertices $\mathcal{R}$ and set of edges $\{(R,S)\in \mathcal{R}\times\mathcal{R}: R\cap f^{-1}[S]\neq\varnothing\}$. Let $\widehat{\Sigma}:=\Sigma(\widehat{\mathcal{G}})$ and define $\widehat{\pi}:\widehat{\Sigma}\rightarrow M$ by $\{\widehat{\pi}(\vec{R})\}=\bigcap\limits_{n\geq0}\overline{\bigcap\limits_{i=-n}^n f^{-i}[R_i]}$. Then $\widehat{\pi}$ is well defined, and:
\begin{enumerate}
    \item[(a)] $\widehat{\pi}[\widehat{\Sigma}]\supset NUH_\chi^\#$.
    \item[(b)] $\widehat{\pi}\circ\sigma=f\circ\widehat{\pi}$.
    \item[(c)] $\widehat{\pi}$ is H\"older continuous.
    \item[(d)] $\widehat{\Sigma}$ is locally compact.
\end{enumerate}

\underline{Proof:} This proof is a concise presentation of \cite[theorem~12.3, theorem~12.5]{Sarig}. Step 2 guarantees that the intersection in the expression defining $\{\widehat{\pi}\}$ is not empty. 
The intersection is a singleton, because $\mathrm{diam}(\bigcap_{i=-n}^nf^{-i}[R_i])\leq C\theta^n$. To see this, fix $x,y\in \bigcap_{i=-n}^nf^{-i}[R_i]$, and fix $u\in\Sigma^\#$ s.t. $\pi(u)=x$. It follows that $\bigcap_{i=-n}^nf^{-i}[R_i]\subset \bigcap_{i=-n}^nf^{-i}[Z(u_i)]$, whence $y\in\bigcap_{i=-n}^nf^{-i}[Z(u_i)]$. Let $z:=[x,y]_{Z(u_0)}$. By (4) and (5), $z\in \bigcap_{i=-n}^nf^{-i}[Z(u_i)]$, $f^n(z)\in W^u(f^n(x),Z(u_n))$, $f^{-n}(z)\in W^s(f^{-n}(y),Z(u_{-n}))$, whence by (6) $d(x,y)\leq d(x,z)+d(z,y)=d(f^{-n}(f^n(x)),f^{-n}(f^n(z)))+d(f^{n}(f^{-n}(z)),f^{n}(f^{-n}(y)))\leq C\theta^n$.
Therefore $\widehat{\pi}$ is well defined. 
\begin{enumerate}
    \item[(a)] Let $x\in NUH_\chi^\#$. By theorem \ref{DefOfPi}, $f^i(x)\in\pi[\Sigma^\#]$ for all $i\in\mathbb{Z}$. By definition $\mathcal{Z}$ is a cover of $\pi[\Sigma^\#]$, and $\mathcal{R}$ refines $\mathcal{Z}$; whence $f^i(x)\in\bigcupdot\mathcal{R}$ $\forall i\in\mathbb{Z}$. Denote by $R_i$ the unique partition member of $\mathcal{R}$ which contains $f^i(x)$. $\vec{R}$ is an admissible chain by definition, and $\widehat{\pi}(\vec{R})=x$. Let $\vec{u}\in\Sigma^\#$ s.t. $\pi(\vec{u})=x$, then $Z(u_i)\supset R_i$ $\forall i\in\mathbb{Z}$. $\mathcal{Z}$ is locally-finite, so by the pigeonhole principle, $\vec{R}$ must be in $\widehat{\Sigma}^\#$.
    \item[(b)] $$\{(f\circ\widehat{\pi})(\vec{R})\}=\bigcap_{n\geq0}\bigcap_{i=-n}^n\overline{f^{-(i-1)}[R_{i-1+1}]}=\{\widehat{\pi}(\sigma \vec{R})\}\text{ }(\because f\text{ is a diffeomorphism}).$$
    \item[(c)] Suppose $d(\vec{R},\vec{S})=e^{-n}$, then $R_i=S_i$ for $|i|\leq n$, whence $\widehat{\pi}(\vec{R}),\widehat{\pi}(\vec{S})\in\overline{\bigcap_{i=-n}^nf^{-i}[R_i]}$. We just saw in the begining of the proof of step 3 that $\mathrm{diam}(\bigcap_{i=-n}^nf^{-i}[R_i])\leq C\theta^n$. So $d(\widehat{\pi}(\vec{R}),\widehat{\pi}(\vec{S}))\leq C\theta^n=C\cdot d(\vec{R},\vec{S})^{\log\frac{1}{\theta}}$.
    \item[(d)] Fix a vertex $R\in \mathcal{R}$. Assume $R\rightarrow S$. Let $u\in \mathcal{V}$ s.t. $R\subset Z(u)$. Take $x\in R\cap f^{-1}[S]\subset Z(u)$, then $\exists \vec{u}\in\Sigma^\#$ s.t. $u_0=u$ and $\pi(\vec{u})=x$. Then $f(x)\in S\cap Z(u_1)$. Thus, by step 1(c), and lemma \ref{lemma131}, $$\#\{S\in\mathcal{R}:R\rightarrow S\}\leq\sum_{u\in \mathcal{V}:R\subset Z(u)}\sum_{v\in\mathcal{V}:u\rightarrow v}\#\{S\in\mathcal{R}:S\subset Z(v)\}
    <\infty.$$
    A similar calculation shows $\#\{S':S'\rightarrow R\}<\infty$.
\end{enumerate}

\underline{Step 4:} $\widehat{\pi}[\widehat{\Sigma}^\#]$ has full measure for all $\chi$-hyperbolic measures, and $\widehat{\pi}|_{\widehat{\Sigma}^\#}$ is finite-to-one.

\underline{Proof:} The fact that $\widehat{\pi}[\widehat{\Sigma}^\#]$ has full measure for all $\chi$-hyperbolic measures follows from (a) in step 3, and the Poincar\'e recurrence theorem and the Oseledec theorem which say that $NUH_\chi^\#$ carries all $\chi$-hyperbolic measures (see definition \ref{nuhchisharp}). The finite-to-one property follows from theorem \ref{t4.1.2} which we prove below.

\medskip
\noindent\textbf{\underline{Proof of theorem \ref{t4.1.3}:}} Given $x\in\bigcupdot\mathcal{R}=\pi[\Sigma^\#]$, define $R(x):=$ the unique element of $\mathcal{R}$ to which $x$ belongs; and define $\vec{R}(x):=(R(f^i(x)))_{i\in\mathbb{Z}}$. The following properties of $\vec{R}(\cdot)$ follow immediately:
 \begin{enumerate}
     \item[(a)] $\vec{R}(\cdot)$ is a measurable map from $NUH_\chi^\#\subset\bigcupdot\mathcal{R}$ to $\widehat{\Sigma}$.
     \item[(b)] $\sigma\circ\vec{R}=\vec{R}\circ f$.
     \item[(c)] $\vec{R}(\cdot)$ is one-to-one onto its image.
 \end{enumerate}
Given an invariant $\chi$-hyperbolic measure $\mu$ on $M$, define $\nu:=\mu\circ \vec{R}^{-1}$ an invariant measure on $\widehat{\Sigma}$. By (b), if $\mu$ is ergodic, so is $\nu$ (in addition $\nu$ is carried by $\widehat{\Sigma}^\#$, as demonstrated in step 3(a) of the proof of theorem \ref{t4.1.1} above). By (c), $\nu$ has the same metric entropy as $\mu$. Since $\widehat{\pi}\circ\vec{R}=Id$, $\nu\circ\widehat{\pi}^{-1}=\mu$.

\medskip
\noindent\textbf{\underline{Proof of theorem \ref{t4.1.2}:}}
This proof follows  \cite[theorem~12.8]{Sarig} (see \cite{B1,B2,B3,BoyleBuzzi}). The proof in \cite{Sarig} had a mistake, which was later corrected in \cite[theorem~5.6(5)]{SL14}.

$R,S\in\mathcal{R}$ are called {\em affiliated} if $\exists Z_1,Z_2\in\mathcal{Z}$ s.t. $Z_1\supset R,Z_2\supset S$ and $Z_1\cap Z_2\neq\varnothing$ (see \cite[\textsection~12.3]{Sarig}).

\underline{Step 1:} Define $N(R):=\#\{S:S\text{ is affiliated to }R\}$. By the local-finiteness of $\mathcal{Z}$ and step 1(c) in the proof of theorem \ref{t4.1.1}, $N(R)<\infty$ for all $R\in\mathcal{R}$. Consider a chain $\vec{R}'\in\widehat{\Sigma}^\#$, then there exist $R,S\in\mathcal{R}$ s.t. $R'_i=R$ for infinitely many $i\geq0$, $R'_i=S$ for infinitely many $i\leq0$
. Denote all chains in $\widehat{\Sigma}^\#$ which shadow the same point as $\vec{R}'$ (including $\vec{R}'$) by $\{\vec{S}^{(k)}\}_{k\in K}$.

\underline{Step 2:} For all $k\in K $, and for all $i\in\mathbb{Z}$, $R'_i$ and $S_i^{(k)}$ are affiliated.

\underline{Proof:} Given $k\in K$, consider two chains $\vec{v},\vec{u}\in \Sigma$ s.t. $\forall i\in\mathbb{Z}$, $Z(u_i)\supset R'_i, Z(v_i)\supset S_i^{(k)}$ and $\pi(\vec{u})=\widehat{\pi}(\vec{R}')=\widehat{\pi}(\vec{S}^{(k)})=\pi(\vec{v})$.\footnote{Fix some $u_0\in\mathcal{V}$ s.t. $Z(u_0)\supset R_0$. $R_0\rightarrow R_{1}$, so there exists a point $z\in f^{-1}[R_{1}]\cap R_0$. In particular $z\in Z(u_0)$, and so there exists by definition a chain $\vec{u}'\in \Sigma^\#$ s.t. $\pi(\vec{u}')=z$ and $u_0'=u_0$; then $Z(u_1')\supset R_1$ and $u_0\rightarrow u_1'$. Set $u_1:= u_1'$. Repeating this procedure inductively (forwards and backwards) gives us an admissible chain $\vec{u}\in \Sigma$ s.t. $Z(u_i)\supset R_i$ for all $i\in\mathbb{Z}$. $\forall i\in\mathbb{Z}$, $f^i(\widehat{\pi}(\vec{R}))\in \overline{R_i}\subset\overline{Z(u_i)}$, whence $\vec{u}$ shadows $\widehat{\pi}(\vec{R})$, and so $\pi(\vec{u})=\widehat{\pi}(\vec{R})$ (see proposition \ref{firstofchapter}(4)). $\vec{v}$ is constructed similarly.} By the facts $\forall T\in\mathcal{R}$ $N(T)<\infty$ and  $\vec{R}',\vec{S}^{(k)}\in\widehat{\Sigma}^\#$ and by the pigeonhole principle, $\vec{u},\vec{v}\in\Sigma^\#$. Therefore $R'_i$ is affiliated to $S_i^{(k)}$ for all $i\in\mathbb{Z}$.

\underline{Step 3:} $\exists \vec{R},\vec{S}\in \{\vec{S}^{(k)}\}_{k\in K}$, $i_-\leq0,i_+\geq0$ s.t. $S_{i_-}=R_{i_-},S_{i_+}=R_{i_+}$ and $(S_{i_-},...,S_{i_+})\neq(R_{i_-},...,R_{i_+})$.

\underline{Proof:} Assume $|K|>N(R)N(S)$, then there exists a subset $\widetilde{K}\subset K$ s.t. $|\widetilde{K}|=N(R)N(S)+1$. We assume $\{\vec{S}^{(k)}\}_{k\in K}$ are all different, then there exists some $n$ s.t. $\forall k\neq l\in \widetilde{K}$, $(S_{-n}^{(l)},...,S_{n}^{(l)})\neq(S_{-n}^{(k)},...,S_{n}^{(k)})$. Denote by $i_+$ the first $i\geq n$ s.t. $R'_i=R$, and by $i_-$ the first $i\leq -n$ s.t. $R'_i=S$. By the pigeonhole principle and the choice of $\widetilde{K}$, there exist two chains $\vec{S},\vec{R}\in\widetilde{K}$ s.t. $S_{i_-}=R_{i_-}$ and $S_{i_+}=R_{i_+}$.

\underline{Step 4:} For all $m\geq0$, $\bigcap_{j=0}^{m}f^{-j}[S_{i_-+j}], \bigcap_{j=0}^{m}f^{-j}[R_{i_-+j}]\neq\varnothing$.

\underline{Proof:} See step 2 of the proof of theorem \ref{t4.1.1}.

\underline{Step 5:} $\exists z_R\in R_0, z_S\in S_0$ s.t. if $i>i_+$ or $i<i_-$, then $f^i(z_R),f^i(z_S)\in S_i$, and if $i_-\leq i\leq i_+$ then $f^i(z_S)\in S_i,f^i(z_R)\in R_i$.

\underline{Proof:} Using step 4, take $x\in \bigcap_{j=i_-}^{i_+}f^{-j}[S_{j}],y\in \bigcap_{j=i_-}^{i_+}f^{-j}[R_{j}]$. 
Define $t:=f^{-i_-}([f^{i_-}(x),f^{i_-}(y)]_{Z_1})$, $z_R:=f^{-i_+}([f^{i_+}(t),f^{i_+}(x)]_{Z_2})$ for some $Z_1,Z_2\in\mathcal{Z}$ s.t. $Z_1\supset S_{i_-}=R_{i_-},Z_2\supset S_{i_+}=R_{i_+}$. Let $z_S:=x$. By the symbolic Markov property, if $i>i_+$ or $i<i_-$, then $f^i(z_R),f^i(z_S)\in S_i$ 
, and if $i_-\leq i\leq i_+$, then $f^i(z_S)\in S_i,f^i(z_R)\in R_i$.

\underline{Step 6:} There cannot be more than $N(R)N(S)$ chains in $\widehat{\Sigma}^\#$ which shadow the same point as $\vec{R}'$.

\underline{Proof:} 
In addition to the property of $z_S, z_R$ stated in step 5, when $i\in [i_-,i_+]$,  $f^i(z_S)\in S_i$, $f^i(z_R)\in R_i$ and $S_i$  and $R_i$ are affiliated (as shown in step (2)). Therefore any chain $\vec{w}\in\Sigma$ which shadows $z_S$ must also shadow $z_R$ (Bowen's property with finite degree, see (7) in begining of \textsection \ref{chapter5}). So $z_S=z_R$. But $z_S\in \bigcap_{j=i_-}^{i_+}f^{-j}[S_{j}]$ and $z_R\in \bigcap_{j=i_-}^{i_+}f^{-j}[R_{j}]$ which are disjoint by the choice of $n$, a contradiction to the assumption $|K|>N(R)N(S)$.

\medskip
\noindent\textbf{\underline{Proof of theorem \ref{t4.2.1}:}} This follows the arguments in \cite[\textsection~1.1]{Sarig}, which have been done before by Katok and Buzzi in \cite{K3,Bu4}.

\underline{Step 1:} Let $\mu$ be an ergodic hyperbolic probability measure of maximal entropy on $f$. Let $\chi>0$ s.t. $\mu$ is $\chi$-hyperbolic. By theorems \ref{t4.1.1} and \ref{t4.1.3}, $\mu$ lifts to an invariant ergodic measure $\nu$ on a countable TMS $\widehat{\Sigma}$. $\nu$ must be a measure of maximal entropy as well of the same entropy (otherwise project a measure with greater entropy by the almost-everywhere-finite-to-one factor $\widehat{\pi}$, which preserves entropy, and get a measure on $M$ with entropy $>h_{top}(f)$). In particular, $h_{\max}(\sigma):=\sup\{h_{\nu'}(\sigma): \nu' \text{ is a shift-invariant probability}\}$ must be equal to $h_{top}(f)$.

\underline{Step 2:} $\nu$ is ergodic, thus it is carried by an irreducible component of $\widehat{\Sigma}$, $\Sigma(\widehat{\mathcal{G}}')$ s.t. $\widehat{\mathcal{G}}'$ is a subgraph of $\widehat{\mathcal{G}}$. The irreducibilty of $\Sigma(\widehat{\mathcal{G}}')$ means that $(\Sigma(\widehat{\mathcal{G}}'),\sigma)$ is topologically transitive.

\underline{Step 3:} By \cite{G1,G2}, a topologically transitive TMS admits at most one measure of maximal entropy; and such a measure must satisfy $$\exists p\in \mathbb{N}\forall R\text{ a vertex of }\widehat{\mathcal{G}}'\exists C_R>0\text{ s.t. }\#\{\vec{R}\in \Sigma(\widehat{\mathcal{G}}'): R_0=R,\sigma^n\vec{R}=\vec{R}\}=C_R^{\pm1}e^{n h_{\max}(\sigma)}\text{ }\forall n\in p\mathbb{N}.$$ By theorem \ref{t4.1.2}, each periodic point in $\{\vec{R}\in \Sigma(\widehat{\mathcal{G}}'): R_0=R,\sigma^n\vec{R}=\vec{R}\}$ is mapped by $\widehat{\pi}$ to a periodic point of the same period in $M$, with the mapping being at most $N(R)^2$-to-one.  Therefore, for any $R$ a vertex of $\widehat{\mathcal{G}}'$ and $n\in p\mathbb{N}$,
\begin{align*}
P_n(f)=&\#\{x\in M: f^n(x)=x\}
\geq \frac{1}{N(R)^2}\#\{\vec{R}\in \Sigma(\widehat{\mathcal{G}}'): R_0=R,\sigma^n\vec{R}=\vec{R}\}\geq \frac{1}{N(R)^2C_R}e^{n h_{top}(f)}.
\end{align*}

\textbf{\underline{Proof of theorem \ref{t4.2.2}:}} We have seen in step 2 of the proof of theorem \ref{t4.2.1}, that each $\chi$-hyperbolic ergodic measure of maximal entropy corresponds to a subgraph of $\widehat{\mathcal{G}}$ which induces an irreducible component of $\widehat{\Sigma}$. For a fixed $\chi$, all subgraphs of $\widehat{\mathcal{G}}$ which induce an irreducible component of $\widehat{\Sigma}$ are disjoint. Thus, there can be at most countably many disjoint subgraphs of $\widehat{\mathcal{G}}$  for every fixed $\chi>0$. Take some countable vanishing sequence $\chi_n>0$, and for each $n$ there could be at most countably many ergodic $\chi_n$-hyperbolic measures of maximal entropy. Every hyperbolic ergodic measure is $\chi_n$-hyperbolic for some $n\in\mathbb{N}$. The countable union of countable sets of ergodic hyperbolic measures of maximal entropy is countable as well.

\medskip
We state the following proposition which is verbatim as in \cite[\textsection 12]{Sarig}, for reference purposes.
\begin{prop}
For every $x\in \widehat{\pi}[\widehat{\Sigma}]$, $T_xM=E^s(x)\oplus E^u(x)$, where
$$(1)\text{ }\limsup\limits_{n\rightarrow\infty}\frac{1}{n}\log\|d_xf^n|_{E^s(x)}\|\leq -\frac{\chi}{2}\text{ , }(2)\text{ }\limsup\limits_{n\rightarrow\infty}\frac{1}{n}\log\|d_xf^{-n}|_{E^u(x)}\|\leq -\frac{\chi}{2}.
$$
The maps $\underline{R}\mapsto E^{s/u}(\widehat{\pi}(\underline{R}))$ are H\"older continuous as maps from $\widehat{\Sigma}$ to $TM$.
\end{prop}
\begin{proof}
The proof is the same as in \cite[proposition~12.6]{Sarig}, so we limit ourselves to a brief sketch; One first writes $x$ as the image of a chain $\underline{u}$ in $\Sigma$ (it is possible since the partition $\mathcal{R}$ is a refinement of the cover $\mathcal{Z}$). Then the tangent space $T_xM$ splits by $T_xV^s(\underline{u})\oplus T_xV^u(\underline{u})$. The rest is clear by lemma \ref{Lambda}(2) and proposition \ref{prop133}.
\end{proof}
\bibliographystyle{alpha}
\bibliography{SymbolicDynamicsSBO}

\end{document}